\documentclass[11pt, reqno]{amsart}
\usepackage{amsthm}
\usepackage{amsfonts}
\usepackage{xr}
\usepackage{graphicx}
\usepackage{amsmath}
\usepackage{epsfig,subfigure}
\usepackage{color}
\usepackage{amssymb}

\def\Z{\mathbb{Z}}
\def\R{\mathbb{R}}
\def\N{\mathbb{N}}

\def\epsilon{\varepsilon}
\def\hat{\widehat}
\def\tilde{\widetilde}

\newcommand{\me}{\mathrm{e}}

\newcommand{\SE}{\setcounter{equation}{0} \section}
\newcommand{\be}{\begin{equation}}
\newcommand{\ee}{\end{equation}}
\newcommand{\baa}{\begin{array}}
\newcommand{\eaa}{\end{array}}
\newcommand{\ba}{\begin{eqnarray}}
\newcommand{\ea}{\end{eqnarray}}
\numberwithin{equation}{section}

\newtheorem{theorem}{\bf Theorem}[section]
\newtheorem{lemma}[theorem]{\bf Lemma}
\newtheorem{prop}[theorem]{\bf Proposition}
\newtheorem{corollary}[theorem]{\bf Corollary}
\newtheorem{definition}[theorem]{\bf Definition}
\newtheorem{remark}[theorem]{\bf Remark}

\newtheorem{assume}[theorem]{\bf Assumption}

\newtheorem{claim}[theorem]{\bf Claim}

\begin{document}
\date{}
\title[Time-periodic reaction-diffusion equations]{Dynamics 
of time-periodic reaction-diffusion equations with front-like initial data on $\R$} 

\thanks{This work was supported by the Japan Society for the Promotion of Science (16H02151, 17F17021).}

\author[W. Ding and H. Matano]{Weiwei Ding$^\dag$ and Hiroshi Matano$^\ddag$}
\thanks{
$^\dag$ School of Mathematical Sciences, South China Normal University,
Guangzhou 510631, China;  Meiji Institute for Advanced Study of Mathematical Sciences, Meiji University, Tokyo 164-8525, Japan (e-mail: dingww@mail.ustc.edu.cn)}

\thanks{
$^\ddag$ Meiji Institute for Advanced Study of Mathematical Sciences, Meiji University, Tokyo 164-8525, Japan (e-mail: matano@meiji.ac.jp)}



\begin{abstract}
This paper is concerned with the Cauchy problem
\begin{equation*}
\left\{\baa{ll}
u_t=u_{xx} +f(t,u), & x\in\R,\,t>0, \vspace{3pt}\\
u(0,x)= u_0(x), & x\in\R,\eaa\right.
\end{equation*}
where $f$ is a rather general nonlinearity that is periodic in $t$, and satisfies $f(\cdot,0)\equiv 0$ and that the corresponding ODE has a positive periodic solution $p(t)$. Assuming that $u_0$ is front-like, that is, $u_0(x)$ is close to $p(0)$ for $x\approx -\infty$ and close to $0$ for $x\approx \infty$, we aim to determine the long-time dynamical behavior of the solution $u(t,x)$ by using the notion of propagation terrace introduced by Ducrot, Giletti and Matano (2014). We establish the existence and uniqueness of propagating terrace for a very large class of nonlinearities, and show the convergence of the solution $u(t,x)$ to the terrace as $t\to\infty$ under various conditions on $f$ or $u_0$.  We first consider the special case where $u_0$ is a Heaviside type function, and prove the converge result without requiring any non-degeneracy on $f$. Furthermore, if $u_0$ is more general such that it can be trapped between two Heaviside type functions, but not necessarily monotone, we show that the convergence result remains valid under a rather mild non-degeneracy assumption on $f$. Lastly, in the case where $f$ is a non-degenerate multistable nonlinearity, we show the global and exponential convergence for a much larger class of front-like initial data.  
\end{abstract}

\maketitle


\SE{Introduction and main results}
In this paper, we consider the following Cauchy problem
\begin{subequations}\label{E}
\begin{eqnarray}
&u_t=u_{xx}+f(t,u), \ \ \ &x\in\R,\,\,t>0,
\label{equation}\\[4pt]
&u(0,x)=u_0(x),\ \ \ &x\in\R,
\label{initial}
\end{eqnarray}
\end{subequations}
where the initial data $u_0\in L^{\infty}(\R)$ is piecewise continuous. The nonlinearity $f:\R\times [0,\infty)\to\R$
is locally H{\"o}lder continuous in $\R\times[0,\infty)$, and it is of class $C^1$ with respect to $u$. 
We assume that 
\begin{equation}\label{0state}
f(t,0)=0\,\,\hbox{ for all } t\in\R,
\end{equation}
and that $f$ is $T$-periodic in $t$ for some $T>0$, that is, 
\begin{equation}\label{tperiod}
f(t+T,u)=f(t,u) \,\,\hbox{ for all } t\in\R,\,u\geq 0.
\end{equation}
If the solution $u$ is spatially homogeneous, then $u=u(t)$ satisfies the following ODE, which will play an important role in the later argument:
\begin{equation}\label{ODE-1}
\frac{du}{dt} =f(t,u),\quad t>0.  
\end{equation}
We assume that \eqref{ODE-1} has a positive $T$-periodic solution $p(t)$. Namely, $p(t)$ is a function that satisfies 
\begin{equation}\label{ODE}
\left\{\baa{l}
\smallskip \displaystyle\frac{dp}{dt}=f(t,p)\, \,\hbox{ for } t\in\R,\vspace{3pt}\\
 p(t)\equiv p(t+T) \,\hbox{ for } t\in\R.\eaa\right.
\end{equation}
In the special case where $f$ is independent of $t$, that is, when \eqref{equation} is autonomous, the solution $p$ of \eqref{ODE} is nothing but a zero of $f=f(u)$. 

In the present work, we study the long-time behavior of solutions of \eqref{E} with ``front-like" initial data. Roughly speaking, $u_0$ is assumed to satisfy $0\leq u_0(\cdot)\leq p(0)$ and that it is close to $p(0)$ for $x\approx -\infty$ and close $0$ for $x\approx \infty$ (our actual hypotheses on $u_0$ will be formulated later). 
Our aim is to establish results that cover a large class of nonlinearities, including but not limited to the classical monostable, bistable, ignition nonlinearities.  

For $f$ belonging to one of the three classic types of nonlinearities mentioned above, it is well known that the asymptotic behavior of $u(t,x)$ can be described by periodic traveling waves.  By a {\bf periodic traveling wave} connecting $0$ and $p$, we mean an entire solution $U(t,x)$ of \eqref{equation} satisfying that, for some $c\in\R$,
\begin{equation}\label{define-tw}
U(t+T,x+cT)\equiv U(t,x)
\end{equation}
along with the asymptotics 
\begin{equation*}
\lim_{x\to\infty} U(t,x) =0,\quad  \lim_{x\to-\infty} U(t,x) =p(t) \,\,\,\hbox{ locally uniformly in } t\in\R. 
\end{equation*}
The real number $c$ is called the {\bf wave speed} of $U$. 
It is easily checked that $U(t,x)$ is a periodic traveling wave connecting $0$ and $p$ with wave speed $c$ if and only if it has the form $U(t,x)=\tilde{U}(t,x-ct)$, where $\tilde{U}$ is $T$-periodic in its first variable and satisfies $\tilde{U}(t,\infty)=0$ and $\tilde{U}(t,-\infty)=p(t)$ uniformly in $t\in\R$.

It is known that, in the bistable or combustion cases, there exists a unique (up to spatial shifts) periodic traveling wave connecting $0$ and $p$, and its speed is uniquely determined (see \cite{abc,con,shen-2,ss-1}),   
while in the monostable case, there exists a continuum of admissible speeds $[c_*,\infty)$ (see \cite{LYZ,shen-3}). It was also proved in \cite{abc,con,shen-1} that for bistable equations, any solution with front-like initial data converges to the periodic traveling wave.  For the convergence in the monostable or combustion cases, some additional assumption on the asymptotics of $u_0(x)$ as $x\to\infty$ is needed (see \cite{hr,shen-3,ss-2}).

For general $f$, the situation is more complicated, even for the autonomous equation 
\begin{equation}\label{autonomous}
u_t=u_{xx}+f(u)\,\,\hbox{ for } \, t>0,\,x\in\R.
\end{equation}
Indeed, if there are other stable zeros of $f$ between $0$ and $p$, a traveling wave of \eqref{autonomous} connecting $0$ and $p$ may not exist (see e.g., \cite{fm}). In such a case, the asymptotic behavior of \eqref{autonomous} cannot be represented by a single wave; it is represented by a combination of multiple waves, namely, a stacked family of traveling waves whose speeds may differ from one another. Such systems of traveling waves were first studied by Fife and McLeod  in \cite{fm,fm2} under the name ``minimal decomposition". Assuming that $f$ is a stacked composition of two bistable nonlinearities and that the zeros of $f$ are non-degenerate, they proved in \cite{fm} that, for all front-like initial data, solutions of \eqref{autonomous} converge to a minimal decomposition, which, in that context, means either a single traveling wave connecting $0$ to $p$ or a pair of traveling waves, one connecting $0$ to $q$ and the other connecting $q$ to $p$, where $q$ is a zero of $f$ satisfying $0<q<p$. In \cite{fm2}, more general nonlinearities were considered including multistable nonlinearities with an arbitrary number stable zeros between $0$ and $p$, and also the combustion type nonlinearities that have continuum of zeros, but the initial data were assumed to be monotone.  The convergence results of \cite{fm2} have been extended in \cite{v,vvv} to cover more general nonlinearities $f$ that are stacked combinations of bistable, monostable and combustion type nonlinearities. The results of \cite{fm} have been extended to cooperative reaction-diffusion systems on $\R$ with multistable nonlinearities \cite{rtv}. In those works, either the non-degeneracy condition on $f$ or the monotonicity restriction on $u_0$ is assumed. In the recent paper \cite{po3}, these conditions were all removed, and the convergence of solutions of \eqref{autonomous} to a family of traveling waves, called ``propagating terrace",  was proved for rather general initial data. In particular, if the steady states $0$ and $p$ are stable with respect to the ODE $h_t=f(h)$, then the convergence result holds for all front-like $u_0$.

The notion of propagating terrace was first introduced by Ducrot, Giletti and Matano in \cite{dgm}, which is concerned with a more general framework of spatially periodic equations, namely, $f=f(x,u)$ is periodic in $x$ (we will give the precise definition of propagating terrace later, directly in the framework of time-periodic equation \eqref{equation}). 
Under some stability assumption on the state $p$, it was shown in \cite{dgm} that any solution with Heaviside type initial data converges to a propagating terrace. In a follow-up paper \cite{gm}, further properties of propagating terrace were studied, and the convergence result was generalized. Apart from the aforementioned works for one-dimensional equations, it is known from some recent progress \cite{dm2,gr,po2} that the propagating terrace is also a fundamental concept in understanding the propagation dynamics of high-dimensional equations.

 In the present paper, we focus our attention on problem \eqref{E} with time-periodic nonlinearity $f$, and assume that the initial function $u_0$ is front-like. Our results reveal that the propagating terrace also plays a crucial role in determining the long-time behavior of solutions of \eqref{E} with general $f$.  We first show the existence and uniqueness of propagating terrace under some generic condition on $f$ and prove that any solution of \eqref{E} starting from Heaviside type function converges to this propagating terrace. These results do not require any non-degeneracy of $f$. The proof is similar to that given in \cite{dgm,gm} for spatially periodic equations.

The main part of our paper is devoted to the study of the convergence to propagating terrace for more general front-like initial data. Although this problem has been well addressed in the autonomous case \cite{po3}, the presence of time heterogeneity makes it significantly more difficult. Indeed, the proof given in \cite{po3} relies strongly on the method of phase plane analysis, while the usual ODE tools no longer work in our nonautonomous case. In our results, we develop the steepness arguments introduced in \cite{dgm,gm}, which allow us to handle the case where $u_0$ can be trapped between two Heaviside type initial functions, but not necessarily monotone (see (H2) below).  We first give a precise description of the asymptotic behavior of solutions of \eqref{E} with such initial data, and then prove the convergence to a propagating terrace under a mild non-degeneracy condition on $f$. Moreover, in the case where $f$ is a non-degenerate multistable nonlinearity, we show the global and exponential convergence for a much larger class of $u_0$ (see (H3) below). The proof of this result is based on a super and sub-solution method.


As announced above, the present work is concerned with the propagation dynamics of \eqref{E} with front-like initial data. We mention here that, for other types of initial data, such as nonnegative and compactly supported functions or more general functions with limits at $x\to\pm\infty$ equal to $0$, the asymptotic behavior of \eqref{E} has been extensively studied in the autonomous case (see e.g., \cite{dm,dm2,dp,mp,z1}) and  the nonautonomous case (see  e.g., \cite{dingm,fp,po1}).

\subsection{Propagating terrace: some definitions}  As mentioned above, the notion of propagating terrace was introduced in \cite{dgm} for spatially periodic equations (see also \cite{gm} for a slightly generalized version). For our time-periodic equation \eqref{equation}, the definition of propagating terrace can be presented as follows.

\begin{definition}\label{terrace}
A {\bf propagating terrace} connecting $0$ to $p$ is a pair of finite sequences $(p_i)_{0\leq i\leq N}$ and $(U_i,c_i)_{1\leq i\leq N}$ such that
\begin{itemize}
\item Each $p_i$ is a nonnegative solution of \eqref{ODE} satisfying 
$$p=p_0>p_1>\cdots>p_{N}=0; $$
\vskip 3pt
\item For each $1\leq i\leq N$, $U_i(t,x)$ is a periodic traveling wave solution of \eqref{equation} connecting $p_i$ to $p_{i-1}$ with wave speed $c_i\in\R$;
\vskip 3pt
\item The sequence $(c_i)_{1\leq i\leq N}$ satisfies $c_1\leq c_2\leq \cdots \leq c_N$. 
\end{itemize}
We denote such a propagating terrace by $\mathcal{T}:= ((p_i)_{0\leq i\leq N},(U_i,c_i)_{1\leq i\leq N})$ and call $(p_i)_{0\leq i\leq N}$ the {\bf platforms} of $\mathcal{T}$. 
\end{definition}

Hereinafter, by {\it a periodic traveling wave $U_i$ connecting $p_i$ to $p_{i-1}$ with wave speed $c_i$}, we always mean that $U_i$ is an entire solution of \eqref{equation} satisfying \eqref{define-tw} with $c=c_i$, along with the asymptotics
$$\lim_{x\to-\infty} U(x,t)=p_{i-1}(t),\quad \lim_{x\to\infty} U(x,t)=p_i(t) \quad \hbox{locally uniformly in } t\in\R.$$ 
Note that, in the above definition of propagating terrace, we do not assume any sign condition on the speed $c_i$, $1\leq i\leq N$. If $c_i<0$ (resp. $c_i>0$), then $U_i$ propagates to the left (resp. right); if $c_i=0$, then $U_i$ is a $T$-periodic solution of \eqref{equation}.

We will show below that, among all propagating terraces, only some particular terraces can be used to determine the propagation dynamics of \eqref{E}. To explain what they are, let us first introduce the following notion:  

\begin{definition}\label{steepness}
Let $v_1(x)$ and $v_2(x)$ be two piecewise continuous  functions defined on $x\in\R$. We say that $v_1$ is {\bf steeper than} $v_2$ if for any $x_1,\,x_2\in\R$ such that $v_1(x_1)=v_2(x_2)$, we have 
$$v_1(x+x_1)\geq v_2(x+x_2) \,\hbox{ for } x>0, \quad v_1(x+x_1)\leq v_2(x+x_2) \,\hbox{ for } x<0; $$
we say that $v_1$ is {\bf strictly steeper than} $v_2$ if the above two inequalities hold strictly.
Furthermore, for any two entire solutions $u_1(t,x)$ and $u_2(t,x)$ of \eqref{equation}, we say that $u_1$ is steeper (resp. strictly steeper) than $u_2$ if for each $t\in\R$, $u_1(t,\cdot)$ is steeper (resp. strictly steeper) than $u_2(t,\cdot)$.
\end{definition}

From the above definition, one easily sees that the concept of steepness is independent of the spatial positions of the two functions $v_1$, $v_2$. More precisely, if $v_1$ is steeper (resp. strictly steeper) than $v_2$ , then for any constants $a,\, b \in\R$, $v_1(\cdot+a)$ is steeper (resp. strictly steeper) than $v_2(\cdot+b)$. In other words, the steepness property is preserved by spatial translations. It is also easily seen that $v_1$ and $v_2$ are mutually steeper than each other if and only if either $v_1\equiv v_2$ up to a spatial translation or the ranges of $v_1$ and $v_2$ are disjoint. Moreover, if $v_1$ is strictly steeper than $v_2$, then for any $a,\, b \in\R$, the graph of $v_1(\cdot+a)$ and that of $v_2(\cdot+b)$ intersect at most once; the converse is also true. 

Additionally, it is easily seen from the strong maximum principle that for any two entire solutions $u_1$, $u_2$ of \eqref{equation}, if $u_1$ is steeper than $u_2$, then either $u_1(t,x)\equiv u_2(t,x+x_0)$ for some $x_0\in\R$ or $u_1$ is strictly steeper than $u_2$. We will also show in Lemma \ref{ini-steep} below that, if $u_1(0,x)$ is steeper than $u_2(0,x)$, then such steepness is preserved for any $t>0$. This property will be a key tool in showing the main results of the present paper (except Theorem \ref{converge3}).

\begin{remark}\label{rem-steep}
As mentioned above, the notion of steepness which we introduced in Definition \ref{steepness} is independent of the spatial positions of the solutions $u_1(t,x)$, $u_2(t,x)$ of \eqref{equation}.  Note that this definition is different from that given in \cite{dgm,gm} in which the notion of steepness is defined by using time shifts $u_1(t+t_1,x)$, $u_2(t+t_2,x)$ instead of spatial translations, thus this notion is independent of time shifts. The difference comes from the fact that the papers \cite{dgm,gm} deal with equations of the form $u_t=u_{xx}+f(x,u)$ that is spatially heterogeneous but time-homogeneous (autonomous), while our equation \eqref{equation} is time-heterogeneous but spatially homogeneous.     
\end{remark}

We are now ready to define a special class of propagating terraces. 

\begin{definition}\label{minimal-terrace}
A propagating terrace $\mathcal{T}= ((p_i)_{0\leq i\leq N},(U_i, c_i)_{1\leq i\leq N})$ is said to be {\bf minimal} if it satisfies the following:
\begin{itemize}

\item For any propagating terrace $\mathcal{T}'= ((p'_i)_{0\leq i\leq N'},(U'_i, c_i')_{1\leq i\leq N'})$ connecting $0$ to $p$, one has 
$$\{ p_i\, |\, 0\leq i\leq N \} \subset \{ p'_i \,|\, 0\leq i\leq N' \};  $$

\item For each $1\leq i \leq N$,  $U_i$ is steeper than any other periodic traveling wave of \eqref{equation} connecting $p_i$ to $p_{i-1}$. 
\end{itemize}
\end{definition}

\vskip 5pt 
 
Before stating our results, let us recall some basic notions on stability, which will be frequently used below. Let $\mathcal{X}_{per}$ denote the set of all nonnegative solutions of \eqref{ODE}. An element $q\in \mathcal{X}_{per}$ is said to be {\bf stable from above} (resp. {\bf below}) with respect to the initial-value problem 
\begin{equation}\label{ode-initial}
\frac{dh}{dt}=f(t,h)\,\,\hbox{ for } \, t>0,\quad  h(0)=h_0\in\R,  
\end{equation}
if it is stable under nonnegative (resp. nonpositive) perturbations of the initial values around $h_0=q(0)$. 
Otherwise $q$ is called {\bf unstable from above} (resp. {\bf below}). An element $q\in \mathcal{X}_{per}$ is said to be {\bf isolated from above} (resp. {\bf below}) if there exists no sequence of other solutions of \eqref{ODE} converging to $q$ from  above (resp. {\bf below}). Moreover, $q\in \mathcal{X}_{per}$ is said to be {\bf linearly stable} (resp. {\bf linearly unstable}) if 
$\int_{0}^{T} \partial_u f(t,q(t))dt < 0 \,\, ({\rm resp.} >0)$; it is said to be degenerate, if it is neither linearly stable nor linearly unstable.

\subsection{Existence and uniqueness of minimal propagating terrace}

We now proceed to the statements of our main results. We begin with the uniqueness of minimal propagating terrace. The uniqueness is meant here up to spatial shifts, that is, given two propagating terraces $((p_i)_{i},(U_i, c_i)_{i})$ and $((p'_i)_{i},(U'_i, c'_i)_{i})$, we say that they are equal up to spatial shifts if $p_i\equiv p'_i$ and $c_i=c'_i$ for every $i$, and $U_i(t,x+a_i)=U_i'(t,x)$ for some constants $a_i$, $i=1,\,\cdots, N$.  The uniqueness result actually follows immediately from the definition. 
For the convenience of later discussions, we state it precisely as follows: 

\begin{prop}\label{unique-mpt}
If there exists a propagating terrace $\mathcal{T}=((p_i)_{i},(U_i, c_i)_{i})$ that is minimal in the sense of Definition \ref{minimal-terrace}, then it is unique up to spatial shifts.
\end{prop}

\begin{proof}
According to Definition \ref{minimal-terrace}, all minimal propagating terraces should share the same platforms. It is also easily seen that, given two adjacent platforms, the steepest periodic traveling wave connecting them should be unique up to spatial shifts. 
\end{proof}

We now discuss the existence of minimal propagating terrace. Note that, in our definition of propagating terrace, only finitely many platforms can appear. Actually, if we only assume \eqref{0state}, \eqref{tperiod} and the existence of positive solution $p$ of \eqref{ODE}, there can exist a minimal propagating terrace with infinitely many platforms in some pathological cases (see \cite{po3} for the existence of such terraces in the autonomous case). To exclude the possibility of such pathological cases, we impose the following assumption, which is satisfied by virtually all the important examples of $f$. 

\begin{assume}\label{assume-decomposition}
There exists a {\bf decomposition} between $0$ and $p$, that is, 
there exists a finite sequence of solutions $(q_i)_{0\leq i\leq M}$ of \eqref{ODE} such that $q_0=p>q_1>\cdots>q_M=0$, and that for each $1\leq m\leq M$, there exists a periodic traveling wave $V_m$ connecting $q_{m}$ to $q_{m-1}$.
\end{assume}

The term decomposition was introduced by Fife and McLeod \cite{fm} for the case $f=f(u)$, and the above notion of decomposition can be viewed as a generalization of their notion. Note that a similar generalization has been given in \cite{gm} for spatially periodic equations. Unlike the definition of propagating terrace (see Definition \ref{terrace}), a decomposition does not require the speeds of the traveling waves to be ordered. Thus, a decomposition is a much weaker concept than a propagating terrace. However,  we will show in Theorem \ref{decomposition-existence} below that existence of a decomposition is enough to guarantee existence of a terrace. Before stating our result, let us first give some simple sufficient conditions for the existence of a decomposition. 

\begin{prop}\label{suff-decom}
Assume that either of the following conditions holds: 
\begin{itemize}
\item [(a)] There are finitely many solutions of \eqref{ODE} between $0$ and $p$;
\vskip 3pt
\item [(b)] $f=f(u)$ is independent of $t$, and the function $F(u): = \int_0^u f(s)ds$ has only finitely many global maximizers in $[0,p]$, all of which are isolated zeros of $f$ in $[0,p]$.
\end{itemize}
Then there exists a decomposition between $0$ and $p$.
\end{prop}

It is clear that solutions of \eqref{ODE} are all ordered. Thus, under condition (a), the solutions between $0$ and $p$ are numbered, say $q_0=p>q_1>q_2>\cdots> q_n=0$, and obviously they are isolated. Furthermore, equation \eqref{equation} restricted to the region between any adjacent $q_i$ and $q_{i+1}$ has a monostable structure. It then follows from the work \cite{LYZ} on time-periodic monostable semiflows that there exists a periodic traveling wave connecting $q_{i+1}$ to $q_i$.  This immediately implies that $(q_i)_{0\leq i\leq n}$ is a decomposition between $0$ and $p$.  Note that finiteness of the number of solutions of \eqref{ODE} between $0$ and $p$ is by no means necessary, since it is known that periodic traveling wave exists for a time-periodic combustion nonlinearity (see e.g., \cite{ss-1}). The condition (a) can be relaxed to include such nonlinearities or even a stacked composition of finitely many such nonlinearities.

Under condition (b), the existence of a decomposition follows from \cite[Theorem 1.2]{po3}, which actually shows the existence of a propagating terrace directly from (b). It is also known from \cite{po3} that, in the autonomous case,  the following condition implies (b): 
\begin{itemize}
\item [(b)'] There exists a solution of \eqref{autonomous} with compactly supported initial function that converges to $p$ from below as $t\to\infty$ locally uniformly on $\R$.
\end{itemize}
Indeed, (b)' holds if and only if $u=p$ is the unique global maximizer of the function $F$ in $[0,p]$ and it is an isolated zero of $f$. In the more general case where $f=f(x,u)$ is allowed to depend on $x$ periodically, under a similar condition to (b)', it was shown in \cite{dgm} that there exists a minimal terrace consisting of traveling waves with positive speeds. After the completion of the present work, we learned that similar existence result was proved for time-periodic equation \eqref{E} in \cite{ww}. Here, inspired by \cite{gm}, we show the existence of a minimal terrace under the more general Assumption \ref{assume-decomposition}.  Our theorem is stated as follows:


\begin{theorem}\label{decomposition-existence} 
Let Assumption \ref{assume-decomposition} holds. Then there exists a unique (up to spatial shifts) minimal propagating terrace $((p_i)_{0\leq i\leq N},(U_i,c_i)_{1\leq i\leq N})$ connecting $0$ to $p$. Moreover, it satisfies 
\begin{itemize}
\item[(i)] For any $1\leq i \leq N$, if $c_i>0$, then $p_{i-1}$ is isolated from below and stable from below; if $c_i<0$, then $p_{i}$ is isolated from above and stable from above; if $c_i=0$, then $p_{i-1}$ is stable from below and $p_i$ is stable from above.
\vskip 3pt
\item [(ii)] All $p_i$ and $U_i$ are steeper than any other entire solutions of \eqref{equation} between $0$ and $p$. 
\end{itemize}
\end{theorem}

Apart from the existence of a minimal propagating terrace, the above theorem also provides some information about what kind of solutions of \eqref{ODE} can possibly be selected as platforms of the terrace. For example, statement (i) implies in particular that for each $1\leq i \leq N-1$, $p_i$ is either stable from below or stable from above. This property will be used in handling a case of non-degenerate multistable nonlinearity (see Assumption \ref{multi-stable} below and its followed discussion). In addition, the steepness of $p_i$ in statement (ii) also implies that $p_i$ is contained in any decomposition between $0$ and $p$.

\subsection{Convergence to minimal propagating terrace}\label{converge-result}
In this subsection, we assume the existence of a minimal propagating terrace, and establish results on the convergence of solutions \eqref{E} to the minimal terrace.

In order to formulate our convergence theorems, let us introduce a notion of limit sets. Given a bounded solution $u(t,x)$ of \eqref{E}, we call $w(t,x)$ an {\bf $\Omega$-limit solution} of $u$ if 
\begin{equation}\label{def-omega}
u(t+k_j T,x+x_{j}) \to w(t,x) \,\,\hbox{ as } \,j\to\infty
\end{equation}
for some subsequence of positive integers $k_j\to\infty$ (as $j\to\infty$) and some sequence $(x_j)\subset \R$.
Here the convergence is understood in the topology of $L^\infty_{loc}(\R^2)$.  By parabolic estimates, 
this convergence also takes place in the $C^{1}(\R^2)$ topology. Clearly, $w(t,x)$ is an entire solution of \eqref{equation}.
Denote by $\Omega(u)$ the set of all $\Omega$-limit solutions. It is easily checked that if $w$ is an element of $\Omega(u)$, then so is $w(\cdot+kT,\cdot+z)$ for any $k\in\Z$ and $z\in\R$.
In Section 2.2 below, we will summarize more basic properties of $\Omega(u)$.

\begin{remark}\label{remark-limit-set}
Note that, if $x_j\equiv 0$ in \eqref{def-omega}, then $\Omega(u)$ coincides with the set of $\omega$-limit solutions defined in \cite{dingm}. The latter can capture the asymptotic behavior of $u(t,x)$ around each fixed point $x\in\R$ but cannot capture the profile of fronts that propagate at non-zero speeds. The above notion of $\Omega(u)$, on the other hand, can capture the profile of propagating fronts of any speeds. This multi-speed observation is particularly important for our study, since, as we will see later, multiple fronts with different speeds may coexist in a solution.    
\end{remark}

We are now ready to state our convergence theorems. Let us begin with a special case where the initial data are of the Heaviside type, that is, 

\vskip 5pt

\begin{itemize}
\item[{\bf (H1)}] {\it There is some $a\in\R$ such that 
\begin{equation*}
u_0(x)=p(0)H(a-x)
\end{equation*}
where $H$ denotes the Heaviside function defined by $H(x)=0$ if $x<0$ and $H(x)=1$ if $x\geq 0$. }
\end{itemize}

\begin{theorem}\label{converge1}
Assume that $((p_i)_{0\leq i\leq N},(U_i,c_i)_{1\leq i\leq N})$ is a minimal propagating terrace of \eqref{equation} connecting $0$ to $p$.   Let $u(t,x)$ be the solution of \eqref{E} with $u_0$ satisfying {\rm (H1)}. Then 
\begin{equation}\label{Omegaset}
\Omega(u)= \left\{ U_i(\cdot,\cdot+\xi): \xi\in\R,\, 1\leq i\leq N \right\} \cup  \left\{ p_i: 0\leq i\leq N \right\}. 
\end{equation}
Furthermore, there are $C^1([0,\infty))$ functions $\eta_1(t), \cdots ,\eta_N(t)$ such that the following statements hold:
\begin{itemize}
\item[(i)] $\eta_i(t) = o(t)$ as $t\to\infty$ for $i=1,\cdots,N$;
\vskip 3pt
\item[(ii)] $\eta_{i+1}(t)-\eta_{i}(t)\to\infty$ as $t\to\infty$ whenever $i\in \{1,\cdots,N-1\}$ satisfies $c_i=c_{i+1}$;
\vskip 3pt
\item[(iii)]  The following convergence holds:
\begin{equation}\label{converge-formula}
\lim_{t\to\infty}\sup_{x\in\R} \left | u(t,x)- \left(  \displaystyle\sum_{i=1}^{N} U_i\big(t,x-\eta_i(t)\big)- \sum_{i=1}^{N} p_i(t)    \right)\right | =0.   
\end{equation}
\end{itemize}
\end{theorem}

Clearly, \eqref{Omegaset} follows immediately from statement (iii). Note that, since $c_i$ is the wave speed of $U_i$, statements (i) and (iii) imply that the $i$-th front of $u$ has the asymptotic speed $c_i$. In the special case where $c_i=0$, the $i$-th front of $u$ moves with asymptotically vanishing speed. A	convergence result similar to the above theorem has been established in \cite{dgm,gm} for spatially periodic problem (i.e., $f=f(x,u)$ is periodic in $x$), but the speed $c_i$ was required to be non-zero.

We remark that the functions $(\eta_i(t))_{1\leq i\leq N}$ may not be convergent or even bounded.  
Actually, by statement (ii), if $c_i=c_{i+1}$ for some $i$, then at least either of $\eta_i(t)$ and $\eta_{i+1}(t)$ is unbounded. Besides, even in the simple case where the terrace consists of a single wave, it is known that, for autonomous KPP equations, the corresponding drift function $\eta(t)$ grows logarithmically as $t\to\infty$ (see \cite{br,hnrr}).  On the other hand, if $N\geq 2$ and if the speeds $(c_i)_{1\leq i\leq N}$ are all different, we will show in Theorem \ref{converge3} below that $(\eta_i(t))_{1\leq i\leq N}$ are convergent provided that $f$ is a non-degenerate multistable nonlinearity.

\vskip 5pt

Theorem \ref{converge1} is stated under very general assumptions on $f$ (only the standing hypotheses and the existence of a minimal propagating terrace), but the initial data are rather special. 
In the next two theorems, we relax the assumption on $u_0$ as follows:  
\vskip 5pt
\begin{itemize}
\item[{\bf (H2)}] {\it $u_0(x)$ is piecewise continuous, and there are two constants $a_-<a_+$ such that 
\begin{equation*}
p(0)H(a_--x) \leq u_0(x)\leq p(0)H(a_+-x) \,\,\hbox{ for } \,x\in\R,
\end{equation*}
where $H$ is the Heaviside function introduced in {\rm (H1)}.} 
\end{itemize}

\vskip 5pt

Note that any $u_0$ satisfying (H2) is not necessarily monotone. By the comparison principle, the solution $u(t,x)$ of \eqref{E} with such an initial function is bounded from above and below by two solutions starting from Heaviside type functions for all $t\geq 0$, each of which converges to the minimal terrace as shown in Theorem \ref{converge1}. In the case of the autonomous equation \eqref{autonomous}, it follows immediately from \cite[Theorem 1.2]{po3} that $u(t,x)$ converges to the minimal terrace. However, in the present time-periodic problem, the technique used in \cite{po3} does not apply, and it is much harder to prove the convergence. At the moment, all we can show without any extra condition on $f$ is the following:


\begin{theorem}\label{converge2}
Assume that $((p_i)_{0\leq i\leq N},(U_i,c_i)_{1\leq i\leq N})$ is a minimal propagating terrace of \eqref{equation} connecting $0$ to $p$.  Let $u(t,x)$ be the solution of \eqref{E} with $u_0$ satisfying {\rm (H2)}. Then 
$$\{p_i\}_{0\leq i\leq N} \subset \Omega(u).$$ 
Furthermore, for every $w\in \Omega(u)$, $w$ is either spatially constant or strictly decreasing in $x\in\R$, and  one of the following cases holds:
\begin{itemize}
\item[{\rm (a)}] $w(t,x)\equiv p_i(t)$ for some $0\leq i \leq N$; 
\vskip 3pt
\item[{\rm (b)}] $w(t,x)$ is a periodic traveling wave of \eqref{equation} connecting $p_{i}$ to $p_{i-1}$ with wave speed $c_i$ for some $1\leq i \leq N$; 
\vskip 3pt
\item[{\rm (c)}] 
There are two periodic traveling waves $V_{\pm}$ of \eqref{equation} connecting $p_{i}$ to $p_{i-1}$ and sharing the same wave speed $c_i$ for some $1\leq i \leq N$  such that
\begin{equation*}
w(t,x)- V_{\pm}(t,x) \to 0 \, \hbox{ as } t\to \pm \infty \, \hbox{ uniformly in } \,x\in\R,
\end{equation*}
and that $V_+$ is either strictly steeper or strictly less steep than $V_-$.
\end{itemize}
\end{theorem}

The above theorem immediately implies that $\{p_i\}_{0\leq i\leq N}$ are the only spatially homogeneous functions in $\Omega(u)$. 
Moreover, since the set $\{w(t,\cdot):\, t\in\R,\,w\in\Omega(u)\}$ is compact and connected in $L_{loc}^{\infty}(\R)$ (see Section 2.2 below for more details), we have the following corollary:

\begin{corollary}\label{coro-conn-tw}
Let the assumptions of Theorem \ref{converge2} hold. Then for each $1\leq i\leq N$, $\Omega(u)$ contains at least one periodic traveling wave connecting $p_{i}$ to $p_{i-1}$ with speed $c_i$. Furthermore, if $U_i$ is the unique (up to spatial shifts) such periodic traveling wave, then case {\rm (c)} of Theorem \ref{converge2} does not occur; hence,  \eqref{Omegaset} holds.  
\end{corollary}

As mentioned earlier, in the autonomous case, it is known from \cite[Theorem 1.2]{po3} that \eqref{Omegaset} holds for any solution of \eqref{autonomous} with {\rm (H2)}-type initial function. In fact, this result can also be derived directly from the above corollary, as a simple ODE argument can prove that, 
up to spatial shifts, $U_i$ is the unique traveling wave connecting $p_i$ to $p_{i-1}$ with speed $c_i$. Our next theorem shows that, the same is true for the time-periodic equation \eqref{E} under an additional assumption on $f$. 
Recall that $\mathcal{X}_{per}$ denotes the set of all nonnegative solutions of \eqref{ODE}. Our assumption is stated as follows:
  
\begin{assume}\label{non-degenerate}
Each element $q\in \mathcal{X}_{per}$ between $0$ and $p$ satisfies the following:  
\begin{itemize}
\item[(i)]  If $q>0$ is stable from below, then there exist a real number $\sigma_0>0$ and a $T$-periodic function $g(t)$ such that 
\begin{equation}\label{s-stable-be-1}
\int_{0}^{T}g(t)dt\leq 0,\quad\hbox{and}\quad   \partial_u f(t, u) \leq g(t)\, \,\hbox{ for all } \, u\in (q(t)-\sigma_0,q(t)],\,t\in\R;
\end{equation}
\item [(ii)] If $q<p$ is stable from above, then there exist a real number $\sigma_0>0$ and a $T$-periodic function $g(t)$ such that 
\begin{equation}\label{s-stable-be-2}
\int_{0}^{T}g(t)dt\leq 0,\quad\hbox{and}\quad   \partial_u f(t, u) \leq g(t)\, \,\hbox{ for all } \, u\in [q(t),q(t)+\sigma_0),\,t\in\R.
\end{equation}
\end{itemize} 
\end{assume}

By an easy comparison argument applied to \eqref{ode-initial}, one can check that a simple sufficient condition for our Assumption \ref{non-degenerate} to hold is that: 

\begin{assume}\label{non-degenerate-1}
Each element $q\in \mathcal{X}_{per}$ between $0$ and $p$ which is stable from above or from below is linearly stable.
 \end{assume}

It should be noted that Assumption \ref{non-degenerate} is weaker than Assumption \ref{non-degenerate-1}. 
 A simple example is that $f(t,u)=b(t)\bar{f}(u)$, where
$b(t)$ is a positive $T$-periodic $C(\R)$ function, $\bar{f}(u)$ is an autonomous combustion nonlinearity satisfying $\bar{f}(u) =0$ for $u\in [0,\theta]\cup\{p\}$,  $\bar{f}(u)>0$ for $u\in (\theta, p)$ and $\bar{f}'(u)\leq 0$ for $u$ close to $p$. Clearly, such a nonlinearity satisfies Assumption \ref{non-degenerate}, but not Assumption \ref{non-degenerate-1}. In fact, Assumption \ref{non-degenerate} allows infinitely many elements of $\mathcal{X}_{per}$ between $0$ and $p$, while there can only be finitely many such elements if Assumption \ref{non-degenerate-1} is satisfied. In the latter case, all the elements of $\mathcal{X}_{per}$ between $0$ and $p$ are isolated, since any   element of $\mathcal{X}_{per}$ that is unstable both from above and from below is isolated by definition, and any linearly stable element is also isolated as it is asymptotic stable.

Notice that the above two assumptions do not require anything on the elements of $\mathcal{X}_{per}$ that are unstable both from above and from below, therefore those elements can be degenerate.
 
 \vskip 5pt
 
Under Assumption \ref{non-degenerate}, we have the following uniqueness result.  
 
 \begin{prop}\label{unique1}
Let Assumption \ref{non-degenerate} hold and let $q_1$, $q_2$ be elements of $\mathcal{X}_{per}$ satisfying $0\leq q_1<q_2\leq p$. Let $V_1$ and $V_2$ be two periodic traveling waves connecting $q_1$ to $q_2$ with wave speeds $c_1$ and $c_2$, respectively. Assume that $V_1$ is steeper than $V_2$. Then there holds $c_1\leq c_2$. Furthermore, if $c_1=c_2$, then $V_1$ is equal to $V_2$ up to a spatial shift.   
\end{prop}

Combining Theorem \ref{converge2} and Proposition \ref{unique1}, we easily obtain that \eqref{Omegaset} holds for solutions of \eqref{E} with {\rm (H2)}-type initial data. Furthermore, we have the following  theorem. 
 
\begin{theorem}\label{converge2-2}
Let Assumption \ref{non-degenerate} hold. Assume that there exists a minimal propagating terrace $((p_i)_{0\leq i\leq N},(U_i,c_i)_{1\leq i\leq N})$ connecting $0$ to $p$. Then all the conclusions of Theorem \ref{converge1} hold for solutions of \eqref{E} with {\rm (H2)}-type initial data.  
\end{theorem}

The following result is an easy consequence of Proposition \ref{suff-decom} and Theorems \ref{decomposition-existence}, \ref{converge2-2}.

\begin{corollary}\label{existence-convergence} 
Let Assumption \ref{non-degenerate-1} hold. Then there exists a unique (up to spatial shifts) minimal propagating terrace $((p_i)_{0\leq i\leq N},(U_i,c_i)_{1\leq i\leq N})$ connecting $0$ to $p$. Furthermore, all the conclusions of Theorem \ref{converge1} hold for solutions of \eqref{E} with {\rm (H2)}-type initial data. 
\end{corollary}

\subsection{Convergence in a non-degenerate multistable case}

Our last main result is concerned with the asymptotic behavior of solutions of \eqref{E} where $f$ is of multistable type in the following sense: 

\begin{assume}\label{multi-stable}
The elements $0$ and $p$ of $\mathcal{X}_{per}$ are linearly stable, and any other element between $0$ and $p$ which is stable from above or from below is linearly stable.
 \end{assume}

Under the above assumption, it is clear that all the elements of $\mathcal{X}_{per}$ between $0$ and $p$ are isolated, and hence, there are only finitely many such elements. It then follows immediately from Proposition \ref{suff-decom} and Theorem \ref{decomposition-existence} that a minimal propagating terrace $((p_i)_{0\leq i\leq N}, (U_i, c_i)_{1\leq i\leq N})$ exists. Furthermore, in view of statement (i) of Theorem \ref{decomposition-existence}, for each $0\leq i\leq N$, $p_i$ is linearly stable. 
In other words, for each $1\leq i\leq N$, $U_i$ is a periodic traveling wave connecting two linearly stable solutions of \eqref{ODE}.

Clearly, Assumption \ref{multi-stable} is stronger than Assumption \ref{non-degenerate}, therefore Theorem \ref{converge2-2}
immediately implies that the minimal terrace attracts solutions with initial data satisfying (H2). But in this subsection, we will show that this holds for a larger class of initial data. To formulate our hypotheses on $u_0$, we denote $I_{+}$ and  $I_{-}$ by the intervals of attraction with respect to the equation \eqref{ode-initial} of the periodic solutions $p(t)$ and $0$, respectively. Namely, the set $I_+$ (resp. $I_-$) consists of element $h_0\in\R$ such that the solution of \eqref{ode-initial} with initial value $h_0$ converges to $p(t)$ (resp. $0$) as $t\to\infty$. 
Since we have assumed that $p$ and $0$ are linearly stable, it is easily checked that $I_+$ and $I_-$ are open intervals containing $p(0)$ and $0$, respectively. Our hypotheses on $u_0$ are stated as follows: 

\begin{itemize}
\item[{\bf (H3)}] {\it $u_0(x)$ is bounded and piecewise continuous, and it satisfies 
\begin{equation}\label{large-negative}
\liminf_{x\to- \infty} u_0(x)\in I_+,\quad \sup_{x\in\R} u_0(x)\in I_+,
\end{equation}
\begin{equation}\label{large-positive}
\limsup_{x\to\infty} u_0(x)\in I_-,\quad \inf_{x\in\R} u_0(x)\in I_-.
\end{equation}}
\end{itemize}

\begin{theorem}\label{converge3}
Let Assumption \ref{multi-stable} hold. Assume further that $\partial_u f(t,u)$ is locally Lipschitz continuous in $u$ uniformly for $t$.  Let $u(t,x)$ be the solution of \eqref{E} with $u_0$ satisfying {\rm (H3)}. 
Then there exists a unique (up to spatial shifts) minimal propagating terrace $((p_i)_{0\leq i\leq N},(U_i,c_i)_{1\leq i\leq N})$ connecting $0$ to $p$. Furthermore, the following statements hold:
\begin{itemize}
\item[(i)]
There are $C^1([0,\infty))$ functions $\eta_1(t), \cdots ,\eta_N(t)$
such that statements (i)-(iii) of Theorem \ref{converge1} hold for $u(t,x)$;
 \vskip 3pt
 \item[(ii)]
If the speeds satisfy $c_1<c_2<\cdots<c_N$, then for each $i=1,\,\cdots,\,N$, there is some constant $\bar{\eta}_i\in\R$ such that
\begin{equation*}
\lim_{t\to\infty}{\eta_i}(t)=\bar{\eta}_i,
\end{equation*}
 and there are constants $\nu>0$, $C>0$ such that
\begin{equation*}
\left | u(t,x)- \left(  \displaystyle\sum_{i=1}^{N} U_i\big(t,x-\bar{\eta}_i\big)- \sum_{i=1}^{N} p_i(t)    \right)\right | \leq C\me^{-\nu t} \,\,\hbox{ for all }\,  t\geq 0,\,x\in\R.
\end{equation*}
\end{itemize}
\end{theorem}

We remark that Theorem \ref{converge2-2} and Theorem \ref{converge3} (i) treat different cases, and the methods used to prove them are different. On the one hand, Theorem \ref{converge2-2} allows nonlinearities to be degenerate, but requires more restrictions on initial data. The proof relies heavily on zero-number arguments. On the other hand, a key step in the proof of Theorem \ref{converge3} (i) is to show that, up to some error terms with exponential decay, the solution $u(t,x)$ can be bounded from above and from below by solutions with Heaviside type initial data for all large times (see Lemma \ref{resemble} below). The non-degeneracy of $(p_i)_{1\leq i\leq N}$ plays an important role in this step.       

Theorem \ref{converge3} (ii) implies that, if the speeds $(c_i)_{1\leq i \leq N}$ are all different, then the functions $(\eta_i(t))_{1\leq i\leq N} $ are convergent, and the solution $u(t,x)$ converges to the minimal terrace with an exponential rate. In the autonomous case, similar convergence results have been proved in \cite{fm,po3} for scalar equation \eqref{autonomous}, and in \cite{rtv} for cooperative systems.  Our proof is based on ideas in these work, but new techniques are needed to overcome considerable difficulties arising from time heterogeneity.

In the special case $N=1$ (the terrace consists of a single wave), Theorem \ref{converge3} (ii) covers earlier stability results of periodic traveling waves for bistable equations \cite{abc,con}, and extends their results to the case where there may exist more than one intermediate solution of \eqref{ODE} between $0$ and $p$. It should be pointed out that, when $N\geq 2$, the assumption on the mutual distinctness of $(c_i)_{1\leq i \leq N}$ is necessary. Otherwise, if $c_i=c_{i+1}$ for some $i$, then by Theorem \ref{converge1} (i), $\eta_{i+1}(t)-\eta_i(t)\to \infty$ as $t\to\infty$, and hence,  at least one of $\eta_{i+1}(t)$ and $\eta_i(t)$ cannot be convergent.

\vskip 8pt

\noindent{\bf Outline of the paper.} In Section 2, we will present some preliminaries. We will recall some basic properties of zero-number arguments and the $\Omega$-limit set, and show a lemma on the stability of certain solutions of \eqref{ODE}. Section 3 is concerned with the asymptotic behavior of solutions of \eqref{E} with Heaviside type initial functions, and the proof of Theorems \ref{decomposition-existence} and \ref{converge1}. In Section 4, we will study the propagation dynamics of \eqref{E} with  $u_0$ satisfying (H2), and prove Theorems \ref{converge2} and  \ref{converge2-2}. Section 5 is devoted to the proof of Theorem \ref{converge3}. 


\SE{Preliminaries}

\subsection{Zero-number properties}
In this subsection, we recall some basic properties of zero-number arguments (also known as intersection comparison arguments). They will be key ingredients of our proofs in Sections 3-4. 

Let $\mathcal{Z}(w)$ denote the number of sign changes of a real-valued function $w(x)$ defined on $\R$, namely, the supremum over all $k\in\N$ such that there exist real numbers $x_1<x_2<\cdots<x_{k+1}$ with 
$$ w(x_i)\cdot w(x_{i+1})<0 \,\hbox{ for all } i=1,2,\ldots,k.$$ 
We set $\mathcal{Z}(w)=-1$ if $w\equiv 0$. Clearly, if $w$ is a smooth function having only simple zeros on $\R$, then $\mathcal{Z}(w)$ coincides with the number of zeros of $w$. 

The following intersection-comparison principle holds (see \cite{an,dm,dgm}). 

\begin{lemma}\label{zero1}
Let $w\not\equiv 0$ be a solution of the equation 
\begin{equation}\label{zeroeq}
w_t=w_{xx}+c(t,x)w \quad \hbox{for}\ \ t\in(t_1,t_2),\ x\in\R,
\end{equation} 
where the coefficient function $c$ is bounded. Then the following statements hold:
\begin{itemize}
\item[{\rm (i)}]
For each $t\in (t_1,t_2)$, all zeros of $w(t,\cdot)$ are isolated;
\vskip 3pt 
\item[{\rm (ii)}]
$t\mapsto\mathcal{Z}(w(t,\cdot))$ is a nonincreasing function with values in $\N\cup\{0\}\cup\{\infty\}$; 
\vskip 3pt 
\item[{\rm (iii)}]
If $w(t^*,x^*)=w_x(t^*,x^*)=0$ for some $t^*\in (t_1,t_2)$, $x^*\in\R$, then
$$\mathcal{Z}(w(t,\cdot))> \mathcal{Z}(w(s,\cdot)) \quad \hbox{for all}\ \ t\in(t_1,t^*),\ s\in(t^*,t_2) $$
whenever $\mathcal{Z}(w(s,\cdot))<\infty$.
\end{itemize}
\end{lemma}

One can check that $\mathcal{Z}$ is semi-continuous with respect to pointwise convergence, that is, the pointwise convergence $w_n(x)\to w(x)$ implies 
\begin{equation}\label{semi-continuous}
w\equiv 0 \quad \hbox{or} \quad \mathcal{Z}(w)\leq \liminf_{n\to\infty}\mathcal{Z}(w_n).
\end{equation}
This semi-continuity immediately implies the following lemma. 

\begin{lemma}\label{semicontinuous}
Let $(w_n)_{n\in\N}:\R\to\R$ be a sequence of functions converging to $w:\R\to\R$ pointwise on $\R$. If for each $n\in \N$, 
$w_n$ is steeper than $v:\R\to\R$, then $w$ is steeper than $v$.
\end{lemma}

The following lemma is a consequence of the application of Lemma \ref{zero1} when at most one intersection occurs and the fact the shape of the intersection remains the same as long as it exists. 

\begin{lemma}\label{ini-steep}
Let $u_1$ and $u_2$ be two bounded solutions of \eqref{equation}. Assume that $u_1(0,x)$ is piecewise continuous and bounded,  $u_2(0,x)$ is continuous and bounded, and that $u_1(0,x)$ is steeper than $u_2(0,x)$. Then for any $t>0$, the function $x\mapsto u_1(t,x)$ is steeper than $x\mapsto u_2(t,x)$. 
Furthermore, either $u_1\equiv u_2$ up to a spatial shift or for any $t>0$ and $z\in\R$, the function $x\mapsto u_1(t,x)-u_2(t,x+z)$ has at most one (simple) zero.  
\end{lemma}
 
 \begin{proof}
The proof follows from that of \cite[Lemma 2.4 and Corollary 2.5]{dgm} with obvious modifications, therefore we omit the details.  
 \end{proof}
 
\begin{remark}\label{steep-preserve}
We emphasize that, in the above two lemmas, the steepness of functions is independent of spatial translations. 
In view of this, Lemma \ref{ini-steep} implies that if $u_1(0,x)$ is steeper than $u_2(0,x)$, then for any $t>0$, the curves (not necessarily simple) $\{(u_1(t,x), \partial_x u_1(t,x)): x\in\R\}$  and $\{(u_2(t,x), \partial_x u_2(t,x)): x\in\R\}$ do not intersect unless they are equal. This property is indeed a key point in showing the convergence theorems in the autonomous case \cite{po3}, where the above curves are called spatial trajectories of solutions of \eqref{autonomous}. 
\end{remark} 
 
We also recall the following lemma which will be used repeatedly in proving Theorem \ref{converge2}.  The proof can be found in \cite{dm}. 

\begin{lemma}\label{zero3}
Let $w_n(t,x)$ be a sequence of functions converging to $w(t,x)$ in $C^1((t_1,t_2)\times I)$, where $I$ is an open finite interval in $\R$. Assume that for each $t\in(t_1,t_2)$ and $n\in\N$, the function $x\mapsto w_n(t,x)$ has only simple zeros in $I$, and that $w(t,x)$ satisfies an equation of the form \eqref{zeroeq} on $ (t_1,t_2)\times I$. Then for every $t\in(t_1,t_2)$, either $w\equiv 0$ on $I$ or $w(t,x)$ has only simple zeros on $I$. 
\end{lemma}

\subsection{Basic properties of the set of $\Omega$-limit solutions}  Recall that for any bounded solution $u(t,x)$ of \eqref{E}, $\Omega(u)$ denotes the set of $\Omega$-limit solutions defined in Section \ref{converge-result}. Obviously, our definition of $\Omega(u)$ is different from the standard notion of $\omega$-limit set. In this subsection, we summarize some basic properties of $\Omega(u)$.

First, since $u(t,\cdot)$ is uniformly bounded in $L^{\infty}(\R)$ for $t>0$, by the regularity assumption on $f$ and the standard parabolic estimates, $u(\cdot,\cdot)$ is bounded in $C^{1,2}([1,\infty)\times\R)$. This immediately implies that $\Omega(u)$ is a nonempty compact subset of $L^{\infty}_{loc}(\R^2)$. 

Next, we show that the following set 
\begin{equation*}
\bar{\Omega}(u):=\{w(t,\cdot):\, t\in\R,\,w\in\Omega(u)\}
\end{equation*}
coincides with the set of all limit points of the trajectory $\{u(t,\cdot):t>0\}$ with arbitrary spatial translations, that is, 
\begin{equation}\label{coincide-standard}
\bar{\Omega}(u)=\Omega^*(u),
\end{equation}
where 
$$\Omega^*(u):= \big\{\phi :\, u(t_j,\cdot+x_j) \to \phi(\cdot) \hbox{ for some sequences of real numbers }  t_j\to\infty \hbox{ and } x_j\in\R\}.$$ 
Here the convergence is with respect to the topology of $L_{loc}^{\infty}(\R)$. Indeed, the relation  
$\bar{\Omega}(u)\subset\Omega^*(u)$ is easily seen, so it suffices to prove $\bar{\Omega}(u)\supset\Omega^*(u)$.  Choose any $\phi\in\Omega^*(u)$ and let $t_j\to\infty$ and $x_j\in\R$ be such that  $u(t_j,\cdot+x_j) \to \phi(\cdot)$ in $L_{loc}^{\infty}(\R)$. For each $j\in\N$, let us write $t_j=t_j'+\tau_j$ with $t_j'\in T\N $ and $\tau_j \in [0,T)$. By choosing a subsequence if necessary, we can assume that the following limits exist:
\[
u(t+t_j',x+x_j)\to w(t,x)\ \ \hbox{in} \  L_{loc}^{\infty}(\R^2), \ \ \ \tau_j\to\tau_\infty\in [0,T].
\]
Clearly, $w$ belongs to $\Omega(u)$. Therefore, $\phi(\cdot)=w(\tau_{\infty}, \cdot) \in \bar{\Omega}(u)$, which proves \eqref{coincide-standard}. 

Finally, we note that  $\bar{\Omega}(u)$ is a nonempty, compact and connected subset of $L_{loc}^{\infty}(\R)$. Thanks to \eqref{coincide-standard}, this follows directly from the fact that $\Omega^*(u)$ is nonempty, compact and connected in $L_{loc}^{\infty}(\R)$ (the proof of such a property is standard in the theory of dynamical systems; at least the same argument that is known for autonomous equations applies to our time-periodic equations without any changes).

\subsection{Stability of solutions of \eqref{ODE} connected by traveling wave}  In this subsection, given a periodic traveling wave $U$ connecting two solutions $q_{\pm}$ of \eqref{ODE}, we investigate the link between the stability of $q_{\pm}$ and the sign of the speed of $U$.  The lemma stated below will be used frequently in later sections, and it is also of independent interest in its own. 

\begin{lemma}\label{speed-chara}
Let $q_-<q_+$ be two solutions of \eqref{ODE}.  Assume that there is a periodic traveling wave $U(t,x)$ of \eqref{equation} connecting $q_-$ to $q_+$ with speed $c_0\in\R$. Then the following statements hold:
\begin{itemize}
\item[(i)] If $c_0>0$, then $q_+$ is stable from below and isolated from below;
\vskip 3pt
\item[(ii)] If $c_0<0$, then $q_-$ is stable from above and isolated from above;
\vskip 3pt
\item[(iii)] If $c_0=0$, then $q_+$ is stable from below and $q_-$ is stable from above.
\end{itemize}
\end{lemma}

\begin{proof}
Let us first show that if $c_0\geq 0$, then $q_+$ is stable from below. Assume by contradiction that $q_+$ is unstable from below. It then follows directly from \cite[Proposition 3.5]{dingm} that there exists $R>0$ sufficiently large such that 
the following problem 
\begin{equation*}
\left\{\baa{ll}
\smallskip  \varphi_t=\varphi_{xx}+f(t,\varphi), & \hbox{ for } \,t\in\R,\,-R<x<R, \vspace{3pt}\\
\smallskip \varphi(t+T,x)=\varphi(t,x), & \hbox{ for }\,\,t\in\R,\,-R\leq x\leq R, \vspace{3pt}\\
\smallskip q_-(t) <\varphi(t,x)<q_+(t), & \hbox{ for }\,\,t\in\R,\,-R<x<R, \vspace{3pt}\\
\varphi(t,\pm R)=q_+(t), & \hbox{ for }\, t\in\R ,\eaa\right.
\end{equation*}
has a classical solution $\varphi(t,x)$ satisfying 
$$\partial_x \varphi (t,x)<0 \hbox{ for } t\in\R,\,x\in [-R,0) \quad\hbox{and}\quad  \partial_x \varphi (t,x)>0 \hbox{ for } t\in\R,\,x\in (0,R]. $$
Notice that the traveling wave $U(t,x)$ satisfies the following asymptotics 
\begin{equation}\label{U-q--q+}
\lim_{x\to\infty} U(t,x)=q_-(t), \quad \lim_{x\to-\infty} U(t,x)=q_+(t) \,\,\hbox{ locally uniformly in }\,t\in\R.
\end{equation}
It is easily checked from the above that, at $t=0$, there is some $x_0\in\R$ such that  $U(0,x+x_0)$ and $\varphi(0,x)$ intersects at some $\xi_0\in (-R,0]$, and that  
$$U(0,x+x_0)\leq \varphi(0, x)  \,\,\hbox{ for }\,  x\in [-R,R].$$
Clearly, $U(t,x+x_0) <  \varphi(t,x)$ for $t\geq 0$, $x=\pm R$. 
Then by the strong maximum principle, we have
\begin{equation}\label{U-sleq-var}
U(t,x+x_0) < \varphi(t,x) \,\,\hbox{ for } t>0,\,-R\leq x \leq R. 
\end{equation}

If $c_0=0$, then $U(t,x)$ is $T$-periodic in $t$. Since $\varphi(t,x)$ is also $T$-periodic, we have  
$U(T,x_0+\xi_0)=\varphi(T,\xi_0)$, which is a contradiction with \eqref{U-sleq-var}. In the case $c_0>0$, 
since 
\begin{equation}\label{travel-wave-U-k}
U(t+kT,x)=U(t,x-c_0kT) \,\,\hbox{ for each } \,k\in\Z,
\end{equation}
it follows from \eqref{U-q--q+} that 
$U(t+kT,x)$ converges to $p_+(t)$ as $k\to\infty$  locally uniformly in $t\in\R$ and $x\in\R$.   
In particular, we have $\lim_{k\to\infty} U(kT,x_0)=p_+(0)$. This also contradicts \eqref{U-sleq-var}, as $\varphi(kT,0)=\varphi(0,0)<p_+(0)$ for each $k\in\Z$. Thus, $q_+$ is stable from below if $c_0\geq 0$. 

In the case $c_0\leq 0$, one can proceed similarly as above to prove that $q_-$ is stable from above, and the details are omitted.   

Let us now show that $q_+$ is isolated from below in the case $c_0>0$.  Assume by contradiction that $q_+$ is an accumulation solution of \eqref{ODE} from below, that is, there exists some sequence $(q_j)_{j\in\N}$ of solutions of \eqref{ODE} such that $q_j\to q_+$ as $j\to\infty$ and $q_j<q_+$ for each $j\in\N$. It is then easily seen that 
$$\int_{0}^T \partial_u f(t,q_+(t)) dt=0, $$
and thus the following problem 
\begin{equation}\label{eigen}
\left\{\baa{l}
\smallskip  \phi_t- \partial_u f(t,q_+(t)) \phi =0\,\, \hbox{ for } \,t\in\R, \vspace{3pt}\\
\smallskip \phi(t+T)=\phi(t)\, \hbox{ for }\, t\in\R,  \quad \phi(0)=1,  \eaa\right.
\end{equation}
has a unique positive solution $\phi\in C^1(\R)$. 

To find a contradiction, we construct a super-solution of \eqref{equation} as follows. Let $c\in (0,c_0)$, $\lambda \in (0,c)$ be two constants and $v(t,x)$ be a function defined by 
$$v(t,x):= \min\left\{q_+(t), \me^{-\lambda(x-ct)}\phi(t)+q_j(t) \right\}  \, \hbox{ for } t\in\R,\,x\in\R,$$  
where $\phi$ is given by \eqref{eigen}, and $j\in\N$ is to be determined later. Let $t\mapsto y(t)$ be the function satisfying 
$$v(t,y(t))=q_+(t)\,\hbox{ for } t\in\R, \quad  v(t,x)<q_+(t) \hbox{ for }  x>y(t),\,t\in\R. $$ 
Clearly, $y(t)/t\to c$ as $t\to\infty$. It is also straightforward to compute on the set $D:= \{(t,x)\in [0,\infty) \times\R: x\geq y(t) \}$ that  
\begin{equation*}
\begin{split}
v_t-v_{xx}-f(t,v)\,&\,= \me^{-\lambda (x-ct)} (\phi_t+\lambda c\phi-\lambda^2 \phi)+f(t,q_j)-f(t,v) \vspace{3pt}\\
\medskip\,&\, = \me^{-\lambda (x-ct)} \left ( \lambda c\phi-\lambda^2 \phi  + \partial_u f(t,q_+(t))- \partial_u f(t, q_j+\theta \me^{-\lambda (x-ct)}\phi)\right),
\end{split}
\end{equation*}
for some $\theta=\theta(t,x) \in [0,1]$. Since $0<\lambda <c$, and since 
$$q_j(t)+\theta \me^{-\lambda (x-ct)}\phi(t) \to q_+(t) \hbox{ as } j\to\infty  \hbox{ uniformly in } D. $$
it then follows that for any $j$ sufficiently large 
$$v_t-v_{xx}-f(t,v)>0   \,\hbox{ in }  D. $$
Thus, $v$ is a super-solution of \eqref{equation} over $D$.  

Let $x_1>0$ be a sufficiently large number such that 
$$U(0,x+x_1)\leq v(0,x)\,\,\hbox{ for }\, x\geq y(0).$$
Then by the comparison principle, we have 
$$U(t,x+x_1)\leq v(t,x)\,\,\hbox{ for }\,  t\geq0,\,x\geq y(t).$$
This implies in particular that there exists some $\sigma>0$ such that
\begin{equation}\label{U-yt-v}
U(t,y(t)+1+x_1)\leq v(t,y(t)+1)\leq q_+(t)-\sigma \,\,\hbox{ for }\,  t\geq0.
\end{equation}
On the other hand, since $y(t)/t\to c$ as $t\to\infty$ and $c<c_0$,  it follows that $y(t)-c_0t\to -\infty$ as $t\to\infty$. 
Combining this with \eqref{U-q--q+} and \eqref{travel-wave-U-k}, we obtain 
$$ U(kT,y(kT)+1+x_1)= U(0,y(kT)-c_0kT+1+x_1)\to q_+(0)  \,\,\hbox{ as } \, k\to\infty.  $$
This is a contradiction with \eqref{U-yt-v}. Therefore,  $q_+$ is isolated from below if $c_0>0$. 

In the case $c_0<0$, one can argue analogously to conclude that $q_-(t)$ is isolated from above. The proof of Lemma \ref{speed-chara} is thus complete.
\end{proof}


\SE{Existence of minimal terrace and convergence with Heaviside type initial data}

In this section, we show the existence of minimal propagating terrace (i.e., Theorem \ref{decomposition-existence}), and the convergence to a minimal terrace when the initial data are of Heaviside type (i.e., Theorem \ref{converge1}). Recall that by a Heaviside type initial function, we mean a function $u_0$ has the form $u_0(x)=p(0)H(a-x)$ for some $a\in\R$.  Hereinafter, we denote by $\hat{u}(t,x)$ the solution of \eqref{E} with such an initial function. In some parts of later sections, we will write $\hat{u}(t,x;a)$ instead of $\hat{u}(t,x)$ to stress the dependence on $a$.

As mentioned earlier, the proof is inspired from the papers \cite{dgm,gm} devoted to spatially periodic equations, but the method has to be adapted here to the time-periodic framework. Let us first state a key lemma.

\begin{lemma}\label{steepest}
Any $\Omega$-limit solution of the solution $\hat{u}(t,x)$ is steeper than any other entire solution of \eqref{equation} lying between $0$ and $p$. 
\end{lemma}

\begin{proof}
Let  $w(t,x)$ be an $\Omega$-limit solution of $\hat{u}(t,x)$. Then there exist a sequence of positive integers $(k_j)_{j\in\N}$ ($k_j\to\infty$ as $j\to\infty$) and a sequence of real numbers  $(x_j)_{j\in\N}$ such that
$$ \hat{u}(t+k_jT,x+x_j) \to w(t,x) \hbox{ as } j\to\infty  \hbox{ in } C^1(\R^2). $$

For any entire solution $v(t,x)$ of \eqref{equation} between $0$ and $p$, it is easily seen that for each $j\in\N$, the function $\hat{u}(0,x+x_j)$ is steeper than $v(-k_jT,x)$ in the sense of Definition \ref{steepness}. By using Lemma \ref{ini-steep}, for any $t>-k_jT$, $\hat{u}(t+k_jT,x+x_j)$ is steeper than $v(t,x)$.  Then by Lemma \ref{semicontinuous}, passing to the limit as $j\to\infty$, we obtain that $w(t,x)$ is steeper than $v(t,x)$, where $t\in\R$ is arbitrary. This ends the proof of Lemma \ref{steepest}.
\end{proof}

In Subsection 3.1, we will use the above lemma to show the convergence of $\hat{u}(t,x)$ to a unique limit function around a given level set, and further prove that the limit function is either a solution of \eqref{ODE} or a periodic traveling wave connecting two solutions of \eqref{ODE}.  Once we obtain this convergence property, we will be able to construct a minimal propagating terrace by an iterative argument, and complete the proof of Theorem \ref{decomposition-existence} (see Subsection 3.2). The assumption that there exists a decomposition between $0$ and $p$ will be used to ensure that the iteration process ends in a finite number of steps. 
In Subsection 3.3, we will give the proof of Theorem \ref{converge1}, which also relies on the convergence property established in Subsection 3.1.

\subsection{Convergence around a given level set}
Let $\hat{u}(t,x)$ be the solution of \eqref{E} with a given Heaviside type initial function. It is clear that $0\leq \hat{u}(t,x)\leq p(t)$ for $t\geq 0$, $x\in\R$, and for each $t>0$, $\hat{u}(t,x)$ is decreasing in $x\in\R$, and satisfies 
$$\lim_{x\to \infty}\hat{u}(t,x) = 0\quad\hbox{and}\quad \lim_{x\to -\infty}\hat{u}(t,x) = p(t). $$
This implies in particular that for each $k\in\N$, there exists a unique $a_k\in\R$ such that 
\begin{equation}\label{choose-ak}
\hat{u}(kT,a_k)=\alpha, 
\end{equation}
where $\alpha\in (0,p(0))$ is a given constant. The following lemma gives the local convergence of $\hat{u}(t+kT,x+a_k)$ around the level $\alpha$ as $k\to\infty$.

\begin{lemma}\label{convege-ak}
For any $\alpha\in (0,p(0))$, let $(a_k)_{k\in\N}$ be the sequence provided by \eqref{choose-ak}. 
Then the following limit exists for the topology of $C^1(\R^2)$:
\begin{equation}\label{u-winfty}
\lim_{k\to\infty} \hat{u}(t+kT,x+a_k):=w_{\infty}(t,x;\alpha).
\end{equation}
The function $w_{\infty}(t,x;\alpha)$ is a positive entire solution of \eqref{equation} which is steeper than any other entire solution between $0$ and $p$. Furthermore, it is either spatially homogeneous or spatially decreasing. 
\end{lemma}

\begin{proof}
By parabolic estimates, the sequence $\{ \hat{u}(t+kT,x+a_k)\}_{k\in\N}$ is uniformly bounded along with their derivatives. Thus, it is relatively compact for the topology of $C^1(\R^2)$. Then there exists a subsequence $(k_j)_{j\in\N}$ of integers such that $k_j\to\infty$ as $j\to\infty$ and that 
$$ \hat{u}(t+k_jT,x+a_{k_j}) \to w_{\infty}(t,x;\alpha) \hbox{ as } j\to\infty  \hbox{ in } C^1(\R^2),$$
where $w_{\infty}(t,x;\alpha)$ is an entire solution of \eqref{equation}. 
Clearly, $w_{\infty}(t,x;\alpha)$ is an $\Omega$-limit solution of $\hat{u}(t,x)$ and $w_{\infty}(0,0;\alpha)=\alpha$.  
It is also easily seen from the strong maximum principle that $0<w_{\infty}(t,x;\alpha)<p(t)$ for $t\in\R$, $x\in\R$.
Furthermore, by Lemma \ref{steepest},  $w_{\infty}(t,x;\alpha)$ is steeper than any other entire solution of \eqref{equation} lying between $0$ and $p$.

It is easily checked from Definition \ref{steepness} that, $w_{\infty}(t,x;\alpha)$ is the unique solution of \eqref{equation} 
that is steeper than any other entire solution between $0$ and $p$ and satisfies $w_{\infty}(0,0;\alpha)=\alpha$. This implies that $w_{\infty}$ does not depend on the choice of $(k_j)$, and hence, the whole sequence $\hat{u}(t+kT,x+a_k)$ converges to $w_{\infty}(t,x;\alpha)$ as $k\to\infty$ in $C^1(\R^2)$.

It remains to show that $w_{\infty}$ is either spatially homogeneous or decreasing in $x\in\R$.  Since for each $k\in\N$, the function $x\mapsto\hat{u}(t+kT,x+a_k)$ is decreasing, sending to the limit as $k\to\infty$, we have 
$$\partial_x w_{\infty}(t,x;\alpha)\leq 0 \,\hbox{ for } t\in\R,\,x\in\R. $$
Then applying the strong maximum principle to the equation satisfied by $\partial_x w_{\infty}(t,x;\alpha)$, we immediately obtain that either $\partial_x w_{\infty}\equiv 0 $ or $\partial_x w_{\infty}(t,x;\alpha)<0$ for $t\in\R$, $x\in\R$. This completes the proof. 
\end{proof}

In Lemma \ref{homo-ptw} below, we will further prove that the limit function $w_{\infty}(t,x;\alpha)$ is either a  solution of  \eqref{ODE} or a periodic traveling wave. The proof will need the following spreading properties of $\hat{u}(t,x)$: 

\begin{lemma}\label{spreading-speed}
There exist constants $-\infty<c_*\leq c^*<\infty$ such that 
\begin{itemize}
\item[(i)] for each $c>c^*$,  $\lim_{t\to\infty} \sup _{x\geq ct} \hat{u}(t,x)=0$;
\vskip 3pt
\item[(ii)] for each $c<c_*$, $\lim_{t\to\infty} \sup _{x\leq ct} | \hat{u}(t,x)-p(t)|=0$.
\end{itemize}
\end{lemma}
  
It should be noted that, the above two constants $c_*$ and $c^*$ are independent of $a\in\R$ (the jumping position of the initial function $p(0)H(a-x)$). 

\begin{proof}
The proof follows from the arguments used in the first part of the proof of \cite[Lemma 2.9]{dgm} with some obvious modifications, therefore we omit the details. 
\end{proof}

Let us now define a sequence of real number $(l_k)_{k\in\N}$ as follows:
\begin{equation*}
l_k:=\left\{\baa{ll}
\smallskip a_k-a_{k-1}, & \hbox{ if } k>1,\vspace{3pt}\\
 a_0,&\hbox{ if } k=0,\eaa\right.
\end{equation*}
where $(a_k)_{k\in\N}$ is the sequence given in \eqref{choose-ak}. Clearly, for all $k\in\N$, $a_k=\sum_{j=0}^k l_j$. 
Since $\hat{u}(kT,a_k)=\alpha \in (0, p(0))$ for each $k\in\N$, it follows from Lemma \ref{spreading-speed} that
$$c_*kT\leq a_k\leq c^*kT\,\,\hbox{ for all large } k\in\N.$$
Thus, we have
\begin{equation}\label{lk-bound}
c_*\leq  \liminf_{k\to\infty}  \frac{ \sum_{j=0}^{k}l_j}{kT} \leq \limsup_{k\to\infty}  \frac{ \sum_{j=0}^{k}l_j}{kT} \leq c^*.
\end{equation}

\begin{lemma}\label{homo-ptw}
For any given $\alpha\in (0,p(0))$, let $w_{\infty}(t,x;\alpha)$ be the entire solution provided by Lemma \ref{convege-ak}. Then either of the following alternatives holds:
\begin{itemize}
\item[(a)] $w_{\infty}(t,x;\alpha)$ is spatially homogeneous, and it is a positive solution of \eqref{ODE};
\vskip 3pt  
\item[(b)] $w_{\infty}(t,x;\alpha)$ is decreasing in $x\in\R$, and it is a periodic traveling wave of \eqref{equation}. Furthermore, $l_k$  converges to some  $l_{\infty}\in [c_*T,c^*T]$ as $k\to\infty$, and $l_{\infty}/T$ is the wave speed of $w_{\infty}$. 
\end{itemize}  
\end{lemma}

\begin{proof}
We split the proof into two parts, according to whether there exists some subsequence of $(l_k)_{k\in\N}$ converging to a finite number. 

{\bf Case (1):} There is a subsequence $(k_j)_{j\in\N}$ such that $k_j\to\infty$ as $j\to\infty$ and that  
$l_{k_j}\to l_{\infty}$ as $j\to\infty$ for some $l_{\infty}\in\R$.  

In this case, by using \eqref{u-winfty},  we have 
\begin{equation}\label{period-proof}
\begin{split}
w_{\infty}(t,x-l_{\infty};\alpha)\,&\,= \lim_{j\to\infty} \hat{u}(t+k_jT,x+a_{k_j}-l_{k_j}) \vspace{3pt}\\
\medskip\,&\, = \lim_{j\to\infty} \hat{u}(t+k_jT,x+a_{k_j-1})= w_{\infty}(t+T,x;\alpha).
 \end{split}
\end{equation}
Notice from Lemma \ref{convege-ak} that $w_{\infty}$ is either spatially homogeneous or spatially decreasing. Therefore, $w_{\infty}$ is a positive solution of \eqref{ODE} if $\partial_x w_{\infty}(t,x) \equiv 0$; it is a periodic traveling wave connecting two distinct solutions of \eqref{ODE} if $\partial_x w_{\infty}(t,x)< 0$.

Clearly, if the latter case occurs, then $l_{\infty}/T$ is the wave speed of the wave $w_{\infty}$. We now show that, in such a situation, the whole sequence $l_k$ converges to $l_{\infty}$ as $k\to\infty$. 
Let $\epsilon_0>0$ be a sufficiently small constant. Thanks to the $C^1$ convergence of $\hat{u}(kT,x+a_{k-1})$ 
to $w_{\infty}(T,x;\alpha)$ as $k\to\infty$ and the fact that  $w_{\infty}(T,x;\alpha)<0$ for $x\in\R$, there exists some large $K\in\N$ such that 
$$\partial_x   \hat{u}(kT,l_{\infty}+\epsilon+a_{k-1}) \leq \frac{1}{2} \partial_x  w_{\infty}(T,l_{\infty};\alpha)\, \hbox{ for all }\,\epsilon \in [0,\epsilon_0),\,k\geq K. $$
This  implies  in particular that
\begin{equation}\label{hatu-epsilon0}
\hat{u}(kT,l_{\infty}+\epsilon_0+a_{k-1}) \leq  \hat{u}(kT,l_{\infty}+a_{k-1})+ \frac{\epsilon_0}{2} \partial_x  w_{\infty}(T,l_{\infty};\alpha)\, \hbox{ for all } \,k\geq K. 
\end{equation}
Let $\delta$ be a positive constant given by 
$$\delta:= -\frac{\epsilon_0}{4} \partial_x  w_{\infty}(T,l_{\infty};\alpha).  $$
Then, replacing $K$ by some larger integer if necessary, it follows from \eqref{u-winfty} that
$$\hat{u}(kT,l_{\infty}+a_{k-1}) \leq  w_{\infty}(T,l_{\infty};\alpha)+\delta\, \hbox{ for all }\,k\geq K. $$
Combining this with  \eqref{hatu-epsilon0}, we obtain 
\begin{equation}\label{hatu-k-leq-wT}
\hat{u}(kT,l_{\infty}+\epsilon_0+a_{k-1}) \leq  w_{\infty}(T,l_{\infty};\alpha) \, \hbox{ for all }\,k\geq K. 
\end{equation}
Due to \eqref{period-proof}, we have $w_{\infty}(T,l_{\infty};\alpha)=w_{\infty}(0,0;\alpha)=\alpha$. It follows that
$$\hat{u}(kT,l_{\infty}+\epsilon_0+a_{k-1})  \leq \alpha\,\hbox{ for all }\, k\geq K.$$  
Note that for each $k\in\N$, $\hat{u}(kT,a_{k})=\alpha$ and $\hat{u}(kT,x)$ is nonincreasing in $x\in\R$. 
Then we have 
$$a_{k-1}+l_{\infty}+\epsilon_0 \geq a_k\,  \hbox{ for all }\,k\geq K.  $$

By similar arguments to those used in showing \eqref{hatu-k-leq-wT}, we can prove that
$$\hat{u}((k-1)T,-l_{\infty}+\epsilon_0+a_{k})\leq w_{\infty}(-T,-l_{\infty};\alpha)=\alpha \, \hbox{ for all large } \,k,$$ 
and then conclude that 
$$a_{k}-l_{\infty}+\epsilon_0 \geq a_{k-1}  \, \hbox{ for all large } \,k. $$

Combining the above, we have $|a_k-a_{k-1} -l_{\infty}| \leq \epsilon_0$ for all large $k\in\N$. Since $\epsilon_0>0$ can be chosen arbitrarily small, this immediately gives that $ l_k\to l_\infty \,\hbox{ as } k\to\infty$. Furthermore, thanks to \eqref{lk-bound}, the limit $l_{\infty}$ belongs to the interval $[c_*T,c^*T]$. 

Therefore, we can conclude that if case (1) occurs, then either case (a) or case (b) of the present lemma holds.

{\bf Case (2):}  there is no subsequence of $(l_k)_{k\in\N}$ converging to a finite number.

In this situation, we want to show that only case (a) occurs.  As no subsequence of $(l_k)_{k\in\N}$ converges to a finite number, we observe from \eqref{lk-bound} that, there must exist a subsequence $l_{k_j}$ converging to $\infty$ and another one $l_{k_j'}$ converging to $-\infty$ as $j\to\infty$. Let $M>0$ be arbitrary. Then 
we have $l_{k_j}>M$  and $l_{k_j'}<-M$ for all large $j\in\N$. Since $\hat{u}(t,x)$ is nonincreasing in $x$, it follows that
$$\hat{u}(t+k_jT,x+a_{k_j-1}) \geq    \hat{u}(t+k_jT,x+a_{k_j}-M)$$
for all $t\in\R$, $x\in\R$ and large $j\in\N$.
Taking the limit along the sequence $j\to\infty$, we obtain
$$w_{\infty}(t+T,x;\alpha)  \geq    w_{\infty}(t,x-M;\alpha) \,\hbox{ for }  \,t\in\R,\,x\in\R. $$
Then, passing to the limit as $M\to\infty$ followed by letting $x\to\infty$ yields  
\begin{equation}\label{wTinfty-geq}
\lim_{x\to\infty} w_{\infty}(t+T,x;\alpha)\geq  \lim_{x\to-\infty} w_{\infty}(t,x;\alpha)\,\hbox{ for }\,t\in\R.
\end{equation}

Similarly, since 
$$\hat{u}(t+k_j'T,x+a_{k_j'-1})\leq    \hat{u}(t+k_j'T,x+M+a_{k_j'})$$
for all $t\in\R$, $x\in\R$ and large $j\in\N$, passing to the limits as $j\to\infty$, $M\to\infty$ and $x\to -\infty$ in order, we have 
\begin{equation}\label{wTinfty-leq}
\lim_{x\to-\infty} w_{\infty}(t+T,x;\alpha)\leq  \lim_{x\to+\infty} w_{\infty}(t,x;\alpha)\,\hbox{ for }\,t\in\R. 
\end{equation}

Combining \eqref{wTinfty-geq},  \eqref{wTinfty-leq} and the fact that $w_{\infty}(t,x;\alpha)$ is nonincreasing in $x$, we see that $w_{\infty}$ is spatially homogeneous, and hence it is a solution of \eqref{ODE}. The proof of Lemma \ref{homo-ptw} is thus complete. 
\end{proof}

\subsection{Existence of a minimal terrace}
Based on the preparation in the above subsection, we are now ready to construct a minimal propagating terrace under Assumption \ref{assume-decomposition}.

\begin{proof}[Proof of Theorem \ref{decomposition-existence}]
Denote by $(q_m)_{0\leq m\leq M }$ the decomposition between $0$ and $p$, and for each $1\leq m \leq M$, let $V_m(t,x)$ be a periodic traveling wave of \eqref{equation} connecting $q_{m}$ to $q_{m-1}$.  
Choose $\alpha_1>0$ such that $q_1(0)<\alpha_1<p(0)$ and let $w_{\infty}(t,x;\alpha_1)$ be the entire solution provided by Lemma \ref{convege-ak}. 
\begin{claim}\label{step1-ite}
$U_1(t,x):=w_{\infty}(t,x;\alpha_1)$
is a periodic traveling wave of \eqref{equation} connecting $q_{m_1}$ to $p$ for some $m_1\in\{1,\,2,\,\cdots, M\}$. 
\end{claim}

\begin{proof}[Proof of Claim \ref{step1-ite}]
We first prove that $U_1(t,x)$ is a periodic traveling wave. Suppose the contrary that it is not true. Then by Lemma \ref{homo-ptw}, $w_{\infty}(t,x;\alpha_1)$ is spatially homogeneous, and $w_{\infty}(t):=w_{\infty}(t,x;\alpha_1)$ is a solution of \eqref{ODE} with $w_{\infty}(0)=\alpha_1$. 
Notice that $q_1<w_{\infty}<p$, and $V_1(t,x)$ is a periodic traveling wave connecting $q_1$ to $p$. 
Clearly, $V_1(t,x)$ is steeper than $w_{\infty}(t)$ in the sense of Definition \ref{steepness}. This is a contradiction with the assertion stated in Lemma \ref{convege-ak} that $w_{\infty}$ is steeper than any other entire solution between $0$ and $p$.  Therefore, $U_1(t,x)$ is a periodic traveling wave. 

Next, we prove that 
$$\lim_{x\to-\infty} U_1(t,x) = p(t) \,\hbox{ locally uniformly in } \, t\in\R.$$ 
Since $w_{\infty}(t,x;\alpha_1)$ is a periodic traveling wave, the limit $\lim_{x\to-\infty} w_{\infty}(t,x;\alpha_1):=w_{\infty}(t,-\infty)$ holds locally uniformly in $t\in\R$, and $w_{\infty}(t,-\infty)$ is a solution of \eqref{ODE}.  It is clear that
\begin{equation}\label{0-winfty-p} 
0\leq w_{\infty}(t,-\infty)\leq p(t) \hbox{ for } t\in\R.
\end{equation}
As $w_{\infty}(t,x;\alpha_1)$ is a steepest solution of \eqref{equation} between $0$ and $p$, in view of Definition \ref{steepness} and Lemma \ref{semicontinuous}, we see that $w_{\infty}(t,-\infty)$ is steeper than any other entire solution between $0$ and $p$. This implies  in particular that for any $t\in\R$, $w_{\infty}(t,-\infty)$ and $V_1(t,x)$ cannot intersect. Since $w_{\infty}(0,-\infty)>\alpha_1$ and  $V_1(0,x_0)=\alpha_1$ for some $x_0\in\R$, it then follows that $w_{\infty}(t,-\infty)>V_1(0,x)$ for all $x\in\R$. By the comparison principle, we have 
$$w_{\infty}(t,-\infty)\geq V_1(t,x)\,\,\hbox{ for } \,t\geq 0,\,x\in\R.$$
Due to the $T$-periodicity of $w_{\infty}(t,-\infty)$ and the fact the $V_1(t,x)$ converges to $p(t)$ as $x\to-\infty$, we have 
$w_{\infty}(t,-\infty)\geq p(t)$ for $t\in\R$. Combining this with \eqref{0-winfty-p}, we immediately obtain $w_{\infty}(t,-\infty)\equiv p(t)$.

It remains to show 
$$\lim_{x\to\infty} U_1(t,x) = q_{m_1}(t)  \,\hbox{ locally uniformly in } \, t\in\R $$
for some $m_1\in\{1,\,2,\,\cdots, M\}$. In fact, by similar arguments used as above, one can conclude that 
$w_{\infty}(t,\infty):=\lim_{x\to\infty} w_{\infty}(t,x;\alpha_1)$ 
is steeper than any other entire solution of \eqref{equation} between $0$ and $p$. This already implies the desired result. Otherwise, there would exist some $\tilde{m}\in \{2,\,\cdots, M\}$ such that the periodic traveling wave $V_{\tilde{m}}(t,x)$ crosses through $w_{\infty}(t,\infty)$, which is impossible. The proof of Claim \ref{step1-ite} is thus complete.
\end{proof}

For convenience, let us set $m_0=0$. Then $U_1(t,x)$ is a traveling wave connecting some $q_{m_1}$ to $q_{m_0}$. 
If $m_1=M$, i.e., $q_{m_1}\equiv 0$, then we have already obtained a propagating terrace, which consists of a single wave $U_1$. Suppose, on the other hand, that $q_{m_1}$ is positive. We can continue our iteration as follows:

\begin{claim}\label{continue-ite}
Suppose that for some $m_i\in \{1,\,\cdots,M-1\}$ and some $\alpha_i \in (q_{m_i}(0), p(0))$, the function 
$$U_i(t,x):\equiv w_{\infty}(t,x;\alpha_i)$$
is a periodic traveling wave connecting $q_{m_i}$ to some $q_{m_{i-1}}> q_{m_i}$. Then there exists  $\alpha_{i+1}\in (0, q_{m_i}(0))$ such that 
$$ U_{i+1}(t,x) := w_{\infty}(t, x;\alpha_{i+1})$$
is a periodic traveling wave connecting some $q_{m_{i+1}}<q_{m_i}$ to $q_{m_i}$.
\end{claim}

\begin{proof}[Proof of Claim \ref{continue-ite}]
The proof is almost the same as that of Claim \ref{step1-ite}, therefore we omit the details.
\end{proof}

Note that the above iteration process ends when $q_{m_i}\equiv 0$. Since there is a finite number of $q_m$, this clearly happens in a finite number of steps $i=N$ for some $1\leq N\leq M$. Therefore, we obtain a sequence of decreasing numbers $(\alpha_i)_{1\leq i\leq N}$, a sequence of decreasing solutions $(q_{m_i})_{0\leq i\leq N}$ of \eqref{ODE}, and a sequence $(U_i)_{1\leq i\leq N}$ such that for each $1\leq i\leq N$, $U_i(t,x)$ a periodic traveling wave connecting $q_{m_i}$ to $q_{m_{i-1}}$ satisfying $U_i(0,0)=\alpha_i$. 

For each $1\leq i\leq N$, let $c_i$ be the wave speed of $U_i(t,x)$. Now, we want to prove that 
$$\mathcal{T}:=((q_{m_i})_{0\leq i\leq N}, (U_i,c_i)_{1\leq i\leq N})$$ 
is a propagating terrace. According to Definition \ref{terrace}, it suffices to show that
$$c_1\leq c_2\leq \cdots\leq c_N.  $$
For this purpose, for each $1\leq i\leq N$, let us denote by $(a_{i,k})_{k\in\N}$ the sequence obtained in \eqref{choose-ak} with $\alpha$ replaced by $\alpha_i$. Then Lemma \ref{homo-ptw} provides that 
$ \lim_{k\to\infty}(a_{i,k}-a_{i,k-1})=c_iT$. 
This clearly implies  
\begin{equation}\label{speed-formu}
c_i= \lim_{k\to\infty} \frac{a_{i,k}}{kT}. 
\end{equation}
Notice that $\alpha_{i+1}< \alpha_{i}$ for each $1\leq i \leq N-1$ and that the solution $\hat{u}(t,x)$ is decreasing in $x$. We have $a_{i+1,k}> a_{i,k} $ for all $k\in\N$. 
It then follows immediately from \eqref{speed-formu} that $c_{i+1}\geq c_i$. The existence of propagating terrace is thus proved. 

From the above construction of $U_i$, one sees that $q_{m_i}$ and $U_i$ are steeper than any other entire solution between $0$ and $p$, which immediately gives statement (ii) of the present theorem and the minimality of the propagating terrace $\mathcal{T}$. By Proposition \ref{unique-mpt}, $\mathcal{T}$ is unique up to spatial shifts. Lastly, statement (i) follows directly from Lemma \ref{speed-chara}. The proof is thus complete.
\end{proof}


\subsection{Convergence with Heaviside type initial data}  

This subsection is devoted to the proof of Theorem \ref{converge1}.  Namely, provided that there exists a minimal propagating terrace, we show that it attracts all solutions of \eqref{E} with Heaviside type initial data.   

\begin{proof}[Proof of Theorem \ref{converge1}]
Let $((p_i)_{0\leq i\leq N},(U_i,c_i)_{1\leq i\leq N})$ be a minimal propagating terrace connecting 0 to $p$. By Proposition \ref{unique-mpt}, up to spatial shifts, it is the unique minimal terrace. Moreover, thanks to Theorem \ref{decomposition-existence}, we know that each of the $p_i$ and $U_i$ is steeper than any other entire solution of \eqref{equation} between $0$ and $p$. 
 
Let $\hat{u}(t,x)$ be the solution of \eqref{E} with a given Heaviside type initial function. 
For each $i\in \{1,\,\cdots,\,N\}$, let $(a_{i,k})_{k\in\N}$ be the sequence of real numbers such that
$$\hat{u}(kT,a_{i,k})=U_i(0,0)   \,\,\hbox{ for all }  \,k\in\N.$$
Since $U_i(t,x)$ is steeper than any other entire solution between $0$ and $p$,  
it follows from Lemmas \ref{convege-ak} and \ref{homo-ptw} that for any $t\geq 0$, 
\begin{equation}\label{space-inva-converge}
\hat{u}(t+kT,x+a_{i,k}) \to U_i(t,x) \,\hbox{ as } k\to\infty  \hbox{ in } L^{\infty}_{loc}(\R). 
\end{equation}
Since 
\begin{equation}\label{U-i-mT}
U_i(\cdot,\cdot)=U_i(\cdot+mT,\cdot+c_imT)\, \hbox{ for all }\,m\in\Z,
\end{equation}
\eqref{space-inva-converge} implies that 
\begin{equation}\label{u-etai-Ui}
\hat{u}(t,x+ a_{i, \lfloor t/T \rfloor}) - U_i(t,x+c_i\lfloor t/T \rfloor T)\to 0  \,\hbox{ as } t\to\infty  \hbox{ in } L^{\infty}_{loc}(\R), 
\end{equation}
where $\lfloor t/T \rfloor$ is the floor function of $t/T$, that is, the greatest integer not larger than $t/T$.

For each $i\in \{1,\,\cdots,\,N\}$, let $\eta_i: [0,\infty)\to \R$, $t\mapsto \eta_i(t)$ be a $C^1([0,\infty))$ function satisfying 
\begin{equation}\label{defi-eta}
\eta_i(t)+c_i\lfloor t/T \rfloor T-a_{i, \lfloor t/T \rfloor}\to 0 \,\,\hbox{ as }\,\, t\to\infty. 
\end{equation}
Now we want to show that $(\eta_i)_{1\leq i\leq N}$ are the desired functions satisfying statements (i)-(iii).

Notice from \eqref{space-inva-converge} and \eqref{U-i-mT} that $\lim_{k\to\infty} (a_{i,k+1}- a_{i,k})= c_iT$. This implies that $$ \frac{c_i\lfloor t/T \rfloor T-a_{i, \lfloor t/T \rfloor} }{t} = \frac{\lfloor t/T \rfloor T\left( c_i-\frac{a_{i,\lfloor t/T \rfloor} }{T\lfloor t/T \rfloor}\right)}{t} \to\, 0 \,\,\hbox{ as }\,\, t\to\infty.  $$  
It then follows directly from \eqref{defi-eta} that $\eta_i(t)/t\to 0$ as $t\to\infty$. Thus, statement (i) is proved.

Next, from \eqref{space-inva-converge} again, we see that 
$$a_{i+1,k} - a_{i,k} \to \infty\,\, \hbox{ as }\,\, k\to\infty \,\,\hbox{ for }\,\, i\in \{1,\,\cdots,\,N-1\}.$$
This together with \eqref{defi-eta} immediately yields that 
$$\eta_{i+1}(t)-\eta_{i}(t) \to \infty \,\hbox{ as }\, t\to\infty \,\,\hbox{ whenever }\,\, c_i=c_{i+1}. $$
Thus, statement (ii) is obtained.

It remains to show the uniform convergence in statement (iii). Note that each $U_i$ satisfies
$$\lim_{x\to-\infty} U_{i}(t,x+c_it)= p_{i-1}(t) \quad\hbox{and}\quad \lim_{x\to\infty} U_{i}(t,x+c_it)= p_{i}(t)\,\,\hbox{ uniformly in }\, t\in\R. $$
Then given any small $\epsilon>0$, there exists some $M>0$ such that 
\begin{equation}\label{asymptotic-Ui}
U_{i}(t,c_it+M) \leq p_i(t)+\frac{\epsilon}{2}, \quad U_{i}(t,c_it-M) \geq p_{i-1}(t)-\frac{\epsilon}{2} \,\,\hbox{ for } \,t\in\R.
\end{equation}
By \eqref{u-etai-Ui} and \eqref{defi-eta}, we can find $T_0>0$ large enough such that for $t\geq T_0$, 
\begin{equation}\label{local-con}
\left | \hat{u}(t,x) - U_i(t,x-\eta_i(t)) \right| \leq \frac{\epsilon}{2} \quad\hbox{ for } \, c_it+\eta_i(t) -M\leq x\leq c_it+\eta_i(t)+ M.
\end{equation}
This together with \eqref{asymptotic-Ui} implies that 
\begin{equation*}
\hat{u}(t,c_it+\eta_i(t)+M) \leq p_i(t)+\epsilon, \quad \hat{u}(t,c_it+\eta_i(t)-M) \geq p_{i-1}(t)-\epsilon  \,\,\hbox{ for } \,t\geq T_0.
\end{equation*}
Since $\hat{u}(t,x)$ is decreasing in $x\in\R$, it then follows that for each $i\in \{2,\,\cdots,\,N\}$,
\begin{equation}\label{esti-out-m}
-\epsilon\leq \hat{u}(t,x) -p_{i-1}(t) \leq \epsilon\,\,\hbox{ for }\,\, c_{i-1}t+\eta_{i-1}(t) +M\leq x\leq c_it+\eta_i(t)- M,\,\,t\geq T_0,
\end{equation}
 that 
\begin{equation}\label{right-side}
0<\hat{u}(t,x) \leq \epsilon  \,\,\hbox{ for } \,\,x\geq c_Nt+\eta_N(t)+M, \,t\geq T_0,
\end{equation}
and that 
\begin{equation}\label{left-side}
p(t)-\epsilon \leq \hat{u}(t,x) <p(t)  \,\,\hbox{ for } \,\,x\leq c_1t+\eta_1(t)-M, \,t\geq T_0.
\end{equation}

Combining the inequalities \eqref{local-con}-\eqref{left-side}, we obtain that 
for all $t\geq T_0$ and $x\in\R$, 
\begin{equation*}
\left | \hat{u}(t,x)- \left(  \displaystyle\sum_{i=1}^{N} U_i\big(t,x-\eta_i(t)\big)- \sum_{i=1}^{N} p_i(t)    \right)\right | \leq N\epsilon.
\end{equation*}
This immediately gives statement (iii). The proof is thus complete. 
\end{proof}


\SE{Asymptotic behavior of solutions with initial data satisfying (H2)}
In this section, we study the asymptotic behavior of solutions of \eqref{E} with initial data satisfying (H2), and prove Theorems \ref{converge2} and \ref{converge2-2}. 
Throughout this section, we always put 
\begin{equation}\label{give-a-pm}
u_0(x)=p(0)\, \hbox{ for } \,\, x\in  (-\infty, a_-],\quad  u_0(x)=0\, \hbox{ for } \, x\in  [a_+,\infty) 
\end{equation}
for some $a_+>a_-$. 

Let $((p_i)_{0\leq i\leq N},(U_i,c_i)_{1\leq i\leq N})$ be the minimal propagating terrace connecting $0$ to $p$.  Since it is only unique up to spatial shifts, for definiteness, we normalize it as follows: 
\begin{equation}\label{normalize}
U_i(0,0)=\frac{p_{i-1}(0)+p_{i}(0)}{2} \,\,\hbox{ for each }\, i=1,\,\cdots,N.
\end{equation}
With this normalization, each $U_i$ is uniquely determined, and we will assume this in our discussion below. 

The following lemma is fundamental in this section. 

\begin{lemma}\label{parallel-terrace} 
Let $u(t,x)$ be the solution of \eqref{E} with $u_0$ satisfying {\rm (H2)}. Then for every $w\in\Omega(u)$, either of the following alternatives holds:
\begin{itemize}
\item [(a)]  $w\equiv p_i$ for some integer $0\leq i\leq N$;
\vskip 3pt
\item [(b)]  $w(t,x)$ satisfies
\begin{equation}\label{trapp-w}
U_i(t, x+\xi_0-a_-) \leq w(t,x) \leq U_i(t, x+\xi_0-a_+) \,\,\hbox{ for } t\in\R,\,x\in\R,
\end{equation}
for some integer $1\leq i\leq N$ and some $\xi_0\in\R$, where $a_{\pm}\in\R$ are given in \eqref{give-a-pm}. 
\end{itemize}
Moreover, we have $\{p_i\}_{0\leq i\leq N} \subset \Omega(u)$.
\end{lemma}

\begin{proof}
Let $(k_j)_{j\in\N}\subset\N$ and $(x_{j})_{j\in\N}\subset\R$ be the sequences such that 
\begin{equation}\label{def-omega-s4}
u(t+k_j T,x+x_{j}) \to w(t,x) \,\,\hbox{ as } \,j\to\infty  \,\hbox{ in } C^{1}(\R^2).
\end{equation} 
Since $u_0$ satisfies (H2), it follows easily from the comparison principle that, for each $j\in\N$,  
\begin{equation}\label{hatu-compare}
 \hat{u}(t+k_jT,x+x_{j};a_{-}) \leq  u(t+k_jT,x+x_{j}) \leq \hat{u}(t+k_jT,x+x_{j};a_{+}) 
 \end{equation}
for $t\geq 0$ and $x\in\R$. Here $\hat{u}(t,x;a_{\pm})$ denote the solutions of \eqref{E} with initial functions $p(0)H(a_{\pm}-x)$.   

By standard parabolic estimates and possibly up to a subsequence, we may assume that 
  $$ \hat{u}(t+k_jT,x+x_{j};0) \to  \hat{w}(t,x)\,\hbox{ as } j\to\infty \,\hbox{ in } C^{1}(\R^2),$$
where $\hat{w}(t,x)$ is an entire solution of \eqref{equation}. Clearly, $\hat{w}(t,x)$ is an $\Omega$-limit solution of
 $\hat{u}(t,x;0)$, and  
  $$ \hat{u}(t+k_jT,x+x_{k_j};a_{\pm}) \to  \hat{w}(t,x-a_{\pm})\,\hbox{ as } j\to\infty \,\hbox{ in } C^{1}(\R^2).$$
Passing to the limit as $j\to\infty$ in \eqref{hatu-compare}, we obtain   
\begin{equation*}
\hat{w}(t,x-a_-)\leq w(t,x)\leq \hat{w}(t,x-a_+) \,\hbox{ for } t\in\R,\,x\in\R. 
\end{equation*}
Furthermore, it follows directly from Theorem \ref{converge1} that 
\begin{equation*}
\hat{w}\in \left\{ U_i(\cdot,\cdot+\xi): \xi\in\R,\, 1\leq i\leq N \right\} \cup  \left\{ p_i: 0\leq i\leq N \right\}.
\end{equation*}
This clearly gives the alternatives of the present lemma. 

Similarly as above, one can prove that for any $\Omega$-limit solution $\hat{w}$ of the solution $\hat{u}(t,x;0)$, there exists $\tilde{w}\in \Omega(u)$ such that 
$$ \hat{w}(t,x-a_-)\leq \tilde{w}(t,x)\leq \hat{w}(t,x-a_+) \,\hbox{ for } t\in\R,\,x\in\R. $$ 
In particular, if $ \hat{w} \equiv p_i$ for some $0\leq i\leq N$, then $\tilde{w} \equiv p_i$. This immediately implies the relation $\{p_i\}_{0\leq i\leq N} \subset \Omega(u)$. The proof of Lemma \ref{parallel-terrace}  is thus complete. 
\end{proof}

The reaming part of this section is organized as follows. In Subsection 4.1, we will show that if $w(t,x)$ is an $\Omega$-limit solution satisfying case (b) of Lemma \ref{parallel-terrace}, then it is spatially decreasing. This immediately gives the first part of the conclusions of Theorem \ref{converge2}. The remaining conclusions will be proved in Subsection 4.2. Subsection 4.3 is concerned with the proof of Proposition \ref{unique1} and Theorem \ref{converge2-2}.

\subsection{Monotonicity of $\Omega$-limit solutions}

This subsection is devoted to the proof of the following proposition:

\begin{prop}\label{monotonicity}
Let $u(t,x)$ be the solution of \eqref{E} with $u_0$ satisfying {\rm (H2)} and let $w\in\Omega(u)$ satisfy case {\rm (b)} of Lemma \ref{parallel-terrace}. Then for each $t\in\R$,  $w(t,x)$ is decreasing in $x\in\R$.
\end{prop}

Before giving the proof, let us first show some properties of the solution $u(t,x)$ at any finite time, which will be needed later. 

\begin{lemma}\label{large-x}
Let $u(t,x)$ be the solution of \eqref{E} with $u_0$ satisfying {\rm (H2)}. Then for each $t>0$, 
\begin{equation*}
\partial_x u(t,x) < 0 \,\,\hbox{ for }\, \, x\in (-\infty, a_-)\cup (a_+,\infty),
\end{equation*}
where $a_{\pm}$ are the constants given in \eqref{give-a-pm}. 
\end{lemma}

\begin{proof}
This lemma can be proved by a simple reflection argument. Fix any $x_0\in (-\infty,a_-)$ and define 
$$v(t,x):=u(t,x)-u(t,2x_0-x)\,\,\hbox{ for }  -\infty<x\leq x_0,\,t\geq 0. $$ 
Since $u(t,x)$ is bounded,  $f(t,u)$ is $C^1$-smooth in $u$ and $T$-periodic in $t$, $v(t,x)$ satisfies 
$$v_t=v_{xx}+c(t,x)v\,\,\hbox{ for }  -\infty<x< x_0,\,t> 0, $$  
with some bounded function $c(t,x)$. Moreover, it is easily checked that 
$$v(t,x_0)=0 \,\hbox{ for } t>0,\quad v(0,x)\geq  0   \,\hbox{ for } x<x_0,  \quad \hbox{and}\quad v(0,x)\not\equiv 0.  $$ 
Then the strong maximum principle implies that $v(t,x)>0$ for $t>0$, $x<x_0$. It further follows from the Hopf boundary lemma that $\partial_x v(t,x_0)<0$, and hence, $\partial_x u(t,x_0)<0$ for $t>0$. Since $x_0$ can be chosen arbitrarily in $(-\infty,a_-)$, one obtains $\partial_x u(t,x) < 0$ for $x<a_-$, $t>0$. The case $x>a_+$ can be proved in a similar way. The proof of Lemma \ref{large-x} is complete.
\end{proof}

Recall that $\mathcal{Z}(w)$ denotes the number of sign changes of a real-valued function $w(x)$ defined on $\R$. 
The following lemma is an application of the zero-number theory introduced in Subsection 2.1.

\begin{lemma}\label{zero-shift}
Let $u(t,x)$ be the solution of \eqref{E} with $u_0$ satisfying {\rm (H2)}. Then for any $z\in \R$, we have 
\begin{equation}\label{zero-0-shift}
\mathcal{Z}[u(t,\cdot)-u(t,\cdot+z)]<\infty \,\hbox{ for } \,t>0,
\end{equation}
and 
\begin{equation}\label{zero-T-shift}
\mathcal{Z}[u(t+T,\cdot)-u(t,\cdot+z)]<\infty \,\hbox{ for } \,t>0.
\end{equation}
Furthermore, the above two quantities are nonincreasing in $t>0$. 
\end{lemma}

\begin{proof}
Let us first prove \eqref{zero-0-shift}.  Notice that if $z=0$, then the result is trivial. Without loss of generality, we may assume that $z> 0$, as the case $z<0$ can be argued analogously. 
Since both $u(t,x)$ and $u(t,x+z)$ are bounded solutions of \eqref{equation}, and since $f(t,u)$ is $C^1$-smooth in $u$ and $T$-periodic in $t$, $u(t,x)-u(t,x+z)$ satisfies a linear equation of the form \eqref{zeroeq} with $c(t,x):=(f(t,u(t,x))-f(t,u(t,x+z)))/(u(t,x)-u(t,x+z))$ being bounded.

Denote by $\bar{u}(t,x)$ the solution of the Cauchy problem 
$$\bar{u}_t=\bar{u}_{xx} \hbox{ for } t>0,\, x\in\R; \quad \bar{u}(0,x)=u_0(x) \hbox{ for } x\in\R. $$
Due to the boundedness of the solution $u(t,x)$, we find some $K>0$ such that $-Ku \leq f(t,u)\leq Ku$ for $t\geq 0$, $x\in\R$.  Then, a simple comparison argument implies that
$$\me^{-Kt} \bar{u}(t,x)\leq u(t,x)\leq  \me^{Kt} \bar{u}(t,x) \,\hbox{ for all }\, t\geq 0,\,x\in\R.    $$
By the assumption that $u_0(x)=0$ on $[a_+,\infty)$, we have  
\begin{equation*}
\begin{split}
\frac{u(x,t)}{u(x+z,t)}
\,&\, \geq \exp (-2Kt) \frac{\bar{u}(x,t)}{ \bar{u}(x+z,t)}\\
\,&\,=\exp (-2Kt) \frac{\displaystyle\int_{-\infty}^{a_+} \hbox{exp}\Big(-\frac{(x-y)^2}{4t}\Big) u_0(y)dy}{\displaystyle\int_{-\infty}^{a_+} \hbox{exp}\Big(-\frac{\big(x-y+z\big)^2}{4t}\Big) u_0(y)dy}\\
\,&\,\geq \exp (-2Kt) \hbox{exp}\Big(\frac{2z(x-a_+)+z^2} {4t}\Big).
\end{split}
\end{equation*}
for all $x>a_+$, $t>0$. Since $z>0$, passing to the limit as $x\to\infty$,  we obtain that for each $t>0$,  
\begin{equation*}
\frac{u(x,t)}{u(x+z,t)}\to \infty \quad \hbox{ as } x\to\infty.
\end{equation*}
Similarly, by using the assumption that $u_0(x)=p(0)$ on $(-\infty, a_-]$, we can conclude that for each $t>0$,  
\begin{equation*}
\frac{p(t)-u(x,t)}{p(t)-u(x+z,t)}\to 0 \quad \hbox{ as } x\to-\infty.
\end{equation*}
Then, for any given $t_0>0$, it is easily checked from the above that there exists $L>0$ such that 
$$u(t_0,x)-u(t_0,x+z)>0 \,\,\hbox{ for }  \, x\in (-\infty,-L] \cup [L,\infty).   $$
Thus, by Lemma \ref{zero1}, 
$$\mathcal{Z}[u(t,\cdot)-u(t,\cdot+z)]<\infty\,\hbox{ for all }\, t\geq t_0,$$ 
and this quantity is nonincreasing in $t\geq t_0$. Since $t_0>0$ is arbitrary,  \eqref{zero-0-shift} is proved.

Let us now turn to the proof of \eqref{zero-T-shift}. Similarly as above, we may assume without loss of generality that $z\geq 0$. Since $f(t,u)$ is $T$-periodic in $t$, $u(t+T,x)$ is also a solution of \eqref{equation}. Then $u(t+T,x)-u(t,x+z)$ satisfies a linear solution of the form \eqref{zeroeq} with bounded coefficient $c(t,x)$. 

Since $0\leq u_0\leq p(0)$, it is clear that $0< u(t,x) < p(t)$ for $t>0$, $x\in\R$. Notice that $u_0$ satisfies \eqref{give-a-pm}. Then we can choose $\delta>0$ sufficiently small such that 
$$u(T+t,a_+) >  u(t,a_++z)>0 \quad \hbox{and} \quad u(T+t,a_--z) <  u(t,a_-)<p(t) $$
for $0<t\leq \delta$. 
Since $u(T,x) > u(0,x+z)$ for $x\geq a_+$, and $u(T,x) < u(0,x+z)$ for $x\leq a_--z$, the comparison principle implies that 
$$u(T+t,x) >  u(t,x+z)\,\hbox{ for } \, 0< t\leq \delta,\, x\geq a_+,   $$
and 
$$u(T+t,x) < u(t,x+z)\,\hbox{ for } \, 0< t\leq \delta,\, x\leq a_--z.   $$
It then follows from Lemma \ref{zero1} that 
$$\mathcal{Z}[u(T+t,\cdot) - u(t,\cdot+z)]<\infty \,\,\hbox{ for all } 0<t\leq \delta,$$
and this quantity is nonincreasing in $t>0$. 
This immediately implies \eqref{zero-T-shift}. The proof of Lemma \ref{zero-shift} is thus complete.  
\end{proof}

We are now prepared to prove Proposition \ref{monotonicity}. 

\begin{proof}[Proof of Proposition \ref{monotonicity}]
Let us denote $\varphi(t,x):=\partial_x w(t,x)$ for $t\in\R$, $x\in\R$. Clearly, $\varphi(t,x)$ is a solution of the following linear equation
$$\varphi_t=\varphi_{xx} +\partial_u f(t,w) \varphi\quad \hbox{ for }t\in\R,\, x\in\R, $$
where $\partial_u f(t,w):=\partial f(t,u) /\partial u |_{u=w}$ is bounded for $t\in\R$, $x\in\R$. 

For clarity, we divide the proof into 3 steps. 

{\bf Step 1:} we show that for each $t\in\R$, $\mathcal{Z}[\varphi(t,\cdot)]<\infty$, and all zeros of $x\mapsto \varphi(t,\cdot) $ are simple. 

Let us first note that the function  $v(t,x):=\partial_x u(t,x)$, defined on $(t,x)\in [1,\infty)\times \R$,  solves a linear equation of the form \eqref{zeroeq} with bounded coefficient.  Moreover, by Lemma \ref{large-x}, for each $t\geq 1$, the function $x\mapsto v(t,x)$ does not change sign on the set $(-\infty, a_-)\cup (a_+,\infty)$. It then follows from Lemma \ref{zero1} (i) and (ii) that $\mathcal{Z}[v(t,\cdot)]<\infty$, and it is nonincreasing in $t\geq 1$. 
Therefore, it is a constant for all large $t$, and by Lemma \ref{zero1} (iii), we have 
\begin{equation}\label{simple-v}
\hbox{ the function } x\mapsto v(t,x) \hbox{ has only simple zeros on $\R$ for all large $t$}. 
\end{equation}

Let $(k_j)_{j\in\N}\subset\N$ and $(x_{j})_{j\in\N}\subset\R$ be the sequences such that \eqref{def-omega-s4} holds for $w(t,x)$. Then we have  
\begin{equation}\label{derivative-con}
v(t+k_jT,x+x_{j})\to \varphi(t,x)\, \hbox{ as }\,  j\to\infty,  
\end{equation}
where the convergence holds in  $L_{loc}^{\infty}(\R^2)$. For any $t\in\R$, since $\mathcal{Z}[v(t+k_jT,\cdot)]<\infty$ for all sufficiently large $k_j$, it follows from the semi-continuity of $\mathcal{Z}$ (see \eqref{semi-continuous}) that $\mathcal{Z}[\varphi(t,\cdot)]<\infty$.

Furthermore, by standard parabolic estimates, possibly after extracting a subsequence, we know that the convergence \eqref{derivative-con} also takes place in $C^{1}(\R^2)$.  
In view of this and \eqref{simple-v}, and applying Lemma \ref{zero3}, we see that for each $t\in\R$,  either $\varphi(t,x)\equiv 0$ on $\R$ or $x\mapsto \varphi(t,x)$ has only simple zeros on $\R$. The former is impossible, as we have assumed that $w(t,x)$ satisfies \eqref{trapp-w}, thus it cannot be spatially homogeneous. Consequently, $x\mapsto\varphi(t,x)$ has only simple zeros on $\R$.  This ends the proof of Step 1.

{\bf Step 2:}  We show that for any $z\neq 0$ and $t\in\R$, all zeros of $x\mapsto w(t,x)-w(t,x+z)$ are simple. 

The proof follows from similar arguments to those used in Step 1, therefore we only give its outline.  Let $z\neq 0$ be fixed. Thanks to \eqref{zero-0-shift},  it follows from  Lemma \ref{zero1} that, for all large $t>0$, the function $x\mapsto u(t,x)-u(t,x+z)$ has only finitely many zeros on $\R$ and all of them are simple. 
Note that
 \begin{equation*}
u(t+k_jT,x+x_{j}) -u(t+k_jT,x+z+x_{j})  \to w(t,x)-w(t,x+z)\, \hbox{ as }\,  j\to\infty 
\end{equation*}
in $C^{1}(\R^2)$, and that $w(t,x)-w(t,x+z)$ solves a linear equation of the form \eqref{zeroeq} with bounded coefficient. Then by using Lemma \ref{zero3} and the assumption that $w$ satisfies \eqref{trapp-w}, we obtain the desired result of Step 2. 

{\bf Step 3:} We show that for each $t\in\R$, $w(t,x)$ is decreasing in $x\in\R$. 

Assume by contradiction that there exists some $t_0\in \R$ such that $w(t_0,x)$ is not decreasing in $x\in\R$. 
Notice that
$$\lim_{x\to\infty}w(t_0,x)=p_i(t_0)<p_{i-1}(t_0) = \lim_{x\to-\infty} w(t_0,x). $$
It then follows from Step 1 that 
$$1< \mathcal{Z}[\varphi(t_0,\cdot)] <\infty, $$
and all zeros of $x\mapsto \varphi(t_0,\cdot) $ are simple. Denote by $\xi_1$ the minimum of these zeros, and define 
$$x_1:=\min\{x>\xi_1: w(t_0,x)=w(t_0,\xi_1),\, \varphi(t_0,x)\neq 0 \}.$$
Clearly, $x_1\in (\xi_1,\infty)$ is well defined and $\varphi(t_0,x_1)<0$. 
Furthermore, let $\xi_{2}$ be the maximum of the zeros of $\varphi(t_0,\cdot)$ to the left of $x_1$, and let $x_2$ be the point such that $x_2<\xi_1$ and $w(t_0,x_2)=w(t_0,\xi_{2})$. It is then easily checked that $x_2<\xi_1<\xi_2<x_1$, 
$$ \varphi(t_0,x)<0 \,\hbox{ for } x\in [x_2,\xi_1)\cup (\xi_{2},x_1],  $$
and 
$$w(t_0,x_2)= w(t_0,\xi_{2}),\quad   w(t_0,\xi_1)= w(t_0,x_1). $$
This implies in particular that for each $x\in [x_2,\xi_1]$, there exists a unique $y\in  [\xi_2,x_1]$ such that $w(t_0,x)= w(t_0,y)$. Thus, we can find a continuous function $\gamma: [x_2,\xi_1] \to [\xi_2,x_1]$ satisfying $w(t_0,x)= w(t_0,\gamma(x))$. It is easily seen that $\gamma(x_2)=\xi_2$, $\gamma(\xi_1)=x_1$, and that
$$\varphi(t_0,x_2)- \varphi(t_0,\gamma(x_2))<0\quad\hbox{and}\quad\varphi(t_0,\xi_1)- \varphi(t_0,\gamma(\xi_1))>0.$$
From this, it immediately follows that there exists some $x_0\in (x_2,\xi_1)$ such that 
$$\varphi(t_0,x_0)= \varphi(t_0,\gamma(x_0)).$$ 
Therefore, $x_0$ is a degenerate zero of the function $x\mapsto w(t_0,x)-w(t_0,\gamma(x_0)-x_0+x)$, which is a contradiction with the conclusion of Step 2.  

Now we can conclude that for each $t\in\R$, $w(t,x)$ is decreasing in $x\in\R$. The proof of Proposition  \ref{monotonicity} is thus complete.
\end{proof}

\subsection{Completion of the proof of Theorem \ref{converge2}} In this subsection, we prove that elements of ${\omega}(u)$ can be classified as stated in Theorem \ref{converge2}. The key ingredient of our proof is to compare 
the steepness between $w(t,x)$ and $w(t+T,x)$, where $w$ is any $\Omega$-limit solution satisfying case (b) of Lemma \ref{parallel-terrace}. We will show that either $w(t+T,x)$ is equal to a spatial shift of $w(t,x)$, or $w(t+T,x)$ is strictly steeper or strictly less steep than $w(t,x)$. Once we know this, we will be able to prove that $w(t,x)$ is either a periodic traveling wave or it is a heteroclinic solution connecting two periodic traveling waves.

\begin{lemma}\label{steep-l-steep}
Suppose that $w\in\Omega(u)$ satisfies case {\rm (b)} of Lemma \ref{parallel-terrace}. Then one of the following  holds: 
\begin{itemize}
\item[(a)]  $w(t,x)$ is a periodic traveling wave connecting $p_{i}$ to $p_{i-1}$ with wave speed $c_i$ for some $1\leq i \leq N$; 
\vskip 3pt
\item[(b)] $w(t+T,x)$ is strictly steeper than $w(t,x)$; 
\vskip 3pt
\item[(c)] $w(t+T,x)$ is strictly less steep than $w(t,x)$.
\end{itemize}
Furthermore, if case {\rm (b)} (resp. case {\rm (c)}) occurs, then $w(t+kT,x)$ is strictly steeper (resp. strictly less steep) than $w(t,x)$ for all positive integer $k$. 
\end{lemma}

\begin{proof}
Let $z$ be an arbitrary real number. We already know from Lemma \ref{zero-shift} that $\mathcal{Z}[u(T+t,\cdot)-u(t,\cdot+z)] $ is finite for all $t>0$, and it is nonincreasing in $t>0$. By using Lemma \ref{zero1}, we see that the function $x\mapsto u(T+t,x)-u(t,x+z)$ has only simple zeros for all large $t$. 
Then, similarly as in the proof of Proposition \ref{monotonicity}, one can apply Lemma \ref{zero3} to conclude that either of the following alternatives holds:
\begin{equation}\label{w-equiv-wT}
w(t+T,x)\equiv w(t,x+z)
\end{equation} 
or for each $t\in\R$,
\begin{equation}\label{finite-simple}
\hbox{ the function } x\mapsto w(t+T,x)- w(t,x+z) \hbox{ has only simple zeros. }
\end{equation}

As we have assumed that $w(t,x)$ satisfies case {\rm (b)} of Lemma \ref{parallel-terrace}, there exists some $1\leq i\leq N$ such that  
\begin{equation}\label{asym-w}
\lim_{x\to-\infty} w(t,x) = p_{i-1}(t),\quad \lim_{x\to\infty} w(t,x) = p_{i}(t)  \,\,\hbox{ locally uniformly in }\,t\in\R.
\end{equation}

For clarity, we divide the rest of the proof into three steps. 

{\bf Step 1:} We show that if \eqref{w-equiv-wT} holds for some $z\in\R$, then $z= -c_iT$ and case (a) of the present lemma occurs. 

It easily seen from \eqref{w-equiv-wT} and \eqref{asym-w} that $w(t,x)$ is a periodic traveling wave connecting $p_{i}$ to $p_{i-1}$, and $-z/T$ is the wave peed. Thus, we only need to show that $z= -c_iT$.  Suppose the contrary that this is not true. We may assume without loss of generality that $z>-c_iT$. Then, since $U_i$ is a periodic traveling wave with speed $c_i$, it follows that 
$$U_i(mT,x-mz)\, \to \,  p_{i-1}(0)  \,\hbox{ as } \, m\to\infty \, \hbox{ locally uniformly in } \,x\in\R.    $$
Notice from \eqref{trapp-w} that, $w$ satisfies   
$$w(mT,-mz) \geq  U_i(mT, \xi_0-a_--mz) \,\hbox{ for all }\, m\in\Z.    $$
Passing to the limit as $m\to\infty$, we obtain $\liminf_{m\to\infty} w(mT,-mz) \geq p_{i-1}(0)$. This is impossible, since $w(mT,-mz)=w(0,0)<p_{i-1}(0)$ for each $m\in\Z$. Therefore, $z= -c_iT$ holds true, and hence, $w(t,x)$ is a periodic traveling wave with speed $c_i$.

{\bf Step 2:} We assume that \eqref{finite-simple} holds for all $t\in\R$ and $z\in\R$, and prove either case (b) or case (c) of the present lemma occurs. 

Since $w(t,x)$ is spatially decreasing by Proposition \ref{monotonicity}, for any $t\in\R$, we can find a $C^1$ function $\alpha \mapsto \zeta(\alpha;t)$ defined on $(p_i(t), p_{i-1}(t))$ such that 
\begin{equation}\label{inverse} 
w(t,\zeta(\alpha;t)) = \alpha \,\,\hbox{ for }  \, \alpha\in (p_i(t), p_{i-1}(t)). 
\end{equation}

Let $t_0\in\R$ be an arbitrary time and let $\zeta(\alpha;t_0)$ and $\zeta(\alpha;t_0+T)$ be the functions given as in \eqref{inverse}.   Since $p_i$ and $p_{i-1}$ are $T$-periodic, it is clear that $\zeta(\alpha;t_0+T)$ is well defined on  $(p_i(t_0), p_{i-1}(t_0))$ and that
\begin{equation}\label{wto-wtoT} 
w(t_0,\zeta(\alpha;t_0)) =w(t_0+T,\zeta(\alpha;t_0+T))  = \alpha \,\,\hbox{ for }  \, \alpha\in (p_i(t_0), p_{i-1}(t_0)).
\end{equation}

Next, we claim that 
\begin{equation*}
\partial_x w(t_0+T,\zeta(\alpha;t_0+T)) \neq  \partial_x w(t_0,\zeta(\alpha;t_0))\, \hbox{ for all }  \, \alpha\in (p_i(t_0), p_{i-1}(t_0)).
\end{equation*} 
Otherwise, there exists some $\alpha_0\in (p_i(t_0), p_{i-1}(t_0))$ such that the equality holds at $\alpha=\alpha_0$. This together with \eqref{wto-wtoT} implies that $x=\zeta(\alpha_0;t_0+T)$ is a degenerate zero of the function $x\mapsto w(t_0+T,x)-w(t_0, x+\zeta(\alpha_0;t_0)-\zeta(\alpha_0;t_0+T))$, which is a contradiction with our assumption that \eqref{finite-simple} holds for all $t\in\R$ and $z\in\R$. Thus, our claim is proved. 

It then follows that either 
$$\partial_x w(t_0+T,\zeta(\alpha;t_0+T)) > \partial_x w(t_0,\zeta(\alpha;t_0))   \,\,\hbox{ for all } \, \alpha \in (p_i(t_0), p_{i-1}(t_0)), $$
or 
$$\partial_x w(t_0+T,\zeta(\alpha;t_0+T)) < \partial_x w(t_0,\zeta(\alpha;t_0))   \,\,\hbox{ for all } \, \alpha \in (p_i(t_0), p_{i-1}(t_0)). $$
Since $\alpha \in (p_i(t_0), p_{i-1}(t_0))$ is arbitrary, we obtain that $w(t_0+T,\cdot)$ is either strictly steeper or strictly less steep than $w(t_0,\cdot)$ in the sense of Definition \ref{steepness}. Furthermore, by Lemma \ref{ini-steep} and the arbitrariness of $t_0\in\R$, we can conclude that either case (b) or case (c) occurs.

{\bf Step 3:} We show that if case {\rm (b)} (resp. case {\rm (c)}) occurs, then $w(t+kT,x)$ is strictly steeper (resp. strictly less steep) than $w(t,x)$ for all positive integer $k$.  

Let us assume without loss of generality that case {\rm (b)} occurs. Then, for each $k\in\N$, $w(t+kT,x)$ is strictly steeper than $w(t+(k-1)T,x)$. For any $t\in\R$, let $\alpha\mapsto \zeta(\alpha;t+kT)$ be the function given as in \eqref{inverse}. It then follows that for each $t\in\R$, $\alpha\in (p_i(t), p_{i-1}(t))$ and $k\geq 1$, there holds
$$w(t+kT,\zeta(\alpha;t+kT)) =\cdots= w(t+T,\zeta(\alpha;t+T)) =  w(t,\zeta(\alpha;t)) =\alpha, $$ 
and
$$\partial_x w(t+kT,\zeta(\alpha;t+kT)) >\cdots> \partial_x w(t+T,\zeta(\alpha;t+T)) > \partial_x w(t,\zeta(\alpha;t)).$$ 
This implies that $w(t+kT,x)$ is strictly steeper than $w(t,x)$.  The proof of Lemma \ref{steep-l-steep} is thus complete.     
\end{proof}

In the following lemma, we show that if case (b) or case (c) of Lemma \ref{steep-l-steep} holds, then $w(t,x)$ is a heteroclinic solution connecting two periodic traveling waves. 

\begin{lemma}\label{heteroclinic-orbit}
If $w\in\Omega(u)$ satisfies case {\rm (b)} (resp. case {\rm (c)}) of Lemma \ref{steep-l-steep}, then there are $V_{\pm} \in \Omega(u)$ such that 
\begin{equation}\label{heteroclinic-1}
w(t,x)- V_{\pm}(t,x) \to 0 \, \hbox{ as } t\to\pm \infty \, \hbox{ uniformly in } \,x\in\R,
\end{equation}
and $V_{+}(t,x)$ is strictly steeper (resp. strictly less steep) than $V_{-}(t,x)$. 
Furthermore, $V_{\pm}$ are periodic traveling waves of \eqref{equation} connecting $p_{i}$ to $p_{i-1}$ and sharing the same speed $c_i$ for some $i\in\{1,\cdots,N\}$.
 \end{lemma}

\begin{proof}
We only give the proof in the case where $w(t+T,x)$ is strictly steeper than $w(t,x)$, as the proof for the other case is identical. 
 
Note that $w$ satisfies \eqref{trapp-w} for some $1\leq i\leq N$ and some $\xi_0\in\R$.  
We can find a sequence $(z_m)_{m\in\Z}\subset \R$ such that 
\begin{equation*}
w(mT,z_m)=\frac{p_{i-1}(0)+p_i(0)}{2} \,\,\hbox{ for }\, m\in\Z.
\end{equation*}
Furthermore, by the normalization of $U_i$ in \eqref{normalize}, we have
 \begin{equation}\label{estimate-zm}
c_imT-\xi_0+a_-\leq  z_m \leq c_imT-\xi_0+a_+  \,\,\hbox{ for }\, m\in\Z.
\end{equation}
For each $m\in\Z$, let us define  
$$w_m(t,x):= w(t+mT,x+z_m)\,\,\hbox{ for }\, t\in\R,\,x\in\R.$$ 
Clearly, for each $m\in\Z$, $w_m(0,0)=(p_{i-1}(0)+p_i(0))/2$. Moreover, by Lemma \ref{steep-l-steep}, we see that for any integer $k\geq 1$, $w_{m+k}(t,x)$ is strictly steeper than $w_{m}(t,x)$. 

We now show the convergence of $w(t,x)$ to a periodic traveling wave as $t\to \infty$. The convergence as $t\to -\infty$ can be proved analogously. We proceed with three steps.

{\bf Step 1: }we prove the convergence of $w_m(t,x)$ as $m\to\infty$ in $L^{\infty}_{loc}(\R^2)$. 

Let us first notice that from standard parabolic estimates, the sequence $\{w_m(t,x)\}_{m\in\Z}$ is uniformly bounded along with their derivatives. Therefore, it is relatively compact for the topology of $L^{\infty}_{loc}(\R^2)$ with respect to $(t,x)$.  Then there exist a subsequence of integers $(m_j)_{j\in\N}$ ($m_j\to\infty$ as $j\to\infty$) and an entire solution $W_+(t,x)$ of \eqref{equation} such that 
$$w_{m_j}(t,x) \to W_{+}(t,x) \,\,\hbox{ as } \, j\to\infty \,\hbox{ in } \, L^{\infty}_{loc}(\R^2).$$
Clearly, $W_{+}$ belongs to $\Omega(u)$, as $\Omega(u)$ is compact in $L^{\infty}_{loc}(\R^2)$ (see Subsection 2.2). Since for any fixed $m\in\Z$,  $w_{m_j}(t,x)$ is strictly steeper than $w_m(t,x)$ for all large $m_j$, it follows from Lemma \ref{semicontinuous}  that $W_+(t,x)$ is  steeper than each $w_{m}(t,x)$.  

Let $(\tilde{m}_j)_{j\in\N}$ be another subsequence of integers such that $\tilde{m}_j \to\infty$ as $j\to\infty$, and that 
$$w_{\tilde{m}_j}(t,x) \to \tilde{W}_{+}(t,x) \,\,\hbox{ as } \, j\to\infty \,\hbox{ in } \, L^{\infty}_{loc}(\R^2)$$
for some $\tilde{W}_{+}\in \Omega(u)$.  Similarly as above, we can conclude that $\tilde{W}_+(t,x)$ is steeper than each $w_{m}(t,x)$.   

In particular, we have $W_+(t,x)$ is steeper than each $w_{\tilde{m}_j}(t,x)$, and $\tilde{W}_+(t,x)$ is steeper than 
each $w_{m_j}(t,x)$. Then, by using Lemma \ref{semicontinuous} again, we see that the two functions $W_+$ and $\tilde{W}_{+}$ are steeper than each other. Furthermore, neither lies strictly above or below the other one, since 
$$\tilde{W}_{+}(0,0)=W_{+}(0,0)=\frac{p_{i-1}(0)+p_i(0)}{2}.$$
This implies that $\tilde{W}_{+}\equiv W_{+}$. Therefore, 
the whole sequence $w_m(t,x)$ converges to $W_+(t,x)$  as $m\to\infty$ in $L^{\infty}_{loc}(\R^2)$.

{\bf Step 2: }we show that $W_+$ is a periodic traveling wave connecting $p_{i}$ to $p_{i-1}$, and $c_i$ is the wave speed . 

From the definition of $w_m$ and the fact that $w$ satisfies \eqref{trapp-w}, it is easily seen that  for each $m\in\Z$,
\begin{equation}\label{wm-wm+1}
w_m(\cdot+T,\cdot+z_{m+1}-z_m)\equiv  w_{m+1}(\cdot,\cdot),
 \end{equation}
and
\begin{equation}\label{wm-trap-Ui}
U_i(t, x+z_m-c_imT+\xi_0-a_-) \leq w_m(t,x) \leq U_i(t, x+z_m-c_imT+\xi_0-a_+)
\end{equation}
for $t\in\R$, $x\in\R$. By \eqref{estimate-zm}, the sequences $(z_{m+1}-z_m)_{m\in\Z}$ and $(z_m-c_imT)_{m\in\Z}$ are bounded. Then there exist a subsequence $(\bar{m}_k)_{k\in\N}$ and $l_1\in\R$, $l_2\in\R$ such that $\bar{m}_k\to\infty$ as $k\to\infty$, and that
$$z_{\bar{m}_k+1}-z_{\bar{m}_k} \to l_1 \quad\hbox{and}\quad    z_{\bar{m}_k}-c_i\bar{m}_kT\to l_2  \,\,\hbox{ as } k\to\infty. $$
Passing to the limits along the subsequence $\bar{m}_k\to\infty$ in \eqref{wm-wm+1} and \eqref{wm-trap-Ui}, we obtain  
\begin{equation*}
W_+(\cdot+T,\cdot+l_1)\equiv W_+(\cdot,\cdot), 
\end{equation*}
and 
\begin{equation*}
U_i(t, x+l_2+\xi_0-a_-) \leq W_+(t,x) \leq U_i(t, x+l_2+\xi_0-a_+)\,\hbox{ for } \,t\in\R,\,x\in\R.
\end{equation*}
This implies that $W_+(t,x)$ is a periodic traveling wave connecting $p_i$ to $p_{i-1}$, and $l_1/T$ is the wave speed. 
Furthermore, by the arguments used in Step 1 of the proof of Lemma \ref{steep-l-steep}, we have $l_1=c_i/T$. 
This completes the proof of Step 2.

{\bf Step 3:} we show the convergence of $w(t,x)$ to a periodic traveling wave as $t\to\infty$ in $L^{\infty}(\R)$. 

Notice that the limit $l_1=c_i/T$ does not  depend on the choice of the subsequence  $(\bar{m}_k)_{k\in\N}$. Thus, the whole sequence $z_{m+1}-z_m$ converges to $l_1$ as $m\to\infty$, and hence, the whole sequence $z_{m}-c_imT$ converges to  $l_2$ as $m\to\infty$.

Let us now write 
$$V_+(t,x):= W_+(t,x-l_2) \,\,\hbox{ for } t\in\R,\,x\in\R.$$
Clearly, $V_+(t,x)$ is a periodic traveling wave connecting $p_{i}$ to $p_{i-1}$ with speed $c_i$. We want to prove that $w(t,x)$ converges to $V_+(t,x)$ as $t\to\infty$ in $L^{\infty}(\R)$. 

Let $\epsilon>0$ be a given small number. From the asymptotics of $U_i$ and $V_+$, 
there exists $C>0$ such that 
$$p_{i-1}(t)-\frac{\epsilon}{2} \leq  U_{i}(t,x),\, V_+(t,x) \leq  p_{i-1}(t) \,\,\hbox{ for all }\, x-c_it\leq -C,\,t\in\R,   $$
and
$$p_{i}(t) \leq  U_{i}(t,x),\, V_+(t,x) \leq p_{i}(t)+\frac{\epsilon}{2} \,\,\hbox{ for all }\, x-c_it\geq C,\,t\in\R.   $$
It then follows from \eqref{trapp-w} that, after making $C$ larger if necessary, 
\begin{equation}\label{w-V+-ep}
|w(t,x)-V_+(t,x)| \leq \epsilon \,\,\hbox{ for all }  |x-c_it|\geq C,\,t\in\R.
\end{equation}

On the other hand, we know from Step 1 that $w_m(t,x)$ converges to $V_+(t,x+l_2)$ as $m\to\infty$ in $L^{\infty}_{loc}(\R^2)$. Thus, we have  
$$w(t,x+z_{\lfloor t/T\rfloor}) - V_+(t, x+ c_i\lfloor t/T\rfloor T+l_2)\to 0 \hbox{ as } t\to\infty \hbox{ in } L^{\infty}_{loc}(\R),$$
where $\lfloor t/T \rfloor$ denotes the floor function of $t/T$, as introduced in the proof of Theorem \ref{converge1}. 
Notice that $c_i{\lfloor t/T\rfloor}T - z_{\lfloor t/T\rfloor}\to -l_2$ as $t\to \infty$.  There exists some $T_0>0$ sufficiently large such that 
$$\left|w(t,x) - V_+(t, x) \right|\leq  \epsilon \,\,\hbox{ for all } t\geq T_0,\,  |x-c_it|\leq C.  $$
Combining this with \eqref{w-V+-ep}, we immediately obtain 
$$\left|w(t,x) - V_+(t, x) \right|\leq  \epsilon \,\,\hbox{ for all } t\geq T_0,\, x\in\R. $$
Since $\epsilon>0$ is arbitrary, Step 3 is proved.

From the proof of Step 1, we see that $V_+(t,x)$ is steeper than $w(t,x)$. Since $V_+(t,x)$ is not equal to $w(t,x)$ up to any spatial shift, $V_+(t,x)$ is strictly steeper than $w(t,x)$. 

Similarly as above, one can prove that there exists another periodic traveling wave $V_-(t,x)$ connecting $p_{i}$ to $p_{i-1}$ with speed $c_i$ such that  $w(t,x)$ converges to $V_-(t,x)$ as $t\to-\infty$ in $L^{\infty}(\R)$ and that $V_-(t,x)$ is strictly less steep than $w(t,x)$. Thus, \eqref{heteroclinic-1} is proved. It is straightforward to check that $V_+(t,x)$ is strictly steeper than $V_-(t,x)$. This ends the proof of  Lemma \ref{heteroclinic-orbit}. 
\end{proof}

Clearly, Theorem \ref{converge2} follows from Proposition  \ref{monotonicity} and Lemmas \ref{parallel-terrace}, \ref{steep-l-steep}, \ref{heteroclinic-orbit}.

\subsection{Proof of Theorem \ref{converge2-2}}
In this subsection, we let Assumption \ref{non-degenerate} hold and give the proof of	 Proposition \ref{unique1} and Theorem \ref{converge2-2}.

\begin{proof}[Proof of Proposition \ref{unique1}]
The proof is similar to that of \cite[Theorem 4.1]{gm} which is devoted to a spatially periodic problem.  
For the completeness of the present paper, we give its outline.

By Lemma \ref{speed-chara}, either $q_{2}$ is stable from below or $q_{1}$ is stable from above. Without loss of generality, we assume that the former case occurs. Then by Assumption \ref{non-degenerate}, there exist some $\sigma_0>0$ and a $T$-periodic function $g(t)$ such that 
\begin{equation}\label{q-stable}
\int_{0}^{T}g(t)dt\leq 0\,\, \hbox{ and }\,\,    \partial_u f(t, u) \leq g(t)\, \,\hbox{ for all } \, u\in (q_{2}(t)-\sigma_0,q_2(t)],\,t\in\R.
\end{equation}

Let us first show that $c_1\leq c_2$. Suppose the contrary that $c_1>c_2$. By the characterization of periodic traveling waves, it is known that for each $j=1,\,2$, $V_j(t,x+c_jt)$ is periodic in $t$, decreasing in $x$, and it converges to $q_2(t)$ as $x\to -\infty$ uniformly in $t\in\R$. Then there exist some $x_0\in\R$ and some $\sigma\in (0,\sigma_0)$ such that 
$$V_1(t,x+c_1t) \in (q_2(t)-\sigma_0, q_2(t)] \,\,\hbox{ for all } t\in\R,\, x\leq x_0,  $$
$$V_1(t,x+c_1t) \in [q_2(t)-\sigma_0, q_2(t)-\sigma] \,\,\hbox{ for all } t\in\R,\, x= x_0,  $$
and that for some large negative $t_0<0$,  
$$V_2(t,x+c_1t) \in [q_2(t)-\sigma/2, q_2(t)] \,\,\hbox{ for all } \,t\leq t_0,\,x\leq  x_0. $$

Let us define
\begin{equation*}
V(t,x): =V_2(t,x+c_1t)- V_1(t,x+c_1t) \,\, \hbox{ for } \,t\leq t_0,\,x\leq x_0.
\end{equation*}
It is easily checked that the function $V(t,x)$ satisfies: 
\begin{equation*}
\left\{\baa{ll}
\smallskip  V_t= V_{xx} +c_1V_x + \eta(t,x)V & \hbox{ for }\, \,t\leq  t_0,\,x< x_0,  \vspace{3pt}\\
\smallskip  V(t,x)\geq \sigma/2 & \hbox{ for }\, \,t\leq t_0,\,x=x_0, \vspace{3pt}\\
\smallskip \lim_{x\to-\infty} V(t,x)=0 & \hbox{ for }\, \,t\leq t_0, 
\eaa\right.
\end{equation*}
where 
\begin{equation*}
\eta(t,x):=\left\{\baa{ll}
\smallskip \displaystyle \frac{f(t,V_2(t,x+c_1t))-f(t,V_1(t,x+c_1t))}{V_2(t,x+c_1t)-V_1(t,x+c_1t)} & \hbox{ if }\, \,V(t,x)\neq 0,  \vspace{3pt}\\
\smallskip \partial_{u} f(t,V_1(t,x+c_1t)) &  \hbox{ if }\, \,V(t,x)=0.
\eaa\right.
\end{equation*}
Due to \eqref{q-stable}, we have $\eta(t,x)\leq g(t)$ for $t\leq t_0$, $x\leq  x_0$, and $\lambda:=-  \frac{1}{T}\int_{0}^T g(t)dt\geq 0$. It is straightforward to check that for any $\kappa>0$, the function $-\kappa \phi(t)$ is a subsolution of the equation satisfied by $V(t,x)$, where $\phi\in C^1(\R)$ is the solution of 
\begin{equation}\label{ode-compare}
\left\{\baa{l}
\smallskip  \phi_t- g(t) \phi =\lambda \phi\,\, \hbox{ for } \,t\in\R, \vspace{3pt}\\
\smallskip \phi(t+T)=\phi(t)\, \hbox{ for }\, t\in\R,  \quad \phi(0)=1. \eaa\right.
\end{equation}
Notice that $\liminf_{t\to-\infty} \inf_{x\leq x_0} V(t,x)\geq 0$. 
For any $\kappa>0$, there exists a sufficiently large negative integer $k$ such that 
\begin{equation}\label{v-geq-phi}
V(t_0+kT,x)\geq -\kappa \phi(t_0)\,\,  \hbox{ for all } \, x\leq x_0. 
\end{equation}
It then follows from the comparison principle that 
$$V(t+kT,x)\geq -\kappa \phi(t)\,\,\hbox{ for all }\,\,x\leq x_0,\, t\geq t_0.$$
Since $\phi(t)$ is $T$-periodic, we have $V(t_0,x)\geq -\kappa \phi(t_0)$ for all $x\leq x_0$. 
Furthermore, by the arbitrariness of $\kappa>0$, it follows that $V(t_0,x)\geq 0$ for all $x\leq x_0$, that is, 
$$V_2(t_0,x+c_1t_0)\geq  V_1(t_0,x+c_1t_0)\,\,\hbox{ for all } x\leq x_0. $$

Since $V_1$ is steeper than $V_2$, $\eta(t_0,\cdot)$ must be nonpositive on the left of any zero, and hence,  
$$V_2(t_0,x+c_1t_0)\geq  V_1(t_0,x+c_1t_0)\,\,\hbox{ for all } x\in\R. $$
Furthermore, by the comparison principle (applied to equation \eqref{E}), we obtain 
$$V_2(t,x+c_1t_0)\geq  V_1(t,x+c_1t_0)\,\,\hbox{ for all } t\geq t_0\,\,x\in\R. $$
This implies that $V_2$ has to be faster than $V_1$, that is, $c_2\geq c_1$, which is a contradiction with our assumption at the beginning of the present proof. 

It remains to show that if $c_1=c_2$, then $V_1\equiv V_2$ up to a spatial shift. Assume by contradiction that this is not true. Then, since $V_1$ is steeper than $V_2$, it is not difficult to find some $\xi^*\in\R$ and  a continuous real-valued function $x^*(t)$ such that $x^*(t)$ is the only intersection of $V_1(t, \cdot-\xi^*+c_1t)-V_2(t,\cdot+c_1t)$, and that
\begin{equation}\label{asym-V-Ui}
q_{2}(t) -\sigma_0< V_2(t,x +c_1t) \leq    V_1(t,x-\xi^*+c_1t)<q_{2}(t) 
\end{equation}
for $x\leq  x^*(t)$, $t\in\R$, where $\sigma_0$ is the constant given in \eqref{q-stable}.  Clearly, $x^*(t)$ is $T$-periodic. Let us define 
$$W(t,x):= V_2(t,x +c_1t) - V_1(t,x-\xi^*+c_1t)  \,\,\hbox{ for } \, x\leq  x^*(t),\,t\in\R,$$
and let $\kappa_0>0$ be such that $W(0,x)\geq -\kappa_0\phi(0)$ for $x\leq x^*(0)$.
By similar comparison arguments to those used in showing \eqref{v-geq-phi}, we can derive that 
$$W(t,x) \geq -\kappa_0\phi(t)\,\,\hbox{ for }\, t\in\R,\,x\leq x^*(t),$$  
where $\phi\in C^1(\R)$ is the solution of \eqref{ode-compare}. Now we can define 
$$\kappa_*:= \min\left\{\kappa>0:\, W(t,x) \geq -\kappa\phi(t) \hbox{ for } t\in\R,\, x\leq x^*(t)\right\}.$$
Since $W(t,x)<0$ for $x\in (-\infty, x^*(t))$, $t\in\R$, due to \eqref{asym-V-Ui}, it is clear that $\kappa_*>0$.  Then, by using the fact that $W(t,x^*(t))\equiv 0$ and $W(t,-\infty)\equiv 0$, we obtain $W(t,x) \geq -\kappa_*\phi(t)$ with equality at some $t_1\in\R$, $x_1\in (-\infty, x^*(t_1))$. Applying the strong maximum principle, we have $W(t,x) \equiv \kappa_*\phi(t)$, which is obviously impossible. Therefore, $V_1\equiv V_2$ up to a spatial shift.
 This ends the proof of Proposition \ref{unique1}. 
\end{proof}

We are now ready to complete the proof of Theorem \ref{converge2-2}. 

\begin{proof}[Proof of Theorem \ref{converge2-2}]
Let $u(t,x)$ be the solution of \eqref{E} with $u_0$ satisfying {\rm (H2)}. By Theorem \ref{converge2}  and Proposition \ref{unique1}, we immediately obtain that 
\begin{equation}\label{Omega-H2}
\Omega(u)= \left\{ U_i(\cdot,\cdot+\xi): \xi\in\R,\, 1\leq i\leq N \right\} \cup  \left\{ p_i: 0\leq i\leq N \right\}. 
\end{equation}
To complete the proof, it remains to find $C^1([0,\infty))$ functions $(\eta_i)_{1\leq i\leq N}$ such that statements (i)-(iii) of Theorem \ref{converge1} hold for the solution $u(t,x)$ considered here.  To do this, for each $i=1,\cdots,N$, let us choose a sequence $(b_{i,k})_{k\in\N}\subset \R$ such that 
$$u(kT,b_{i,k})= \frac{p_{i-1}(0)+p_i(0)}{2}  \,\,\hbox{ for each }  \,k\in\N.$$
By standard parabolic estimates, the sequence $\{ u(t+kT,x+b_{i,k})\}_{k\in\N}$ is relatively compact for the topology of $L^{\infty}_{loc}(\R^2)$. Thus, it has a subsequence that converges in $L^{\infty}_{loc}(\R^2)$ to an element $w\in \Omega(u)$ with $w(0,0)=(p_{i-1}(0)+p_i(0))/2$. Furthermore, thanks to \eqref{Omega-H2}, 
$U_i(t,x)$ is the only element in $\Omega(u)$ satisfying $U_i(0,0)= (p_{i-1}(0)+p_i(0))/2$. It then follows that 
\begin{equation*}
u(t+kT,x+b_{i,k})  \to U_i(t,x) \,\hbox{ as } k\to\infty  \hbox{ in } L^{\infty}_{loc}(\R^2). 
\end{equation*}
This immediately implies that
\begin{equation*}
u(t,x+b_{i, \lfloor t/T \rfloor}) - U_i(t,x+c_i\lfloor t/T \rfloor T)\to 0  \,\hbox{ as } t\to\infty  \hbox{ in } L^{\infty}_{loc}(\R), 
\end{equation*}
where $\lfloor t/T \rfloor$ is the floor function of $t/T$.

For each $i=1,\cdots,N$, let $\eta_i: [0,\infty)\to \R$, $t\mapsto \eta_i(t)$ be a $C^1([0,\infty))$ function satisfying 
\begin{equation*}
\eta_i(t)+c_i\lfloor t/T \rfloor T-b_{i, \lfloor t/T \rfloor}\to 0 \,\,\hbox{ as }\,\, t\to\infty. 
\end{equation*}
Then by modifying the proof of Theorem \ref{converge1}, one can verify that $(\eta_i)_{1\leq i\leq N}$ are the desired functions. Indeed, in view of the above construction of $(\eta_i)_{1\leq i\leq N}$, the same arguments as those used in the proof of Theorem \ref{converge1} can show that $(\eta_i)_{1\leq i\leq N}$ satisfy statements (i)-(ii), and that for any $1\leq i\leq N$ and any large $M>0$, 
$$  \left \| u(t,x) - U_i(t,x-\eta_i(t)) \right\|_{L^{\infty}([c_it+\eta_i(t) -M, c_it+\eta_i(t) +M])}  \to 0 \quad\hbox{ as } \,t\to\infty.$$
 But for the approach of $u(t,x)$ to $p_i(t)$ on the region $[c_it+\eta_i(t) +M, c_{i+1}t+\eta_{i+1}(t) -M])$, the proof is  different, since $u(t,x)$ is not spatially decreasing any more.  In such a situation, the same result can be proved by using the fact that $u(t,x)$ can be bounded from above and below by solutions with Heaviside type initial functions and that such solutions satisfy statement (iii) of Theorem \ref{converge1}. We leave the details to interested readers. The proof of Theorem \ref{converge2-2} is thus complete.
\end{proof}


\SE{Convergence in a multistable case: Proof of Theorem \ref{converge3}}
This section is devoted to the proof of Theorem \ref{converge3}. Throughout this section, let Assumption \ref{multi-stable} hold, and let $u(t,x)$ be the solution of \eqref{E} with $u_0$ satisfying (H3). 
From Assumption \ref{multi-stable} and its followed discussion, it is known that there exists a minimal propagating terrace $((p_i)_{0\leq i\leq N},(U_i,c_i)_{1\leq i\leq N})$ connecting $0$ to $p$, and each $p_i$ is linearly stable, i.e., 
\begin{equation}\label{define-mui}
\mu_i:= -\frac{1}{T}\int_{0}^{T}\partial_u f(t,p_i(t)) dt>0.
\end{equation}


\subsection{Global convergence to minimal propagating terrace} 
In this subsection, we prove statement (i) of Theorem \ref{converge3}, that is, the solution $u(t,x)$ converges to the minimal propagating terrace as $t\to\infty$ in $L^{\infty}(\R)$. The key step is to show that, up to some error terms with exponential decay, $u(t,x)$ can be bounded from above and below by solutions with Heaviside type initial data for all large times. Let us first begin with the following observation on the behavior of $u(t,x)$ at a certain time.

\begin{lemma}\label{front-like-initial}
 For any $\epsilon>0$, there exist a positive number $a_0=a_0(\epsilon,u_0)$ and a positive integer $k_0=k_0(\epsilon,u_0)$ such that
\begin{equation}\label{vaguely}
\hat{u}(T,x;-a_0)-\epsilon \leq u(k_0T,x) \leq \hat{u}(T,x;a_0)+\epsilon \,\,\hbox{ for } \,x\in\R,
\end{equation}
where $\hat{u}(t,x;\pm a_0)$ are the solutions of \eqref{E} with initial functions $p(0)H(\pm a_0-x)$. 
\end{lemma}

\begin{proof}
We only show the second inequality in \eqref{vaguely}, as the first one can be proved analogously.   

Let us first set a few notations. Since $\sup_{x\in\R} u_0(x) \in I_+$ and $\limsup_{x\to\infty} u_0(x) \in I_-$, there exist real numbers $h_{\pm}\in  I_{\pm}$ such that 
\begin{equation}\label{hi-geq-sup}
h_+ > \sup_{x\in\R}u_0(x)   \quad\hbox{and}\quad h_- > \limsup_{x\to\infty} u_0(x).  
\end{equation}
Let $H_{\pm}(t)$ be the solutions of \eqref{ode-initial} with initial values $h_{\pm}$. It then follows that 
\begin{equation}\label{hti-to-p}
\lim_{k\to\infty} H_+(t+kT) =p(t)\quad\hbox{and}\quad \lim_{k\to\infty} H_-(t+kT) = 0
\end{equation}
locally uniformly in $t\in\R$.  Furthermore, since the function $f(t,u)$ is of class $C^{1,1}$ in $u$ uniformly for $t\in\R$, there exists $L>0$ such that 
\begin{equation}\label{lip-constant}
|\partial_u f(t,u_1)-\partial_u f(t,u_2) | \leq L |u_1-u_2| \,\hbox{ for  all } \, t\in\R,\, u_1, \,u_2 \in [0,\infty).
\end{equation}

We now construct a super-solution of \eqref{E}. Set $\gamma(x)=\frac{1}{2} (1+\tanh \frac{x}{2})$ for $x\in\R$. Thanks to \eqref{hi-geq-sup}, one finds some large number $C_1>0$ such that 
$$u_0(x) \leq h_+ (1-\gamma(x-C_1)) + h_-\gamma(x-C_1)  \,\,\hbox{ for } \,\, x\in\R. $$
Define 
$$W(t,x)= H_+(t) (1-\gamma(x-C_1-C_2t)) + H_-(t)\gamma(x-C_1-C_2t) $$
for $t\geq 0$, $x\in\R$, where 
$$C_2=1+L\sup_{t\in [0,\infty)}|H_+(t)-H_-(t)|.$$
It is clear that $u_0(x)\leq W(0,x)$ for $x\in\R$. Next, we check that 
$$\mathcal{L}W: =W_t- W_{xx} -f(t,W)\geq 0\,\,\hbox{ for } \,t>0,\,x\in\R.  $$
Observe from \eqref{lip-constant} that for any $t>0$, $x\in\R$, 
\begin{equation*}
\begin{split}
 & (1-\gamma)f(t,H_+(t)) +  \gamma f(t,H_-(t))- f(t,W)  \vspace{3pt}\\ 
 \smallskip= \,\,& \gamma(1-\gamma)(H_+(t) -H_-(t)) [ \partial_uf(t,\theta_1H_+ +(1-\theta_1)W) - \partial_uf(t,\theta_2H_- +(1-\theta_2)W) ]\vspace{3pt}\\ 
  \smallskip \geq\,\,  & -L \gamma(1-\gamma)(H_+(t) -H_-(t))^2
\end{split}
\end{equation*}
for some $\theta_1=\theta_1(t,x)$, $\theta_2=\theta_2(t,x) \in [0,1]$. Then, it is straightforward to compute that 
$$ \mathcal{L}W =   (C_2\gamma'+\gamma'')(H_+(t) -H_-(t))-L \gamma(1-\gamma)(H_+(t) -H_-(t))^2.$$
Notice that $\gamma'=\gamma(1-\gamma)$ and $\gamma''=\gamma'(1-2\gamma)$. Owing to the definition of $C_2$, we obtain $\mathcal{L}W\geq 0$ for $t>0$, $x\in\R$. Thus, $W$ is a super-solution of \eqref{E}.
By the comparison principle, we have
\begin{equation}\label{u-leq-w}
u(t,x)\leq W(t,x)\,\hbox{ for } \,t\geq 0,\, x\in\R. 
\end{equation}

For any $\epsilon>0$, by using \eqref{hti-to-p},  we find some $k_0\in\N$ such that
$$\sup_{x\in\R} W(k_0T,x) < p(0)+\epsilon  \quad\hbox{and}\quad \limsup_{x\to\infty} W(k_0T,x) < \epsilon.$$
Since $\hat{u}(T,x;a)$ is decreasing in $x\in\R$, and since 
$$\lim_{a\to-\infty} \hat{u}(T,x;a)=0,\quad \lim_{a\to\infty} \hat{u}(T,x;a)=p(0)\,\,\hbox{ locally uniformly in } \,x\in\R,$$ 
it follows that there exists $a_0>0$ large enough such that
$$W(k_0T,x) \leq  \hat{u}(T,x;a_0)+\epsilon \,\hbox{ for } \, x\in\R. $$
Combining this with \eqref{u-leq-w}, we immediately obtain the second inequality of \eqref{vaguely}. The proof of Lemma \ref{front-like-initial} is thus complete.
\end{proof}

We now show the following key lemma. 

\begin{lemma}\label{resemble}
There exist positive constants $\epsilon_0$, $K_0$ and $\beta_0$ such that if for some $a\in\R$ and $\epsilon\in (0,\epsilon_0]$, there holds
\begin{equation}\label{com-u_0-hatu}
u_0(\cdot) \leq \hat{u}(T,\cdot;a)+\epsilon,
\end{equation} 
then for all $t\geq 0$, 
\begin{equation}\label{u-hatu-super}
u(t,\cdot)\leq \hat{u}(t+T,\cdot-K_0\epsilon;a)+K_0\epsilon\me^{-\beta_0 t}.
\end{equation}
Analogously, if $u_0(\cdot) \geq \hat{u}(T,\cdot;a)-\epsilon$ for some $a\in\R$ and $\epsilon\in (0,\epsilon_0]$, then for all $t\geq 0$,
\begin{equation}\label{u-hatu-sub}
u(t,\cdot)\geq \hat{u}(t+T,\cdot+K_0\epsilon;a)-K_0\epsilon\me^{-\beta_0 t}.
\end{equation}
\end{lemma}

\begin{proof}
Without loss of generality, we assume that $a=0$, and for convenience, we write $\hat{u}(t,x)$ instead of $\hat{u}(t,x;a)$. Let $((p_i)_{0\leq i\leq N},(U_i,c_i)_{1\leq i\leq N})$ be the minimal propagating terrace connecting $0$ to $p$. Then there exist $C^1([0,\infty))$ functions $(\eta_i(t))_{1\leq i\leq N}$ such that all the conclusions of Theorem \ref{converge1} hold true. For any $\delta\in (0,1)$ and $C>0$, let us set
$$ {\rm I}_{\delta}(t):=\bigcup_{i=0}^{N}{\rm I}_{\delta}^i(t) := \bigcup_{i=0}^{N} \, [p_i(t)-\delta,p_i(t)+\delta] \,\,\hbox{ for }  \,t\in\R,$$
and 
$$\Pi_C(t):= \bigcup_{i=1}^{N} \, \Pi _C^i(t):=  \bigcup_{i=1}^{N}  \, [c_it+\eta_i(t)-C, c_it+\eta_i(t)+C ] \,\, \hbox{ for } \,t\geq 0.    $$

To prove \eqref{u-hatu-super}, we will use $\hat{u}(t,x)$ to construct a suitable super-solution of the solution $u(t,x)$.   For clarity, we proceed with 3 steps.

{\bf Step 1:} we show some estimates of $\hat{u}(t,x)$.

For each $i=0,\cdots,N$, let $\mu_i$ be the positive constant defined in \eqref{define-mui}.  By the $C^1$-regularity and the periodicity of $f$, there exists a small positive constant $\delta_0$ such that 
\begin{equation}\label{es-partialu}
|\partial_u f(t,v)-\partial_u f(t,p_i(t))| \leq \frac{\mu_i}{2} \,\,\hbox{ for all } \,v\in {\rm I}_{\delta_0}^{i}(t),\,t\in\R.
\end{equation}
We choose a large constant $C_1>0$ such that 
\begin{equation*}
U_i(t, c_it\pm C_1) \subset {\rm I}_{\delta_0/3}(t) \,\,\hbox{ for } \, i\in \{1,\cdots,\,N\},\,t\in\R. 
\end{equation*}
Since $U_i(t,x)$ is decreasing in $x\in\R$, we have 
\begin{equation}\label{es-Ui-out}
U_i (t, \R \setminus [c_it-C_1,c_it+C_1]) \subset {\rm I}_{\delta_0/3}(t) \,\, \hbox{ for } \, i\in \{1,\cdots,\,N\},\,t\in\R,
\end{equation}
and we can find a positive constant $\rho_1>0$ such that 
\begin{equation}\label{pu-leq-rho1}
\partial_x U_i(t,x) \leq -2\rho_1\,\,\hbox{ for } \, i\in \{1,\cdots,\,N\},\, x\in [c_it-C_1-2,c_it+C_1+2],\, t\in\R.
\end{equation}

Next, by using Theorem \ref{converge1} and \eqref{es-Ui-out}, we can find some $T_1>0$ sufficiently large such that 
\begin{equation}\label{out-IIC}
c_{i}t+\eta_{i}(t)+C_1 <   c_{i+1}t+\eta_{i+1}(t)-C_1-2 \,\,\hbox{ for } \,\, t\geq T_1, \, i\in \{1,\cdots,\,N-1\}, 
\end{equation}
and that
\begin{equation}\label{hatu-out-C}
\hat{u}(t,\R\setminus \Pi_{C_1}(t)) \subset {\rm I}_{\delta_0/2}(t) \,\,\hbox{ for } \, t\geq T_1.
\end{equation}  
Moreover, by standard parabolic estimates, we have 
\begin{equation*}
\max_{x\in \Pi_{C_1+2}^i(t) } \left| \partial_x \hat{u}(t,x)-\partial_xU_i(t, x-\eta_i(t))  \right| \to 0 \,\,\hbox{ as }\, t\to\infty \,\,\hbox{ for } \, i\in \{1,\cdots,\,N\}.
\end{equation*}
This together with \eqref{pu-leq-rho1} implies that there exists $T_2>T_1$ such that
\begin{equation}\label{es-partial-hatux}
\partial_x  \hat{u}(t,x) \leq \max_{1\leq i\leq N} \left\{ \partial_x U_i(t,x-\eta_i(t))     \right\}+\rho_1\leq -\rho_1 \,\,\hbox{ for }\, x\in \Pi_{C_1+2}(t),\,t\geq T_2.
\end{equation}

Note that the following convergences  
$$\lim_{x\to-\infty} \hat{u}(t,x) -p(t) \to 0 \quad\hbox{and} \quad \lim_{x\to\infty} \hat{u}(t,x) \to 0   $$
hold locally uniformly in $t\in[0,\infty)$. 
There exists some constant $C_2>0$ such that
\begin{equation}\label{hatu-out-t2}
\left\{\baa{ll}
\smallskip  \hat{u}(t+T,x)\in  {\rm I} _{\delta_0/2}^0(t) & \hbox{ for }\, x\leq -C_2,\,\, 0\leq t\leq  T_2, \vspace{3pt}\\
\hat{u}(t+T,x)\in {\rm I} _{\delta_0/2}^N(t)& \hbox{ for }\, x\geq C_2,\,\, 0\leq t\leq  T_2. \eaa\right.
\end{equation}
Replacing $C_2$ by some larger constant if necessary, we may assume that 
\begin{equation}\label{requre-C2}
C_2\geq \max_{t\in[0,T]} \left\{ -c_1t-\eta_1(t)+C_1,\,  c_Nt+\eta_N(t)+C_1 \right\}.
\end{equation}
Since $\hat{u}(t,x) \in C^1((0,\infty)\times\R)$ and it is decreasing in $x$, there exists some constant $\rho_2>0$ such that
\begin{equation}\label{choice-rho2}
 \min\left\{-\partial_x\hat{u}(t+T,x):x\in [-C_2-2,C_2+2],\,t\in[0,T_2]   \right\} \geq \rho_2.
\end{equation}

\vskip 3pt
{\bf Step 2:} we introduce some notations and present our super-solution.

Let $(\zeta_i(t,x))_{0\leq i\leq N}$  be a sequence of $C^2([0,\infty)\times \R)$ functions satisfying 
\begin{equation}\label{zeta-initial}
\sum_{i=0}^{N} \zeta_i (0,x)\geq 1\,\,\hbox{ for } \,x\in\R,
\end{equation} 
\begin{equation}\label{defi-zeta-i}
\left.\baa{ll}
\medskip\medskip\medskip
\zeta_0(t,x)\!\!\!\!&=\left\{\baa{l}
\smallskip 1,\, \,\hbox{ if } \,\,x\in (-\infty,\,  c_1t+\eta_1(t)-C_1],\,t\in [0,\infty),\vspace{3pt}\\
0,\, \,\hbox{ if }\,\,   x\in [c_1t+\eta_1(t)-C_1+2,\,\infty),\,t\in [0,\infty),\eaa\right.\\
 \medskip \medskip
\zeta_i(t,x)\!\!\!\!&=\left\{\baa{l}
\smallskip 1,\, \,\hbox{ if }\,\,  x\in [c_{i}t+\eta_{i}(t)+C_1, \, c_{i+1}t+\eta_{i+1}(t)-C_1],\,t\in [T_2,\infty), \vspace{3pt}\\
\smallskip  0,\, \,\hbox{ if } \,\,x \in \R\setminus [c_{i}t+\eta_{i}(t)+C_1-2,\, c_{i+1}t+\eta_{i+1}(t)-C_1+2],\,t\in [T_2,\infty),\vspace{3pt}\\
  0,\, \,\hbox{ if } \,\,x \in \R\setminus [-C_2,C_2],\,t\in [0,T_2),\vspace{3pt}\\
 \eaa\right. \\
 \medskip \medskip \medskip
 & \hbox{ for } \, i\in \{1,\,\cdots,\,N-1\},\vspace{3pt}\\
 \zeta_N(t,x)\!\!\!\!&=\left\{\baa{l}
\smallskip 0,\, \,\hbox{ if } \,\,x\in (-\infty,\,  c_Nt+\eta_N(t)+C_1-2],\, t\in [0,\infty),\vspace{3pt}\\
1,\, \,\hbox{ if }\,\, x\in [ c_Nt+\eta_N(t)+C_1,\,\infty),\,t\in [0,\infty), \eaa\right. \eaa \right\}
\end{equation}
and
\begin{equation}\label{defi-zeta-esti}
0\leq \zeta_i\leq 1, \,\, |\partial_t \zeta_i|\leq \max_{1\leq j\leq N}|c_j|+1 \footnote{Notice that the functions  $(\eta_i)_{1\leq i\leq N}$ are not unique, since for any $C^1([0,\infty))$ functions $(\bar{\eta}_i)_{1\leq i\leq N}$ satisfying $\bar{\eta}_i(t)\to0 $ as $t\to\infty$ for each  $1\leq i\leq N$, the sequence $(\bar{\eta}_i+\eta_i)_{1\leq i\leq N}$ is also associated with $\hat{u}(t,x)$ satisfying Theorem \ref{converge1}.  Therefore, we may assume without loss of generality that $|\eta'_i(t)| \leq \frac{1}{2}$ for $t\geq 0$, $i= 1,\cdots,\,N$. This allows us to choose functions $(\zeta_i)_{0\leq i\leq N}$ satisfying \eqref{defi-zeta-i} and $|\partial_t \zeta_i|\leq  \max_{1\leq j\leq N}|c_j|+1$.} , \,\, |\partial_x \zeta_i| \leq 1,\,\, |\partial_{xx} \zeta_i|\leq 1 \end{equation}
for $(t,x)\in [0,\infty)\times \R$,  $i\in \{0,\,1,\,\cdots,\,N\}$. It is easily seen from the above that 
$$\zeta_i(t,x)\zeta_j(t,x)=0 \,\, \hbox{ for } \, x\in\R,\, t\geq T_2, \, \hbox{ whenever }\, i\neq j,  $$
and 
$$\sum_{i=0}^{N} \zeta_{i}(t,x)=1\,  \,\, \hbox{ for } \, x\in\R\setminus \Pi_{C_1}(t) ,\, t\geq T_2.$$

Define
$$A(t,x)= \sum_{i=0}^N \zeta_i(t,x)b_i(t)\,\,\hbox{ for }  \,\, t\geq 0,\,x\in\R,$$
and
$$B(t)=\int_{0}^{t} \max_{0\leq i\leq N}\{ b_i(\tau) \} d\tau\, \hbox{ for }  \,\, t\geq 0,$$
where for each $i\in \{0,\,1,\,\cdots,\,N\}$, the function $b_i$ is given by
\begin{equation}\label{define-bi}
 b_i(t)={\rm exp}\left(\frac{\mu_it}{2}+\int_{0}^{t}\partial_u f(\tau,p_i(\tau)) d\tau   \right)\,\,\hbox{ for } \,t\geq 0.
\end{equation}
Note that for each $i\in \{0,\,1,\,\cdots,\,N\}$, 
\begin{equation}\label{bound-bi}
0 \leq b_i(t)\leq  M{\rm exp} \left(-\frac{\mu_it}{2}  \right)  \,\,\hbox{ for }\,\, t\geq 0,
\end{equation}
where 
$$ M= \sup_{t\in[0,T],\,0\leq i\leq N} {\rm exp}\left(\mu_it+\int_{0}^{t}\partial_u f(\tau,p_i(\tau)) d\tau   \right).$$
This implies that $b_i(t)$ and $A(t,x)$ converge exponentially to $0$ as $t\to\infty$, and $B(t)$ is uniformly bounded in $t\geq 0$.  Set 
$$K=\frac{\sum_{i=0}^{N} ( \max_{1\leq j\leq N}|c_j|+\mu_i/2+2+2\| \partial_u f\| ) }{\min\{\rho_1,\,\rho_2 \} }$$
and 
$$\epsilon_0=\min\left\{ \frac{\delta_0}{2M},\,\,\frac{1}{KB(\infty)} \right\}, $$
where $\| \partial_u f\|=\max\{ |\partial_u f(t,u)|:\,u\in [-1,\,p(t)+1],\, t\in\R\}$ and $B(\infty)=\lim_{t\to\infty} B(t)$. 
Let $\epsilon \in (0,\epsilon_0]$ be an arbitrary constant. We will show that  
$$V(t,x):=\hat{u}(t+T,x-\epsilon K B(t))+\epsilon A(t,x) \hbox{ for } t\geq 0,\, x\in\R.$$
is a super-solution of \eqref{E}.

\vskip 3pt
{\bf Step 3:} we check that $V(t,x)$ is a super-solution.

When $t=0$,  it follows directly from \eqref{com-u_0-hatu} and \eqref{zeta-initial} that 
$$u_0(x) \leq  \hat{u}(T,x)+\epsilon \leq V(0,x)\, \hbox{ for }\,\, x\in\R. $$
When $t>0$, we calculate that 
\begin{equation*}
\begin{split}
\mathcal{L}V:
\smallskip & =V_t-V_{xx}-f(t,V) \vspace{3pt}\\
 &=-\epsilon K B'(t)\partial_x\hat{u} +\epsilon (A_t-A_{xx}-\partial_u f(t,\hat{u}+\epsilon \theta A)A) 
\end{split}
\end{equation*}
for some $\theta=\theta(t,x)\in [0,1]$. Now we claim that $\mathcal{L}V \geq 0$ for all $x\in\R$,\, $t>0$. We consider the following four cases. 

{\it  Case 1:}  $x\in \R\setminus \Pi_{C_1+1}(t)$,  $t\geq T_2$.

For each $i\in\{0,\cdots,\,N\}$, define 
$$S_i= \{(t,x): \, x \in \R\setminus \Pi_{C_1+1}(t),\, t\geq T_2, \,\zeta_i(t,x)=1 \}. $$
One easily checks from \eqref{out-IIC} and \eqref{defi-zeta-i} that
$$S_i \neq \emptyset \, \hbox{ for each } \,i\in\{0,\cdots,\,N\}, \,\,\hbox{ and }\,\,   S_i\cap S_j =\emptyset  \,\,\hbox{ whenever } \,\,i\neq j,$$
and that
\begin{equation}\label{sum-Si}
\bigcup_{i=0}^{N} S_i = \left\{(t,x):\, x \in \R\setminus \Pi_{C_1+1}(t),\, t\geq T_2 \right\}.
\end{equation}
Then for any fixed $i_0\in \{0,\cdots,\,N\}$,  we compute on the set $S_{i_0}$ and obtain that  
\begin{equation*}
\begin{split}
\mathcal{L}V
\smallskip &\,  \geq \epsilon(A_t-A_{xx}-\partial_u f(t,\hat{u}+\epsilon \theta A)A) \vspace{3pt}\\
\smallskip &\, =\epsilon( b_{i_0}'(t)- \partial_u f(t,\hat{u}+\epsilon \theta b_{i_0})b_{i_0}) \vspace{3pt}\\
&\, = \epsilon b_{i_0} \left(\frac{\mu_{i_0}}{2}+\partial_u f(t,p_{i_0})-\partial_u f(t,\hat{u}+\epsilon \theta b_{i_0})\right),
\end{split}
\end{equation*}
where the first inequality follows from the monotonicity of $\hat{u}$ in $x$. 
Notice from the choice of $\epsilon_0$ that
\begin{equation}\label{es-K0-Bt}
0\leq \epsilon KB(t)\leq 1 \,\, \hbox{ for all  }\,\, t\geq 0,
\end{equation}
and $\epsilon \theta b_{i_0}(t) \leq  \epsilon_0 M  \leq \delta_0/2$ for all $t\geq 0$. 
This, together with \eqref{out-IIC} and  \eqref{hatu-out-C}, implies that  
$$ \hat{u}(t+T, x-K\epsilon B(t))+\epsilon \theta b_{i_0}(t)  \in  {\rm I}_{\delta_0}^{i_0}(t) \,\, \hbox{ for } \,\, (t,x)\in S_{i_0}.$$
Therefore, by using \eqref{es-partialu}, we obtain $\mathcal{L}V \geq 0$ for $(t,x)\in S_{i_0}$. Due to \eqref{sum-Si} and the arbitrariness of $i_0\in \{0,\cdots,\,N\}$, we have $\mathcal{L}V \geq 0$ for $x\in \R\setminus \Pi_{C_1+1}(t)$,  $t\geq T_2$. 

\vskip 3pt

{\it Case 2:}  $x \in  \Pi_{C_1+1}(t)$,  $t\geq T_2$.

In this case, it follows from \eqref{es-K0-Bt} that $x-\epsilon KB(t) \in \Pi_{C_1+2}(t)$. By using \eqref{es-partial-hatux}, we have  $\partial_x \hat{u}(t+T, x-K\epsilon B(t)) \leq -\rho_1$. On the other hand, direct calculation yields that
\begin{equation*}
\begin{split}
& \left |A_t-A_{xx}-\partial_u f(t,\hat{u}+\epsilon \theta A)A \right|\vspace{3pt}\\ 
\smallskip= &\,\,\left | \sum_{i=0}^{N} \left[(\zeta_i)_tb_i- (\zeta_i)_{xx}b_i- \partial_u f(t,\hat{u}+\epsilon \theta A)\zeta_ib_i+ \left(\frac{\mu_i}{2}+ \partial_u f(t,p_i)\right)\zeta_ib_i \right]  \right|
 \vspace{3pt}\\
\smallskip \leq& \,\, \max_{0\leq i\leq N}\{ b_i \} \sum_{i=0}^{N} \left[ |(\zeta_i)_t|+ | (\zeta_i)_{xx}| + 2\|\partial_u f \|\zeta_i +\frac{\mu_i}{2} \zeta_i  \right]  \vspace{3pt}\\
\smallskip \leq  &\,\, \max_{0\leq i\leq N}\{ b_i \} \sum_{i=0}^{N} \left[ \max_{1\leq j\leq N}|c_j|+\frac{\mu_i}{2}+2+ 2\|\partial_u f \| \right]
\end{split}
\end{equation*}
for all $x\in\R$, $t\geq 0$, where the last inequality follows from \eqref{defi-zeta-esti}.  Combining the above, for $x \in  \Pi_{C_1+1}(t)$,  $t\geq T_2$,  we obtain
\begin{equation*}
\begin{split}
\mathcal{L}V &\,\geq  \epsilon K \rho_1 B'(t)- \epsilon  \max_{0\leq i\leq N}\{ b_i \} \sum_{i=0}^{N} \left[\max_{1\leq j\leq N}|c_j|+\frac{\mu_i}{2}+2+ 2\|\partial_u f \| \right] \vspace{3pt}\\ 
\smallskip\,&\, =\epsilon  \max_{0\leq i\leq N}\{ b_i \} \left( K\rho_1 -\sum_{i=0}^{N} \left[ \max_{1\leq j\leq N}|c_j|+\frac{\mu_i}{2}+2+ 2\|\partial_u f \| \right]\right).
\end{split}
\end{equation*}
Hence, by the choice of $K$, it follows that $\mathcal{L}V \geq 0$ for $x \in  \Pi_{C_1+1}(t)$,  $t\geq T_2$. 

\vskip 3pt

{\it  Case 3:}  $|x|\geq C_2+1$,  $t \in(0, T_2)$.

In this case, from \eqref{requre-C2} and \eqref{defi-zeta-i}, we observe that  
\begin{equation*}
\left\{\baa{ll}
\smallskip  A(t,x)\equiv b_0(t) & \hbox{ for }\, x\leq -C_2-1,\,\, 0< t< T_2, \vspace{3pt}\\
A(t,x)\equiv b_N(t) & \hbox{ for }\, x\geq C_2+1,\,\, 0< t< T_2. \eaa\right.
\end{equation*}
Due to \eqref{es-K0-Bt}, we have $|x-K\epsilon B(t)| \geq C_2$. It then follows from  \eqref{hatu-out-t2} that 
\begin{equation*}
\left\{\baa{ll}
\smallskip  \hat{u}(t+T, x-K\epsilon B(t))+\epsilon \theta b_0(t) \in {\rm I}_{\delta_0}^{0}(t) & \hbox{ for }\, x\leq -C_2-1,\,\, 0< t< T_2, \vspace{3pt}\\
 \hat{u}(t+T, x-K\epsilon B(t))+\epsilon \theta b_N(t) \in {\rm I}_{\delta_0}^{N}(t) & \hbox{ for }\, x\geq C_2+1,\,\, 0< t<T_2. \eaa\right.
\end{equation*}
Thus, similar calculations to those used in the proof of {\it Case 1} imply that $\mathcal{L}V \geq 0$ for  $|x|\geq C_2+1$, $t \in(0, T_2)$. 

\vskip 3pt
{\it  Case 4:}  $|x|\leq C_2+1$,  $t \in(0, T_2)$.

In this case, we have $|x-K\epsilon B(t)|\leq C_2+2$, whence by \eqref{choice-rho2}, there holds
$\partial \hat{u}(t+T,x-K\epsilon B(t)) \leq -\rho_2$.
Then, following the lines of the proof of {\it Case 2}, we obtain that for $|x|\leq C_2+1$,  $t \in(0, T_2)$, 
\begin{equation*}
\mathcal{L}V \geq  \epsilon  \max_{0\leq i\leq N}\{ b_i \} \left( K\rho_2 -\sum_{i=0}^{N} \left[ \max_{1\leq j\leq N}|c_j|+\frac{\mu_i}{2}+2+ 2\|\partial_u f \| \right]
 \right) \geq 0.
\end{equation*}

In all cases, we have $\mathcal{L}V \geq 0$, and hence, $V(t,x)$ is a super-solution of \eqref{E}. Then the comparison principle implies that
$$u(t,x)\leq V(t,x)\,\,\hbox{ for }\,\,x\in\R,\,t\geq 0. $$
Taking 
$$K_0=\max\{K B(\infty),(N+1)M\} \quad\hbox{and}\quad \beta_0=\frac{1}{2} \min_{0\leq i\leq N} \{\mu_i\},$$
we immediately obtain \eqref{u-hatu-super}. The proof of \eqref{u-hatu-sub} is analogous and we omit the details.
\end{proof}

Let $\epsilon_0$, $K_0$ and $\beta_0$ be the positive constants provided by Lemma \ref{resemble}. It follows from Lemmas \ref{front-like-initial} and \ref{resemble} that, there exist a positive integer $k_0$ and a positive number $a_0$ such that  
\begin{equation}\label{trap-hatu-exp}
\begin{split}
\smallskip \hat{u}(t+T,x+K_0\epsilon_0;-a_0)-K_0\epsilon_0\me^{-\beta_0t}  \leq &\,  u(t+k_0T,x)  \vspace{3pt}\\
 \leq &\,\hat{u}(t+T,x+K_0\epsilon_0;a_0)+K_0\epsilon_0\me^{-\beta_0t}
\end{split}
\end{equation}
for all $(t,x)\in [0,\infty)\times\R$.  In order to further show that $u(t,x)$ approaches the minimal propagating terrace as $t\to\infty$,  we need the following Liouville type result.

\begin{lemma}\label{liouville}
Let $W(t,x)$ be an entire solution of \eqref{equation} satisfying that for some $\xi_-<\xi_+$ and some $i\in \{1,\,\cdots,\,N\}$,  
\begin{equation*}
U_i(t,x-\xi_-) \leq W(t,x)\leq U_i(t,x-\xi_+) \,\,\hbox{ for } \,t\in\R,\,x\in\R.
\end{equation*}
Then $W\equiv U_i$ up to a spatial shift. 
\end{lemma}

\begin{proof}
This lemma follows directly from \cite[Lemma 4.3]{con} by a sliding method. Let us mention that it can also be proved by a dynamical system approach used in an earlier work \cite{om} (see Corollary 8.3 and Proposition B.2 in Appendix 2 of \cite{om}). 
\end{proof}

\begin{proof}[Proof of statement (i) of Theorem \ref{converge3}]
Let $w$ be an arbitrary element of $\Omega(u)$. Because of \eqref{trap-hatu-exp}, the same arguments as those used in showing Lemma \ref{parallel-terrace} imply that either $w\equiv p_i$ for some $0\leq i\leq N$, or there exist some integer $1\leq i\leq N$ and some $\xi_0\in\R$ such that
\begin{equation*}
U_i(t, x+\xi_0+a_0) \leq w(t,x) \leq U_i(t, x+\xi_0-a_0) \,\,\hbox{ for } t\in\R,\,x\in\R.
\end{equation*}
Moreover, if the later case occurs, then it follows directly from Lemma \ref{liouville} that $w\equiv U_i$ up to a spatial shift.  Thus, we have
\begin{equation*}
\Omega(u)= \left\{ U_i(\cdot,\cdot+\xi): \xi\in\R,\, 1\leq i\leq N \right\} \cup  \left\{ p_i: 0\leq i\leq N \right\}. 
\end{equation*}
The remaining proof is similar to that of Theorem \ref{converge2-2}, therefore we do not repeat the details here. 
\end{proof}


\subsection{Exponential convergence to minimal propagating terrace} 
The aim of this subsection is to prove statement (ii) of Theorem \ref{converge3}, that is, under the additional assumption that $c_1<c_2<\cdots<c_N$, the drift functions $(\eta_i(t))_{1\leq i\leq N}$ are convergent, and the solution $u(t,x)$ converges to the minimal terrace as $t\to\infty$ with an exponential rate. The strategy of the proof, which is inspired by \cite{po3, rtv} for autonomous equations/systems, can be described as follows.  Let $(\bar{c}_i)_{0\leq i\leq N}$ be a sequence of real numbers given by
\begin{equation}\label{define-barci}
\bar{c}_0:=c_1-1, \quad \bar{c}_i:=\frac{c_i+c_{i+1}}{2}\,\, \hbox{ for } \,i=1,\cdots, N-1,\quad \bar{c}_N:=c_N+1.
\end{equation}
Since $c_i$, $i=1,\cdots,N$, are mutually distinct, it is clear that $\bar{c}_{i-1}<c_i<\bar{c}_i$ for each $i=1,\cdots,N$.
We will show that, as $t\to\infty$, $u(t,x)$ approaches a spatial shift of the periodic traveling wave $U_i$ uniformly in $\bar{c}_{i-1}t\leq x\leq \bar{c}_{i}t$, and the approach is exponentially fast. 
In the remaining regions, i.e., $x\leq \bar{c}_{0}t$ and $x\geq \bar{c}_{N}t$, we will prove that $u(t,x)$ converges exponentially to $p(t)$ and $0$, respectively. 

We will proceed by a sequence of lemmas. The first lemma is a simple extension of the well known Fife-McLeod type super/sub-solutions result for bistable equations (see \cite{fm,abc}). To state our lemma, we need a few more notations.  Let $\zeta(x)$ be any $C^2(\R)$ function satisfying 
\begin{equation}\label{smooth-zeta}
\zeta(x)=0 \,\hbox{ in } [3,\infty),\,\,\,\, \zeta(x)=1\,\hbox{ in } (-\infty,0],\,\,\,\, -1\leq \zeta'(x)\leq 0 \hbox{ and } |\zeta''(x)|\leq 1\, \hbox{ in } \R.
\end{equation}
For each $i=1,\cdots,N$, define
\begin{equation}\label{define-Ai}
A_i(t,x)=\zeta(x)b_{i-1}(t)+(1-\zeta(x))b_i(t) \hbox{ for } \,t\geq 0,\,x\in\R,
\end{equation}
where $(b_i)_{0\leq i\leq N}$ are the functions defined in \eqref{define-bi}.

\begin{lemma}\label{tw-super-sub}
Let $i\in \{ 1,\cdots,N\}$ be any fixed integer. If $c>c_i$, then there exists $\epsilon_0>0$ such that for every $\epsilon\in (0,\epsilon_0]$ and $K\in\R$,
$$ \bar{W}_i(t,x):= U_i(t,x+c_it-ct+K)+\epsilon A_i(t,x-ct)$$
satisfies 
$$\partial_t \bar{W}_i \geq \partial_{xx}  \bar{W}_i +f(t,\bar{W}_i)\,\,\hbox{ for }\, x\in\R,\,t>0. $$
Similarly, if $c<c_i$, then there exists $\epsilon_0>0$ such that for every $\epsilon\in (0,\epsilon_0]$ and $K\in\R$,
$$ \underbar{W}_i(t,x):= U_i(t,x+c_it-ct+K)-\epsilon A_i(t,x-ct)$$
satisfies 
$$\partial_t \underbar{W}_i \leq \partial_{xx}  \underbar{W}_i +f(t,\underline{W}_i)\,\,\hbox{ for }\, x\in\R,\,t>0. $$
\end{lemma}

\begin{proof}
This lemma can be proved by slightly modifying the arguments used in \cite[Lemma 3.2]{abc}. For the sake of completeness, and also for the convenience of later applications, we include the details below.
We only give the proof in the case $c>c_i$, since the proof for the other case is identical.

Remember that $\partial_t U_i = \partial_{xx}  U_i +f(t,U_i)$ in $(t,x)\in\R^2$. Direct calculation gives that for $t>0$, $x\in\R$, 
\begin{equation*}
\begin{split}
\mathcal{L} \bar{W}_i:
\smallskip & =\partial_t \bar{W}_i - \partial_{xx}  \bar{W}_i -f(t,\bar{W}_i) \vspace{3pt}\\
 &=(c_i-c)\partial_x U_i +\epsilon (\partial_t A_i-\partial_{xx}A_i-c\partial_xA_i - \partial_u f(t,U_i+\epsilon \theta A_i)A_i) 
\end{split}
\end{equation*}
for some $\theta=\theta(t,x)\in [0,1]$. Let $(\mu_i)_{0\leq i\leq N}$  be the positive constants given in \eqref{define-mui}, and let $\delta_0>0$, $C_1>0$, $\rho_1>0$ and $M>0$ be the real numbers such that \eqref{es-partialu}, \eqref{es-Ui-out}, \eqref{pu-leq-rho1} and \eqref{bound-bi} hold. Set
$$\epsilon_0=\min\left\{ \frac{\delta_0}{ 2M},\,\,\frac{2(c-c_i)\rho_1}{M \left( \mu_i/2+ \mu_{i-1} /2+1+|c|+2\|\partial_u f\| \right)} \right\},$$
where $\| \partial_u f\|=\max\{ |\partial_u f(t,u)|:\,u\in [p_i(t)-1,\,p_{i-1}(t)+1],\, t\in\R\}$. 
We will show that, for any $0<\epsilon \leq \epsilon_0$, $\mathcal{L} \bar{W}_i \geq 0$ for $t>0$, $x\in\R$. 

Let us first check $\mathcal{L} \bar{W}_i \geq 0$ when $x-ct+K \geq C_1$. Replacing $C_1$ by some larger constant if necessary, we may assume that $C_1\geq K+3$.
Then we have $\zeta(x-ct)\equiv 0$, whence $A_i\equiv b_i$ and $\partial_x A_i=\partial_{xx} A_i=0$. Since $\partial_x U_i <0$, it follows that
\begin{equation*}
\begin{split}
\mathcal{L} \bar{W}_i & \geq \epsilon (\partial_t A_i- \partial_u f(t,U_i+\epsilon \theta A_i)A_i) 
  \vspace{3pt}\\
 &=\epsilon b_i\left(\frac{\mu_i}{2}+ \partial_u f(t,p_i(t)) - \partial_u f(t,U_i+\epsilon \theta b_i)\right).
\end{split}
\end{equation*}
By \eqref{es-partialu}, \eqref{es-Ui-out} and the fact that $0\leq \epsilon \theta b_i \leq \delta_0/2$, we obtain $\mathcal{L} \bar{W}_i \geq 0$ when $x-ct+K \geq  C_1$. In a similar way, one can conclude that $\mathcal{L} \bar{W}_i \geq 0$ when $x-ct+K \leq -C_1$. 

For the remaining values of $x$ and $t$, i.e.,  $-C_1\leq x-ct+K\leq C_1$,  we have 
\begin{equation*}
\begin{split}
& |\partial_t A_i-\partial_{xx}A_i-c\partial_xA_i - \partial_u f(t,U_i+\epsilon \theta A_i)A_i|\vspace{3pt}\\
 \leq &\max\{b_{i-1}(t), b_i(t)\} \left( \mu_i/2+\mu_{i-1}/2+1+|c|+2\|\partial_u f\| \right).\end{split}
\end{equation*}
It then follows from  \eqref{pu-leq-rho1} and \eqref{bound-bi} that
$$\mathcal{L} \bar{W}_i \geq 2\rho_1(c-c_i)- \epsilon M \left( \mu_i/2+\mu_{i-1}/2+1 +|c|+2\|\partial_u f\| \right)\geq 0.$$
This ends the proof of the lemma. 
\end{proof}

Next we show that, in the regions where the graph of $u(t,x)$ is flat, $u(t,x)$ converges to the platforms $(p_i)_{0\leq i\leq N}$ with an exponential rate as $t\to\infty$.  

\begin{lemma}\label{exp-esti}
Let $(\bar{c}_i)_{0\leq i\leq N}$ be the constants given in \eqref{define-barci} and let $\varrho$ be any positive constant satisfying 
\begin{equation}\label{choose-varrho}
0<\varrho\leq \frac{1}{4} \min_{1\leq i\leq N} \{c_i-\bar{c}_{i-1},\, \,\bar{c}_i-c_i\}. 
\end{equation}
Then there are positive constants $\nu>0$, $t_0>0$ and $C>0$ such that
\begin{equation}\label{ab-u-pi-exp}
\left\{\baa{ll}
u(t,x)\leq p(t)+C\me^{-\nu t}, & \hbox{ for } \,\, x\in\R,\,\,t\geq t_0, \vspace{3pt}\\
u(t,x)\leq p_i(t)+C\me^{-\nu t},& \hbox{ for } \,\,x\geq (\bar{c}_i-\varrho )t,\,t\geq t_0, \, i=1,\cdots,N,\eaa\right.
\end{equation}
and 
\begin{equation*}
\left\{\baa{ll}
u(t,x)\geq p_i(t)-C\me^{-\nu t}, & \hbox{ for } \,\, x\leq (\bar{c}_i+\varrho )t,\,t\geq t_0,\, i=0,1,\cdots,N-1, \vspace{3pt}\\
u(t,x)\geq -C\me^{-\nu t},& \hbox{ for } \,\,x\in\R,\,t\geq t_0.\eaa\right.
\end{equation*}
\end{lemma}

\begin{proof}
We only prove the estimates stated in \eqref{ab-u-pi-exp}, as the proof for the others is similar. 

Let $H(t;h_0)$ be the solution of \eqref{ode-initial} with initial value $h_0=\sup_{x\in\R}u_0(x)$. Since $ h_0 \in I_+$, it is clear that $H(t;h_0) -p(t) \to 0$ as $t\to\infty$.  Moreover, by a simple comparison argument applied to \eqref{ode-initial}, one finds some $C>0$ and $\nu\in (0,\mu_0)$ ($\mu_0$ is the constant provided by \eqref{define-mui}) such that 
$$H(t;h_0)\leq p(t)+ C\me^{-\nu t}\,\,\hbox{ for }\, x\in\R,\,t>0.  $$
On the other hand, applying the comparison principle to the equation satisfied by $u(t,x)-H(t;h_0)$, we deduce
\begin{equation*}
u(t,x)\leq H(t;h_0) \,\,\hbox{ for } \, x\in\R,\,t>0.
\end{equation*}
Combining the above two inequalities, we immediately obtain that the first inequality of \eqref{ab-u-pi-exp} holds for all $t>0$, $x\in\R$. 

Let us now turn to prove the second inequality of \eqref{ab-u-pi-exp}. Let $1\leq i\leq N$ be any fixed integer and let $M_i$ be a large positive constant such that 
\begin{equation}\label{choose-Mi}
U_i(t,c_it-M_i) \geq \frac{p_{i-1}(t)+p_i(t)}{2} \,\,\hbox{ for all }\, t\in\R.
\end{equation}
Remember that the solution $u(t,x)$ satisfies statement (i) of Theorem \ref{converge3}. One finds some $k_i\in\N$ and a $C^1$ function $\xi_i(t)$ on $[k_iT,\infty)$ such that $\xi_i(t)/t \to c_i $ as $t\to\infty$ and that
 \begin{equation}\label{define-xii}
u(t,x) \leq \frac{p_{i-1}(t)+p_i(t)}{2} \,\,\hbox{ for all }\,  x\geq \xi_i(t),\, t\geq k_iT. 
\end{equation}
Since $\varrho\leq \frac{1}{4} (\bar{c}_i-c_i)$, replacing $k_i$ by some larger integer if necessary, we may assume 
\begin{equation}\label{relation-xii-cit}
\xi_i(t)\leq (\bar{c}_i-2\varrho)t-M_i \,\,\hbox{ for all }\, t\geq k_iT. 
\end{equation}
Let $\epsilon\in (0,\epsilon_0]$ be a fixed real number, where $\epsilon_0$ is the positive constant determined in the first statement of Lemma \ref{tw-super-sub} with $c=\bar{c}_i-2\varrho$ (one easily sees from the proof of Lemma \ref{tw-super-sub} that, after making some adjustment, $\epsilon_0$ can be chosen independent of $i$). We claim that there exists some large constant $K_i>0$ such that 
\begin{equation}\label{u-ki-Ui}
u(k_iT,x)\leq U_i(0,x-(\bar{c}_i-2\varrho)k_iT-K_i)+\epsilon\,\,\hbox{ for all }\, x\geq \xi_i(k_iT).
\end{equation} 
Indeed, in the case $1\leq i\leq N-1$, this claim can be easily proved by using \eqref{define-xii}, the monotonicity of $U_i(0,x)$ in $x$, and the fact that $\limsup_{x\to\infty} u(k_iT,x) < p_i(0)$. In the case $i=N$, since $\lim_{t\to\infty}\limsup_{x\to\infty} u(t,x)=0$, replacing $k_i$ by some larger integer if necessary, we may assume  $\limsup_{x\to\infty} u(k_iT,x)\leq \epsilon$. Then the same reasoning as above implies \eqref{u-ki-Ui}. 

Let us define 
$$ \bar{W}_i(t,x)= U_i(t,x+c_it-(\bar{c}_i-2\varrho)(t+k_iT)-K_i)+\epsilon A_i(t,x-(\bar{c}_i-2\varrho) t) $$
for $x\geq \xi_i(t+k_iT)$, $t\geq 0$, where  $A_i$ is the function defined in \eqref{define-Ai}.
Clearly,  \eqref{u-ki-Ui} implies 
$$u(k_iT,x) \leq  \bar{W}_i(0,x) \,\, \hbox{ for all } \,  x\geq \xi_i(k_iT).$$
It is also easily seen from Lemma \ref{tw-super-sub} that 
$$\partial_t \bar{W}_i \geq \partial_{xx}  \bar{W}_i +f(t,\bar{W}_i)\,\,\hbox{ for }\, x> \xi_i(t+k_iT),\,t>0. $$
Moreover, by \eqref{choose-Mi}, \eqref{define-xii} and the $T$-periodicity of $p_{i-1}$, $p_i$, we have
$$u(t+k_iT,\xi_i(t+k_iT))\leq   U_i(t,c_it-M_i) \, \,\hbox{ for all } t\geq 0. $$
It further follows from \eqref{relation-xii-cit} and the monotonicity of $U_i(t,x)$ in $x$ that
$$u(t+k_iT,\xi_i(t+k_iT)) \leq   U_i(t, \xi_i(t+k_iT)-(\bar{c}_i-2\varrho)(t+k_iT) +c_it)  <   \bar{W}_i(t,\xi_i(t+k_iT))$$
for all $t>0$. Then, the comparison principle implies that 
$$u(t+k_iT,x)\leq \bar{W}_i(t,x)   \hbox{ for all }  x\geq \xi_i(t+k_iT), \,t\geq 0.$$
In particular, there exists some large time $t_0\geq k_iT$ such that   
\begin{equation*}
u(t,x)\leq U_i(t,\varrho t+c_i t-K_i )+\epsilon b_i(t) \,\,\hbox{ for all } t\geq t_0,\, x\geq (\bar{c}_i-\varrho)t.  
\end{equation*}
Notice from \cite[Theorem 2.2]{abc} that $U_i(t,\varrho t+c_i t-K_i )$ approaches $p_i(t)$ as $t\to\infty$ with an exponential rate. Moreover, we know from \eqref{bound-bi} that $b_i(t)$ converges to $0$ as $t\to\infty$ exponentially. Thus, making some adjustment to the constants $C$ and $\nu$ if necessary, we obtain the second estimate of \eqref{ab-u-pi-exp}. This ends the proof of Lemma \ref{exp-esti}.
\end{proof}

Since $U_i$ is a periodic traveling wave connecting two linearly stable solutions of \eqref{ODE}, it is known from \cite{abc,con} that $U_i$ is global and exponential stable with asymptotic phase. In the following lemma, we show that this stability remains valid when there is an exponentially decaying inhomogeneity in the equation. Similar results can be found in  \cite[Lemma 6.23]{po3} and \cite[Theorem 3.1]{rtv} for autonomous equations/systems.

\begin{lemma}\label{exp-pert}
Assume that $g(t,x)$ is a continuous function on $[0,\infty)\times \R$ such that for some positive constants $K>0$ and $\gamma>0$, there holds
\begin{equation}\label{r-exp-decay}
|g(t,x)| \leq K\me^{-\gamma t} \,\,\hbox{ for all }\, x\in\R,\,t\geq 0.
\end{equation}
Let $w(t,x)$ be a solution of 
\begin{equation*}
w_t=w_{xx}+f(t,w)+g(t,x)\,\,\hbox{ for } \,x\in\R,\,t>0
\end{equation*}
satisfying 
\begin{equation}\label{assu-converge-w}
\inf_{\eta\in\R} \| w(t,\cdot)-U_i(t,\cdot-\eta) \|_{L^{\infty}(\R)} \to 0 \,\,\hbox{ as } \,t\to\infty,
\end{equation}
for some $1\leq i\leq N$. Then there exist $\nu>0$, $\bar{\eta}_i\in\R$ and $C>0$ such that 
\begin{equation*}
 \| w(t,\cdot)-U_i(t,\cdot-\bar{\eta}_i) \|_{L^{\infty}(\R)} \leq  C\me^{-\nu t} \,\,\hbox{ for all } t>0. 
\end{equation*}
\end{lemma}

To prove this lemma, we need the following local stability of $U_i$. 

\begin{lemma}\label{property-Ui}
For each $i=1,\cdots,N$, $U_i$ is local stable in the following sense: there exist $\delta^*\in (0,1)$, $\mu^*\in (0,1)$ and $k^*\in \N$ such that for any $\psi \in C(\R)$ satisfying 
\begin{equation*}
\min_{\eta\in\R} \| \psi(\cdot)-U_i(0,\cdot-\eta) \|_{L^{\infty}(\R)} \leq \delta^*,
\end{equation*}
there holds
\begin{equation*}
\min_{\eta\in\R} \| v(k^*T,\cdot;\psi)-U_i(0,\cdot-\eta) \|_{L^{\infty}(\R)} \leq \mu^*\min_{\eta\in\R} \| \psi(\cdot)-U_i(0,\cdot-\eta) \|_{L^{\infty}(\R)}, 
\end{equation*}
where $v(t,\cdot;\psi)$ denotes the solution of \eqref{E} with $u_0$ replaced by $\psi$. 
\end{lemma}

\begin{proof}
This lemma follows directly from the proof of \cite[Theorem 3.6]{abc}. 
\end{proof}

\begin{proof}[Proof of Lemma \ref{exp-pert}]
Let $\delta^*\in (0,1)$, $\mu^*\in (0,1)$ and $k^*\in \N$ be the constants provided by Lemma \ref{property-Ui}.
Making $\mu^*\in (0,1)$ larger if necessary, we may assume that  
\begin{equation}\label{assum-mu*}
\mu^* \me^{\gamma k^*T}> 1, 
\end{equation}
where $\gamma>0$ is the exponential decay rate of $g$ in \eqref{r-exp-decay}. Due to the assumption \eqref{assu-converge-w}, one finds some $j^*\in \N$ such that 
\begin{equation*}
\min_{\eta\in\R} \| w(t,\cdot)-U_i(t,\cdot-\eta) \|_{L^{\infty}(\R)} \leq \delta^* \,\,\hbox{ for all } \, t\geq j^*T.
\end{equation*}
For each $j\geq j^*$,  set
$$Z_j(t,x)=w(t,x)-v(t-jT,x;w(jT,\cdot))\,\,\hbox{ for }\, x\in\R,\,t\geq jT, $$
where $v(t,x;w(jT,\cdot))$ is the solution of \eqref{E} with $u_0(\cdot)$ replaced by $w(jT,\cdot)$. It is clear that $Z_j$ satisfies the following inhomogeneous linear parabolic equation 
\begin{equation*}
\left\{\baa{ll}
\smallskip \partial_t Z_j=\partial_{xx} Z_j +c_j(t,x)Z_j+g(t,x), & x\in\R,\,t>jT, \vspace{3pt}\\
Z_j(jT,x)= 0, & x\in\R,\eaa\right.
\end{equation*}
for some bounded function $c_j(t,x)$. One easily checks that 
$$|c_j(t,x)| \leq  C_1 \, \hbox{ for }\, x\in\R,\,t>jT,\,j\geq j^*,$$
where $C_1=\max\{ |\partial_u f(t,u)|:\,u\in [-1,\,p(t)+1],\, t\in\R\}$. 

We claim that
\begin{equation}\label{estimate-zj}
\|Z_j(t,\cdot)  \|_{L^{\infty}(\R)} \leq C_2 \me^{-\gamma jT} \,\,\hbox{ for all } jT\leq t\leq (j+k^*)T,
\end{equation}
for some positive constant $C_2$ independent of $j\geq j^*$.  Since $g(t,x)$ satisfies \eqref{r-exp-decay}, it follows from the comparison principle that 
$$ H_-(t) \leq Z_j(t,x)\leq H_+(t)  \,\,\hbox{ for }\, x\in\R,\,jT\leq t\leq (j+k^*)T,$$
where $H_{\pm}$ are the solutions of the following ODEs 
$$\frac{d H_{\pm}}{dt} = \pm C_1  H_{\pm} \pm K\me^{-\gamma t}\,\, \hbox{ for } \,  jT< t\leq (j+k^*)T; \quad  H_{\pm} (jT)=0. $$
Making some adjustment to $C_1$ if necessary, we may assume that $C_1>\gamma$. Then direct calculation yields 
$$-\frac{K}{C_1-\gamma} \me^{-\gamma k^*T} \me^{-\gamma jT} \leq Z_j(t,x)\leq \frac{K}{C_1+\gamma} \me^{C_1k^*T} \me^{-\gamma jT} $$
for all $x\in\R$, $jT\leq t\leq (j+k^*)T$. This immediately implies that \eqref{estimate-zj} holds with 
$$C_2=\max\left\{\frac{K}{C_1-\gamma} \me^{-\gamma k^*T} ,\,  \frac{K}{C_1+\gamma} \me^{C_1k^*T}  \right\}. $$

Next, we prove that  
\begin{equation}\label{esti-discret-w-U}
\min_{\eta\in\R} \| w(mk^*T,\cdot)-U_i(0,\cdot-\eta)\|_{L^{\infty}(\R)} \leq C_3 (\mu^{*})^m\,\,\hbox{ for all }\, m\in \N
\end{equation}
for some positive constant $C_3$ independent of $m$. Clearly, for each $m\geq 1$, we have
\begin{equation*}
\begin{split}
& \min_{\eta\in\R} \| w(mk^*T,\cdot)-U_i(0,\cdot-\eta)\|_{L^{\infty}(\R)}  \vspace{3pt}\\
\leq  &\,\, \|Z_{(m-1)k^*} (mk^*T,\cdot ) \|_{L^{\infty}(\R)}+\min_{\eta\in\R} \| v(k^*T,\cdot;w((m-1)k^*T,\cdot)-U_i(0,\cdot-\eta)\|_{L^{\infty}(\R)}. 
\end{split}
\end{equation*}
Let $m^*$ be the least integer such that $m^*k^*\geq j^*$.
It then follows from Lemma \ref{property-Ui} and \eqref{estimate-zj}  that for all $m > m^*$, 
\begin{equation*}
\begin{split}
& \min_{\eta\in\R} \| w(mk^*T,\cdot)-U_i(0,\cdot-\eta)\|_{L^{\infty}(\R)}  \vspace{3pt}\\
\leq  &\,\, C_2  \me^{-\gamma (m-1)k^*T} +\mu^*\min_{\eta\in\R} \| w((m-1)k^*T,\cdot)-U_i(0,\cdot-\eta)\|_{L^{\infty}(\R)}. 
\end{split}
\end{equation*}
Notice that $ \| w(m^*k^*T,\cdot)-U_i(0,\cdot-\eta)\|_{L^{\infty}(\R)} \leq \sigma^*$. Then by a simple induction argument, we deduce that for all $m> m^*$, 
$$  \min_{\eta\in\R} \| w(mk^*T,\cdot)-U_i(0,\cdot-\eta)\|_{L^{\infty}(\R)}   \leq \sum_{l=1}^{m-m^*} C_2\me^{-\gamma (m-l)k^*T}(\mu^*)^{l-1}+\sigma^*(\mu^*)^{m-m^*}.$$
By using \eqref{assum-mu*}, we obtain that for all $m> m^*$, 
$$ \min_{\eta\in\R} \| w(mk^*T,\cdot)-U_i(0,\cdot-\eta)\|_{L^{\infty}(\R)} \leq  C_2\frac{ \me^{-\gamma(m^*-1)k^*T} }{\mu^*\me^{\gamma k^*T}-1}(\mu^*)^{m-m^*}+\sigma^*(\mu^*)^{m-m^*}. $$
This implies that \eqref{esti-discret-w-U} holds with some $C_3>0$ (independent of $m$).

Finally, choosing $\nu=- \ln \mu^*/k^*T$, we see from \eqref{esti-discret-w-U} that
$$ \min_{\eta\in\R} \| w(mk^*T,\cdot)-U_i(0,\cdot-\eta)\|_{L^{\infty}(\R)} \leq C_3 \me^{-\nu mk^*T}\,\,\hbox{ for all }\, m\in\N.$$
Then, similar comparison arguments to those used in proving \eqref{estimate-zj} imply that, for each $m\in\N$, there exist positive constants $C_4$ and $C_5$ (both are independent of $m\in\N$) such that 
$$ \min_{\eta\in\R} \| w(t,\cdot)-U_i(t,\cdot-\eta)\|_{L^{\infty}(\R)} \leq C_4 \me^{-\nu mk^*T}+C_5 \me^{-\gamma mk^*T}$$
for all  $mk^*T \leq t\leq  (m+1)k^*T$.  Since $\nu<\gamma$ because of \eqref{assum-mu*}, one easily derives that 
$$ \min_{\eta\in\R} \| w(t,\cdot)-U_i(t,\cdot-\eta)\|_{L^{\infty}(\R)} \leq (C_4+C_5)\me^{\nu k^*T} \me^{-\nu t}  \,\,\hbox{ for all } \,t>0.$$
This ends the proof of Lemma \ref{exp-pert}. 
\end{proof}

Now we can complete the proof of Theorem \ref{converge3} (ii) by showing the following lemma:

\begin{lemma}\label{exp-con-inter}
Let $(\bar{c}_i)_{0\leq i\leq N}$ be the constants given in \eqref{define-barci}. There exist $C>0$, $\nu>0$ and $t_0>0$ such that
\begin{equation}\label{estimate-exp-mid}
\left | u(t,x)-U_i(t,x-\bar{\eta}_i) \right 
| \leq C\me^{-\nu t} \,\, \hbox{ for }\,  \bar{c}_{i-1}t\leq x\leq \bar{c}_it,\,t\geq t_0,\, i=1,\cdots,\,N,
\end{equation}
for some $\bar{\eta}_i\in\R$. 
\end{lemma}

\begin{proof}
Let $i=1,\cdots, N$ be any fixed integer and let $\varrho$ be a positive constant satisfying \eqref{choose-varrho}.  We first choose some large $t_0>0$ such that 
\begin{equation*}
\left\{\baa{l}
\smallskip (\bar{c}_{i-1}t-3,\, \bar{c}_{i-1}t]\subset ((\bar{c}_{i-1}-\varrho)t, \,(\bar{c}_{i-1}+\varrho)t)\vspace{3pt}\\
\,[\bar{c}_{i}t, \,\bar{c}_{i}t+3)\subset ((\bar{c}_{i}-\varrho)t, \,(\bar{c}_{i}+\varrho)t)
\eaa\right. \hbox{ for all }\, t\geq t_0.
\end{equation*}
Since $u-p_{i-1}$ and $u-p_{i}$ are solutions of linear parabolic equations,  by standard parabolic estimates, we obtain some $C_1>0$ such that 
\begin{equation*}
\left\{\baa{l}
\smallskip |u_x(t,x)| \leq  C_1|u(t,x)-p_{i-1}(t) | \,\, \hbox{ for all }  \, \bar{c}_{i-1}t-3\leq x\leq \bar{c}_{i-1}t,\,t\geq t_0,\vspace{3pt}\\
|u_x(t,x)| \leq  C_1|u(t,x)-p_{i}(t) | \,\, \hbox{ for all }  \, \bar{c}_{i}t\leq x\leq \bar{c}_{i}t+3,\,t\geq t_0.
\eaa\right. 
\end{equation*}
It then follows from Lemma \ref{exp-esti} that there exist $\nu>0$ and $C_2>0$ such that, possibly after replacing $t_0$ by some larger constant, 
\begin{equation}\label{estimate-ux}
|u_x(t,x)|  \leq C_2 \me ^{-\nu t} \,\,\hbox{ for all }\, x\in [ \bar{c}_{i-1}t-3,\, \bar{c}_{i-1}t]\cup [ \bar{c}_{i}t,\, \bar{c}_{i}t+3],\,t\geq t_0. 
\end{equation}

Next, we define a function $w(t,x)$ on $[t_0,\infty)\times\R$ by 
\begin{equation*}
w(t,x)=\left\{\baa{ll}
\smallskip \zeta(x- (\bar{c}_{i-1}t-3)) p_{i-1}(t)+(1-  \zeta(x- (\bar{c}_{i-1}t-3)))u(t,x) & \hbox{ for }  \, x\leq c_it,\vspace{3pt}\\
\zeta(x- \bar{c}_{i}t) u(t,x)+(1-  \zeta(x- \bar{c}_{i}t))p_{i}(t) & \hbox{ for }  \, x\geq c_it,
\eaa\right. 
\end{equation*}
where $\zeta(x)$ is a $C^2(\R)$ function satisfying \eqref{smooth-zeta}. It is easily seen that $w\in C^{1,2}([t_0,\infty)\times\R)$, and that 
\begin{equation}\label{chara-w-u}
w(t,x)=\left\{\baa{ll}
\smallskip p_{i-1}(t) & \hbox{ if }  \, x\leq \bar{c}_{i-1}t-3, \,t\geq t_0,\vspace{3pt}\\
\smallskip u(t,x) & \hbox{ if }  \, \bar{c}_{i-1}t \leq x\leq \bar{c}_{i}t, \,t\geq t_0,\vspace{3pt}\\
 p_{i}(t) & \hbox{ if }  \, x\geq \bar{c}_{i}t+3, \,t\geq t_0,\vspace{3pt}\\
\eaa\right. 
\end{equation}
Set 
$$g(t,x)= w_t-w_{xx}-f(t,w) \,\,\hbox{ for } \, x\in\R,\,t\geq t_0. $$
Clearly, $g(t,x)$ is continuous on $[t_0,\infty)\times\R$ and
$$g(t,x)=0\,\hbox{ for }\,  x\in (-\infty,\, \bar{c}_{i-1}t-3]\cup [\bar{c}_{i-1}t, \, \bar{c}_{i}t] \cup [ \bar{c}_{i}t+3,\, \infty),\,t\geq t_0. $$
We claim that there exists some constant $C_3>0$ such that
\begin{equation}\label{estimate-g-mid}
|g(t,x)| \leq C_3\me^{-\nu t} 
\end{equation}
for all $x\in [ \bar{c}_{i-1}t-3,\, \bar{c}_{i-1}t]\cup [ \bar{c}_{i}t,\, \bar{c}_{i}t+3]$, $t\geq t_0$. 
Indeed, when $x\in [ \bar{c}_{i-1}t-3,\, \bar{c}_{i-1}t]$, $t\geq t_0$, it is straightforward to calculate that 
\begin{equation*}
\begin{split}
g(t,x) = &\,\, (\zeta''+\bar{c}_{i-1}\zeta')(u-p_{i-1})+(1-\zeta)(f(t,u)-f(t,p_{i-1}))   \vspace{3pt}\\
 &\,\,+ (f(t,p_{i-1})-f(t,w))+2\zeta'u_x(t,x),
\end{split}
\end{equation*}
where $\zeta$, $\zeta'$ and $\zeta''$ stand for $\zeta(x- (\bar{c}_{i-1}t-3))$, $\zeta'(x- (\bar{c}_{i-1}t-3))$ and $\zeta''(x- (\bar{c}_{i-1}t-3))$, respectively. Due to the $C^1$-regularity and the $T$-periodicity of $f$, it then follows from Lemma \ref{exp-esti} and \eqref{estimate-ux} that \eqref{estimate-g-mid} holds for $x\in [ \bar{c}_{i-1}t-3,\, \bar{c}_{i-1}t]$,  $t\geq t_0$. The proof for $x\in [ \bar{c}_{i}t,\, \bar{c}_{i}t+3]$,  $t\geq t_0$ is analogous, therefore we omit the details.    

Finally, let $\eta_i(t)$ be the $C^1([0,\infty))$ function provided by Theorem \ref{converge3} (i). By our choice of $(\bar{c}_j)_{0\leq j\leq N}$, it is easily checked that
$$\|u(t,\cdot)- U_i(t,\cdot-\eta_i(t)) \| _{L^{\infty}([\bar{c}_{i-1}t,\,\bar{c}_it])} \to 0 \,\hbox{ as }\, t\to\infty. $$
Then, by using Lemma \ref{exp-esti}, \eqref{chara-w-u} and the asymptotics of $U_i(t,c_it+x)$ as $x\to\pm\infty$,  we deduce that
$$\|w(t,\cdot)- U_i(t,\cdot-\eta_i(t)) \| _{L^{\infty}(\R)} \to 0 \,\hbox{ as }\, t\to\infty. $$
Therefore, all the conditions of Lemma \ref{exp-pert} are fulfilled. Consequently, there exist $\nu>0$, $\bar{\eta}_i\in\R$ and $C>0$ such that 
\begin{equation*}
 \| w(t,\cdot)-U_i(t,\cdot-\bar{\eta}_i) \|_{L^{\infty}(\R)} \leq  C\me^{-\nu t} \,\,\hbox{ for all } t>0. 
\end{equation*}
This together with \eqref{chara-w-u} immediately gives \eqref{estimate-exp-mid}. The proof of Lemma \ref{exp-con-inter} is thus complete.
\end{proof}

It is clear that Theorem \ref{converge3} (ii) follows directly from Lemmas \ref{exp-esti} and \ref{exp-con-inter}.



\end{document}